\documentclass[10pt,a4paper,leqno]{amsart}
\usepackage{amssymb}
\usepackage{array}
\usepackage{graphicx}
\usepackage{pdfsync}
\usepackage{graphicx}
\usepackage{pdfsync}
\usepackage{bm}
\usepackage{xspace}
\usepackage{fancyhdr}
\usepackage[english]{babel}
\usepackage{amsfonts}
\usepackage[ruled,vlined, linesnumbered, norelsize]{algorithm2e}
\usepackage{latexsym}
\usepackage{graphics}
\usepackage{hyperref}
\usepackage{pifont,bbding}
\usepackage{color}

\usepackage{enumerate,amsthm,amsmath,comment,psfrag}
\usepackage[font=small]{caption}

\newtheorem{remark}{Remark}[section]
\newtheorem{lemma}{Lemma}[section]
\newtheorem{proposition}{Proposition}[section]
\newtheorem{theorem}{Theorem}[section]

\newtheorem{definition}{Definition}[section]
\numberwithin{equation}{section}
\newcommand{\ep}{\varepsilon}
\newcommand{\us}{\upsilon}

\newcommand{\Om}{\varOmega}
\newcommand{\D}{\varDelta}
\newcommand{\E}{\mathcal{E}}
\newcommand{\p}{\partial}
\newcommand{\V}{\mathbb{V}}
\newcommand{\h}{\eta_{\V^n}}
\newcommand{\hh}{\eta_{\V^{n-1}}}
\newcommand{\het}{\eta_{\widehat\V^n}}
\newcommand{\R}{\mathcal{R}}
\newcommand{\Rr}{R}
\newcommand{\I}{\mathcal{I}}
\newcommand{\bg}{\bar{g}_n}
\newcommand{\II}{\mathrm{I}}
\newcommand{\T}{\mathcal{T}}
\newcommand{\iu}{\mathrm{i}}
\newcommand{\Pro}{\mathcal{P}}
\newcommand{\dtd}{\bar\p W^{n-\frac 1 2}}
\newcommand{\TT}{\mathrm{T}}
\newcommand{\Ss}{\mathrm{S}}
\newcommand{\Gc}{\mathcal{G}}
\newcommand{\Gn}{\mathcal{G}_{n}}
\newcommand{\Gnm}{\mathcal{G}_{n-1}}
\DeclareMathOperator\Rea{\mathrm{Re}}
\DeclareMathOperator\Ima{\mathrm{Im}}
\begin{document}
%
%
%-------------------------------------------------------------------------------
% Top matter
%-------------------------------------------------------------------------------
%
%
\title[A posteriori error control \& adaptivity for LS]
{A posteriori error control \& adaptivity\\ for Crank-Nicolson finite element approximations\\ for the linear Schr\"odinger equation }
%{Linear Schr\"{o}dinger equations: A posteriori error control for fully discrete C-N approximations}
%
\date{\today}
\author{Theodoros Katsaounis}
\address[Theodoros Katsaounis]{Department of Applied Mathematics\\
University of Crete\\
71409 Heraklion-Crete, Greece  \\
\textnormal{and} Institute of Applied and Computational
Mathematics, \\ FORTH, Vassilika Vouton\\ GR 700$\,$13 Heraklion-Crete, Greece
}
\email{thodoros@tem.uoc.gr}
\author{Irene Kyza}
\address[Irene Kyza]{Institute of Applied and Computational Mathematics\\ Foundation of Research \& Technology-Hellas\\
Nikolaou Plastira 100, Vassilika Vouton\\
GR 700$\,$13 Heraklion-Crete, Greece} \email{kyza@iacm.forth.gr}

\thanks{The research of I.K. was supported  by the  European Social Fund (ESF) -European Union (EU) and National Resources of the Greek State within the framework of the Action ``Supporting Postdoctoral Researchers''  of the Operational Programme ``Education and Lifelong Learning (EdLL)''.
Th. Katsaounis was partially supported by European Union FP7 program Capacities (Regpot 2009-1), through ACMAC (http://www.acmac.uoc.gr).} %%
\begin{abstract}
We derive optimal order a posteriori error estimates for fully discrete approximations of linear Schr\"odinger-type equations, in the $L^\infty(L^2)-$norm. For the discretization in time we use the Crank-Nicolson method, while for the space discretization we use finite element spaces that are allowed to change in time. The derivation of the estimators is based on a novel elliptic reconstruction that leads to estimates which reflect the physical properties of Schr\"odinger equations. The final estimates are obtained using energy techniques and residual-type estimators. Various numerical experiments for the one-dimensional linear Schr\"odinger equation in the semiclassical regime, verify and complement our theoretical results. The numerical implementations are performed with both uniform partitions and adaptivity in time and space. For adaptivity, we further develop and analyze  an existing time-space adaptive algorithm to the cases of Schr\"odinger equations. The adaptive algorithm reduces the computational cost  substantially and provides efficient error control for the solution and the observables of the problem, especially for small values of the Planck constant.
\end{abstract}
\maketitle
%
%
%-------------------------------------------------------------------------------
\section{Introduction}
%-------------------------------------------------------------------------------

In this paper  we focus on the  a posteriori error control and adaptivity for fully discrete Crank-Nicolson finite element (CNFE) schemes for the general form of linear Schr\"odinger equation: 
%we complete the analysis on optimal order a posteriori error control for linear Schr\"odinger equations, started in \cite{Kyza}, by considering  Crank-Nicolson fully discrete approximations.
%
%To this end, we consider the initial- and boundary-value problem for the linear Schr\"odinger equatio\varPin: 
%
\begin{equation}
\label{LS1} \left \{
\begin{aligned}
&\p_tu-\mathrm{i}\alpha\varDelta u+\mathrm{i}g(x,t)u = f(x,t)    &&\quad\mbox{in $\varOmega\times (0,T]$,}& \\
&u=0  &&\quad\mbox{on $\partial \varOmega\times (0,T]$,}&  \\
&u(\cdot,0)=u_0 &&\quad\mbox{in $\varOmega$},&
\end{aligned}
\right.
\end{equation}
where $\Om$ is a convex ``polygonal'' domain in $\mathbb{R}^d$, $1\le d\le3$, with boundary $\p\Om$, and $0<T<\infty.$ In \eqref{LS1}, $\alpha$ is a positive constant, $g:\varOmega\times(0,T]\to\mathbb{R}$ and $f:\varOmega\times(0,T]\to\mathbb{C}$ are given functions and  $u_0:\varOmega\to\mathbb{C}$ is a given initial value. 
%

%Nevertheless, adaptivity may be proven beneficial for the numerical solution of \eqref{LS1}. In fact, this is something we expect in the case of linear Schr\"odinger equation in the semiclassical regime:
A special case of \eqref{LS1} is the so-called linear Schr\"odinger equation in the semiclassical regime:
\begin{equation}
\label{semiclassical}
\p_t u-\mathrm{i}\frac{\ep}{2}\D u+\frac{\mathrm{i}}{\ep}V(x,t)u=0,
\end{equation}
with high frequency initial data. It is clear that \eqref{semiclassical} %is a special case of 
can be obtained from \eqref{LS1} %with 
by setting $\alpha:=\frac \ep 2$, $g:=\frac 1\ep V$ and $f\equiv 0$. In \eqref{semiclassical}, $\ep$ ($0<\ep \ll 1$) is the scaled Planck constant, $V$ is %a smooth, 
an $L^\infty(L^\infty)$ time-dependent potential and $u$ is the wave function. The wave function $u$ is  used to define primary physical quantities, called \emph{observables} (\cite{BJM, GMMP}), such as the \emph{position density},
\begin{equation}
\label{posden}
N(x,t) :=|u(x,t)|^2,
\end{equation}
and the \emph{current density},
\begin{equation}
\label{curden}
J(x,t):= \Ima\big(\overline {u(x,t)}\nabla u(x,t)\big).
\end{equation}
Problems related to \eqref{semiclassical} are of great interest in physics and engineering. However, the solution of \eqref{semiclassical} is complicated from  the theoretical as well as the numerical analysis point of view. It is well known that for $\ep$ small (close to zero), the solution of \eqref{semiclassical} oscillates with wavelength $\mathcal{O}(\ep)$, preventing $u$ to converge strongly as $\ep\to 0$. Because of this, standard numerical methods fail to correctly approximate  $u$ and the observables, unless very fine mesh sizes and time steps are used.  
In particular, previous works (cf., e.g., \cite{BJM, MPP, MPPS}) suggested that for standard finite element (FE) methods  there is a very restrictive dispersive relation connecting the mesh sizes (space and time) with parameter $\ep$; cf., e.g., \eqref{exact} below. This restrictive dispersive relation can be relaxed using the so-called time-splitting spectral methods, introduced earlier by Bao, Jin \& Markowich in \cite{BJM}, for the approximation of the solution of \eqref{semiclassical}.
%In this work, we aim to investigate whether  the restrictive conditions on the mesh sizes can be relaxed for small values of $\ep$ using adaptivity.
%
%Time-splitting spectral methods for \eqref{semiclassical} were introduced by Bao, Jin \& Markowich in \cite{BJM}.

In this paper,  our goal is to show that constructing adaptive algorithms based on rigorous a posteriori error control leads to CNFE schemes which are competitive to the best available methods for the approximation of the solution (and the observables) of the semiclassical Schr\"odinger equation \eqref{semiclassical}, and in general of linear Schr\"odinger equations of the form \eqref{LS1}.  It also permits, for the first time, realistic computations for rough potentials for the linear Schr\"odinger equation in the semiclassical regime.To achieve our goal, in the current work we:
\begin{enumerate}[$\bullet$]\itemsep=2pt
\item Provide rigorous a posteriori error analysis for \eqref{LS1} for CNFE approximations using FE spaces that are allowed to change in time;
\item Study the advantages of adaptivity through the obtained estimators for the efficient error control of \eqref{LS1}.
\end{enumerate}

Optimal order a posteriori error estimates for the heat equation for CNFE schemes with FE spaces that are allowed to change in time have been derived very recently by B\"ansch, Karakatsani \& Makridakis in \cite{BKM}.  However the extension of those ideas from the simple heat equation to the linear Schr\"odinger equation \eqref{LS1} is of increased difficulty due to the complex-value and multiscale nature of the problem.  Because of this, novel ideas and techniques are introduced. More precisely, our main contributions are:  
%
%Despite the fact that in some places we follow the methodology of  \cite{BKM}, our analysis includes novel results and ideas. Our main contributions are:
%
\begin{enumerate}[$\bullet$]\itemsep=2pt
\item Derivation of optimal order a posteriori error bounds  in the $L^\infty(L^2)-$norm for CNFE schemes for \eqref{LS1}. The fact that the analysis includes time-dependent potentials, makes the problem more challenging since there are no rigorous results for Schr\"odinger equations for such potentials. In addition the  existing  literature on a posteriori error analysis for problems with time-dependent operators of the form $\mathcal{A}(t):=-\alpha\D+g(x,t)$ is quite limited. To the best of our knowledge, only in \cite{CGM} the authors  consider similar operators. Moreover  the derived estimates hold for $L^\infty(L^\infty)-$type potentials as well, in contrast to the existing literature. In particular, existing results  require smooth $C^1(C^2)-$type potentials. However, this regularity requirement  on the potential is rather restrictive from applications' point of view. Including $L^\infty(L^\infty)$ time-dependent potentials in the analysis is important for another reason: It can be considered as the first step for the a posteriori error control of nonlinear Schr\"odinger (NLS) equations. More precisely, the relaxation  scheme introduced by Besse in \cite{Besse} suggests that a posteriori error bounds for linear Schr\"odinger equations with $L^\infty(L^\infty)-$type time-dependent potentials is essential for the efficient approximation of the solution of certain NLS equations.
\item Introduction of a \emph{novel elliptic reconstruction} leading to upper bounds that do not involve  the global $L^\infty(L^\infty)-$norm of $g$, and thus, to bounds that do reflect the physical properties of the problem.The elliptic reconstruction was developed by Makridakis \& Nochetto in \cite{MN1} to derive optimal order $L^\infty(L^2)$ a posteriori error bounds for FE spatial discrete schemes for the heat equation using energy techniques. A straightforward generalization of this notion of the elliptic reconstruction to Schr\"odinger equations leads to estimates that involve the $L^\infty(L^\infty)-$norm of the potential. Consequently, the obtained estimates  are practically useless and adaptivity is inefficient, even  in the simplest case of constant potentials.  Therefore, proposing a modified elliptic reconstruction based on the physical properties of the problem under consideration is crucial for the efficient error control of \eqref{LS1} (and so \eqref{semiclassical}). Additionally, the new ideas developed for this purpose might be useful for other problems as well, such as convection-diffusion or reaction-diffusion problems.
%The idea of the modified elliptic reconstruction might be useful  for other problems as well, such as convection-diffusion or reaction-diffusion problems.
\item  A detailed numerical study on the reliability and robustness of the a posteriori estimators through a \emph{time-space adaptive algorithm}.  Our starting point is the adaptive algorithm proposed in \cite{NSV}, adapted to the linear Schr\"odinger equation, \eqref{semiclassical}.  The  a posteriori estimators derived in this work are on the solution $u$ of \eqref{LS1}. However, in many applications  observables like the position density \eqref{posden},  or the current density \eqref{curden} are far more important than the solution itself.  Thus, we introduce an  appropriate  modification of  the a posteriori estimators and the adaptive algorithm. This modification is based on a heuristic idea and the results concerning the observables are very encouraging. Overall, the adaptive algorithm  reduces the computational cost substantially and provides efficient error control of   $u$ and the observables for small values of the Planck constant $\ep$.   %Such results are impossible to be obtained via standard techniques and without adaptivity. 
It is very difficult to obtain such results via standard techniques and without adaptivity. We point out that our purpose is not to prove convergence and optimality of the considered time-space adaptive algorithm, but rather to show that adaptivity based on rigorous a posteriori error control can be proven beneficial for the approximation of the solution (and the observables) of the linear semiclassical Schr\"odinger equation \eqref{semiclassical}. In addition, it is to be emphasized that as long as the adaptive algorithm converges, we can guarantee rigorously, based on the a posteriori error analysis, that \emph{total error remains below a given tolerance}. 
%Thoroughly study and further development of a \emph{time-space adaptive algorithm} for \eqref{semiclassical}. Adaptivity leads to impressive results, especially in terms of computational cost, and the efficient error control of $u$ and the observables for small values of the Planck constant $\ep$. Such results are impossible to be obtained via standard techniques and without adaptivity.   
\end{enumerate}

For parabolic problems,  a number of adaptive algorithms exists in the literature; cf., e.g., \cite{CF,ALBERT} and the references therein. However, convergence and optimality of time-space adaptive algorithms are very delicate and difficult issues.  In the literature exists only one proven convergent time-space adaptive algorithm for evolution problems and can be found in \cite{KMSS}. This algorithm is appropriate for the heat equation and backward Euler FE schemes and it is not clear how to generalize it to other problems and higher order in time methods.

%
%\vspace{-1.15cm}

Despite the fact that problem \eqref{LS1} (and thus \eqref{semiclassical}) is linear, a posteriori error bounds and adaptive algorithms for linear Schr\"odinger equations are very limited in the literature.  In particular, a posteriori error estimates in the $L^\infty(L^2)-$norm for fully discrete CNFE schemes have been proven earlier by D\"orfler in \cite{D}; these estimates are  first order accurate in time, thus not optimal. Using these estimates, D\"orfler also proposes an adaptive algorithm in \cite{D}.  In \cite{Kyza} (see also \cite{PhDKyza}), we considered only time-discrete approximations and we managed to prove optimal order a posteriori error estimates for \eqref{LS1} in the $L^\infty(L^2)$ and  $L^\infty(H^1)-$norms. This was achieved using the Crank-Nicolson reconstruction proposed by Akrivis, Makridakis  \& Nochetto in \cite{AMN1}. Similar estimates for \eqref{LS1}, using an alternative reconstruction, proposed by Lozinski, Picasso \& Prachittham in \cite{LPP}, can be found in \cite{PhDKyza}. To the best of our knowledge, optimal order a posteriori error estimates for fully discrete CNFE schemes do not exist  in the literature. Some preliminary results to that direction can be found in \cite{PhDKyza}. However, the a posteriori estimators derived in \cite{PhDKyza} are scaled by the global $L^\infty(L^\infty)-$norm of $g$. Hence, as already mentioned, the derived estimators do not reflect the physical properties of the problem, which makes adaptivity through these estimates not reliable. 

A posteriori error estimates in the $L^\infty(L^2)-$norm have been proven earlier in \cite{KMP} for uniform partitions and  the time-splitting spectral methods for the linear Schr\"odinger equation in the semiclassical regime \eqref{semiclassical}. In \cite{KMP}, only the one-dimensional case in space is studied and the analysis, as in  \cite{BJM}, permits only time-independent potentials,  without being obvious how the theory can be extended to time-dependent potentials.  In addition, the time-spectral methods require smooth potentials; the particular analysis is not applicable for $L^\infty(L^\infty)-$type potentials.

The analysis of the current paper is based on the introduction of appropriate space-time reconstructions.  Such reconstructions for CNFE methods and FE spaces that are allowed to change in time were introduced, for the first time, very recently, by B\"ansch, Karakatsani \& Makridakis in \cite{BKM}, for the proof of optimal order a posteriori estimates in the $L^\infty(L^2)-$norm for the heat equation.  To define those time-space reconstructions, the authors combined the idea of the elliptic reconstruction in \cite{MN1} with the Crank-Nicolson reconstruction of \cite{AMN1,LPP}. The notion of the elliptic reconstruction has also been used earlier in \cite{LM} and \cite{GLM} for the derivation of optimal order a posteriori error estimates for backward Euler FE schemes for the heat and the wave equation, respectively. The reconstruction technique is a useful tool for  deriving optimal order a posteriori error bounds; usually, this is not feasible via a direct comparison of the exact and the numerical solution; cf., e.g., \cite{D, Verfurth}. In our context, time-space reconstructions can be defined through the novel elliptic reconstruction we introduce and the Crank-Nicolson reconstruction of \cite{AMN1}. 

More precisely, the paper is organized as follows. In Section~\ref{prelim}, we introduce  notation, the variational formulation of problem and the fully discrete scheme. We propose the novel elliptic reconstruction and discuss its properties.  %Following ideas from \cite{BKM}, 
With the aim of this new elliptic reconstruction, we then define appropriate time-space reconstructions. The main theoretical results  are stated in Section~\ref{apost}, where the a posteriori analysis is developed and optimal order error bounds are derived using energy techniques, residual-type error estimators and the properties of the reconstructions. The two last sections are devoted to the numerical investigation of the efficiency of the estimators. In particular, in Section~\ref{unif}, we validate numerically the optimal order of convergence of the estimators using uniform partitions. For the linear Schr\"odinger  equation in the semiclassical regime, we verify numerically that the estimators  have the expected behavior with respect to the scaled parameter $\ep$. Finally, in Section~\ref{adapt}, we appropriately modify and apply to the one-dimensional semiclassical Schr\"odinger equation  a time-space adaptive algorithm described in \cite{CF,ALBERT} (see also \cite{NSV}). We further develop the algorithm and we make it applicable for the approximation not only of the exact solution $u$ but also for the observables, and we discuss in detail the benefits of adaptivity for equations of the form \eqref{semiclassical}. 
%
%
%--------------------------------------------------------------------------------
\section{Preliminaries}\label{prelim}
%--------------------------------------------------------------------------------
%--------------------------------------------------------------------------------
\subsection{The continuous problem}
%---------------------------------------------------------------------------------
Problem \eqref{LS1} can be rewritten equivalently  in variational form as 
\begin{equation}
\label{LS2} \left \{
\begin{aligned}
&\langle \p_tu(t),\us\rangle+\mathrm{i}\alpha\langle\nabla
u(t),\nabla\us\rangle+\mathrm{i}\big\langle g(t)u(t),\us\big\rangle=\big\langle f(t),\us\big\rangle,
&&\forall \us\in H_0^1(\varOmega),\, t\in[0,T],& \\
&u(\cdot,0)=u_0 &&\mbox{in $\overline{\varOmega}$},&
\end{aligned}
\right.
\end{equation}
where $\langle\cdot,\cdot\rangle$ denotes the $L^2-$inner product, or the $H^{-1}-H_0^1$ duality pairing, depending on the context.  We also denote by $\|\cdot\|$ the norm in $L^2(\Om).$  It is well known that, if $g\in C^1\big([0,T];C^1(\overline\varOmega)\big),$ $f\in L^2\big([0,T];L^2(\varOmega)\big),\,f_t\in L^2\big([0,T];H^{-1}(\varOmega)\big),$ and $u_0\in H_0^1(\Omega),$
%
%
%\begin{equation*} 
%\left \{
%\begin{aligned}
%&g\in C^1\big([0,T];C^1(\bar\varOmega)\big), \\
%&f\in L^2\big([0,T];L^2(\varOmega)\big),\,f_t\in
%L^2\big([0,T];H^{-1}(\varOmega)\big),
% \\
%&u_0\in H_0^1(\Omega),
%\end{aligned}
%\right.
%\end{equation%}
%
then problem \eqref{LS2} admits a unique weak solution  $u\in C\big([0,T];H_0^1(\varOmega)\big)$ with  $u_t\in C\big([0,T];$
$H^{-1}(\varOmega)\big)$; cf., e.g., \cite[pages 620--630]{Pozzi, DLL}. We thus assume that the data of \eqref{LS1} have the necessary regularity to guarantee the existence of a unique weak solution of \eqref{LS2}. We emphasize  that the a posteriori error estimates derived in the sequel, remain valid for $g\in L^\infty\Big(\Om\times(0,T)\Big)$ as well, provided that \eqref{LS1} is well-posed. In other words, in contrast to the existing analyses, ours includes rough potentials as well, under the knowledge of the well-posedness of \eqref{LS1}. To avoid making the forthcoming analysis more technical,
we  further assume that $g$ satisfies
\begin{equation}
\label{addcond}
\sup_{x\in\Om} g(x,t)\ge -\inf_{x\in\Om} g(x,t),\quad \forall t\in[0,T].
\end{equation}
Condition \eqref{addcond} is not restrictive from applications' point of view, as, in most applications, $g$ denotes a nonnegative potential and thus \eqref{addcond} is automatically satisfied. 
%
%---------------------------------------------------------------------------------
\subsection{The method}
%--------------------------------------------------------------------------------
%
We consider a partition $0=:t_0<t_1<\cdots<t_N:=T$ of $[0,T]$, and let $I_n:=(t_{n-1},t_n]$ and $k_n:=t_n-t_{n-1},\, 1\le n\le N,$ denote the subintervals of $[0,T]$ and the time steps, respectively.  Let also $ k:=\max_{1\le n\le N}k_n$. We discretize \eqref{LS1} by a Galerkin finite element method. To this end, we introduce a family $\{\T_n\}_{n=0}^N$ of conforming shape-regular triangulations of $\Om$. We further assume that each triangulation $\T_n,\, 1\le n\le N,$ is a refinement of a macro-triangulation of $\Om$ and that $\T_{n-1}$ and $\T_n$ are compatible. Two triangulations are said to be compatible if they are derived from the same macro-triangulation by an admissible refinement procedure. For precise definitions of these properties of the family $\{\T_n\}_{n=0}^N$, we refer to \cite{LM, DLM}. Note that the triangulations are allowed to change arbitrarily from one  step to another, provided they satisfy the aforementioned compatibility conditions. These conditions are minimal and allow for heavily graded meshes and adaptivity.
%
%\begin{remark}
%{\upshape
Additionally, the forthcoming analysis is applicable without any quasiuniformity type assumptions on the mesh and without any restrictions on the sizes of neighboring elements of the triangulation.
%}
%
%\end{remark}

For an element $K\in\T_n$, we denote its boundary by $\p K$. Let $h_K$ be the diameter of $K\in\T_n$ and $h:=\max_{0\le n\le N}\max_{K\in\T_n}h_K$. Let also $\varSigma_n(K)$ be the set of internal sides of $K\in\T_n$ (points in $d=1$, edges in $d=2$ and faces in $d=3$) and define $\varSigma_n:=\bigcup_{K\in\T_n}\varSigma_n(K).$ To any side $e\in\varSigma_n$, we associate a unit vector $\bm{n}_e$ on $e$ and for $x\in e$ and a function $\us$, we define 
$$J[\nabla \us](x):=\lim_{\delta\to 0}\Big[\nabla \us(x+\delta\bm{n}_e)-\nabla \us(x-\delta\bm{n}_e)\Big]\cdot\bm{n}_e.$$

To each triangulation $\T_n,$ we associate the finite element space $\V^n$,
$$\V^n:=\{\varPhi_n\in H_0^1(\Om):\forall K\in\T_n,\, \varPhi_n|_K\in\mathbb{P}^r\},$$
where $\mathbb{P}^r$ denotes the space of polynomials in $d$ variables of degree at most $r$.

With $\widehat{\T}_n:=\T_n\wedge\T_{n-1}$ we denote the finest common coarsening triangulation of $\T_n$ and $\T_{n-1}$ and by $\widehat{\V}^n:=\V^n\bigcap\V^{n-1}$ its corresponding finite element space. Finally, let  $\check{\varSigma}_n:=\varSigma_n\bigcup\varSigma_{n-1},$ and for
$K\in\widehat{\T}_n,$ let $\check{\varSigma}_K^n:=\check{\varSigma}_n\bigcap
K,$ where the element $K\in\widehat{\T}_n$ is taken to be closed.

\begin{definition}[discrete Laplacian]
For $0\le n\le N,$ the discrete version $-\D^n:\V^n\to\V^n$ of the Laplace operator $-\D$ onto $\V^n$ is defined as
\begin{equation}
\label{discrLap}
\langle-\D^n\us,\varPhi_n\rangle=\langle\nabla\us,\nabla\varPhi_n\rangle,\quad\forall\varPhi_n\in\V^n.
\end{equation}
\end{definition}

We now discretize problem \eqref{LS1} by a \emph{modified Crank-Nicolson-Galerkin} scheme, introduced earlier for the heat equation in \cite{BKM}.  Given an approximation $U^{n-1}\in\V^{n-1}$ to the exact solution at $t^{n-1}$ we define 
approximation $U^n\in\V^n$ to the exact solution $u$ at the nodes $t_n,\, 0\le n\le N,$ by the numerical method:
\begin{equation}
\label{CNGS}
\begin{aligned}
\frac{U^n-\varPi^n U^{n-1}}{k_n}-\iu\alpha\frac{\varPi^n\D^{n-1}U^{n-1}+\D^nU^n}{2}&+\iu\Pro^n\Big(g(t_{n-\frac{1}{2}})U^{n-\frac{1}{2}}\Big)
%\\&
=\Pro^nf(t_{n-\frac{1}{2}}),
\end{aligned}
\end{equation}
for $1\le n\le N$, with $U^0:=\Pro^0 u_0$ in $\Om$. In \eqref{CNGS}, $t_{n-\frac 12}:=\frac{t_{n-1}+t_n}{2}$, $U^{n-\frac 12}:=\frac{U^{n-1}+U^n}{2}$, and $\Pro^n: L^2(\Om)\to\V^n$, $\varPi^n:\V^{n-1}\to\V^n$ are appropriate projections or interpolants. In Sections~\ref{unif}, \ref{adapt}, where we discuss the numerical experiments, $\Pro^n$ and $\varPi^n$ are taken to be the $L^2-$projection. However, the theory is still valid for other choices of $\Pro^n$ and $\varPi^n$ (cf.\ \cite{BKM,BKM2}), and therefore we  consider the method in this  general setting. Another non-standard term appearing in \eqref{CNGS} is $\varPi^n\D^{n-1}U^{n-1}$ instead of $\D^nU^{n-1}$. As it was observed in \cite{BKM, BKM2}, considering $\D^nU^{n-1}$ instead of $\varPi^n\D^{n-1}U^{n-1}$ may lead to oscillatory behavior of the obtained a posteriori estimators. For this reason, we  consider  the modified scheme \eqref{CNGS} instead of the standard one.
%
%--------------------------------------------------------------------------------
\subsection{Novel elliptic reconstruction--Residual-type estimators}
%--------------------------------------------------------------------------------
The elliptic reconstruction was originally  introduced by Makridakis \& Nochetto in \cite{MN1} for the proof of optimal order a posteriori error estimates in space  in the $L^\infty(L^2)-$norm for evolution problems, using energy techniques. It was also one of the main tools in the a posteriori error analysis of the heat equation for %and wave equations for backward Euler and
 Crank-Nicolson fully discrete schemes; cf.\ \cite{BKM}.  For the linear Schr\"odinger equation \eqref{LS1}, we introduce a new type of elliptic reconstruction which reflects the physical properties of the problem, and in particular the physical properties of the semilcassical Schr\"odinger equation \eqref{semiclassical}.
 %We introduce now a modification of the elliptic reconstruction that will play  a significant role in our  analysis.
To this end, we introduce, in each $I_n$, the constant
\begin{equation}
\label{average}
\bg:=\frac 1 2\big[\sup_{x\in\Om} g(x,t_{n-\frac 12})+\inf_{x\in\Om} g(x, t_{n-\frac 1 2})\big].
\end{equation}
The main reason for the choice of \eqref{average} is that the knowledge on ``how far from $\bg$ is $g$ in $\Om$'' gives qualitative information on the behavior of the exact solution, especially in the case of linear Schr\"odinger  equation in the semiclassical regime. In order for the elliptic reconstruction we introduce below to be well defined, we need $\bg\ge 0$,  which is automatically satisfied due to \eqref{addcond}.
\begin{definition}[novel elliptic reconstruction]\label{mellipticrec}
For fixed $V_n\in\V^n$ we define the elliptic reconstruction $\R^nV_n\in H_0^1(\Om)$ of $V_n$ to be the weak solution of the elliptic problem
\begin{equation}
\label{ellipticrec}
\alpha \langle\nabla\R^n V_n,\nabla\phi\rangle+\bg\langle\R^n V_n,\phi\rangle=\big\langle(-\alpha\D^n+\bg)V_n,\phi\big\rangle,\quad\forall\phi\in H_0^1(\Om).
\end{equation}
\end{definition}

As we shall see in the sequel, the above modified elliptic reconstruction will allow us to obtain qualitatively  better a posteriori error estimators compared to those obtained using the standard elliptic reconstruction; cf., \cite{PhDKyza}. In fact, the $\sup_{x\in\Om}|g(x,t)|$ that appears in the standard results of a priori error analysis, can now be replaced, due to \eqref{ellipticrec}, by $\sup_{x\in\Om}|g(x,t)-\bg|,\, t\in I_n,$  leading to better constants. %when $g$ doesn't change much with respect to the spatial variable. 
A very interesting question here, that needs further investigation, is whether the global constant $\sup_{x\in\Om}|g(x,t)-\bg|$ can be localized in each element. This will not only lead to better constants in the final a posteriori error estimators, but also might give the inspiration of proposing appropriate adaptive strategies.

Using \eqref{discrLap}, we see that $\R^n$ satisfies the orthogonality property
\begin{equation}
\label{orthogonal}
\alpha\big\langle\nabla(\R^n-\II)V_n,\nabla\varPhi_n\big\rangle+\bg\big\langle(\R^n-\II)V_n,\varPhi_n\big\rangle=0,\quad\forall\varPhi_n\in\V_n.
\end{equation}

Let now $z$ be the weak solution of the following elliptic problem
\begin{equation}
\label{dualprob1}
\langle\nabla z,\nabla\phi\rangle=\big\langle(\R^n-\II)V_n,\phi\rangle,\quad\forall\phi\in H_0^1(\Om),
\end{equation}
and let  $\I_n z$ be  its Cl\'{e}ment-type interpolant in $\V^n$ (for the definition of the Cl\'ement-type interpolant and its properties we refer to \cite{BS, Clement, SZ}). %and \cite[Section 4.8]{BS}). 
Then we can prove the next auxiliary lemma.
\begin{lemma}\label{auxiliary}
Let $z$ be the solution of \eqref{dualprob1} and $\I_n z$ its Cl\'ement-type interpolant. Then, for all $V_n\in\V^n$, we have the following estimate for  $\R^nV_n$
\begin{equation}
\label{auxil1}
\|(\R^n-\II)V_n\|^2\le\big|\big\langle-\D^nV_n,z-\I_nz\big\rangle-\big\langle\nabla V_n,\nabla(z-\I_n z)\big\rangle\big|.
\end{equation}
\end{lemma} 
\begin{proof}
Using \eqref{dualprob1}, we obtain 
$$\|(\R^n-\II)V_n\|^2=\big\langle\nabla(\R^n-\II)V_n,\nabla z\big\rangle,$$
and thus, invoking the definition of the modified elliptic reconstruction \eqref{ellipticrec} and the orthogonality property \eqref{orthogonal}, we arrive at
$$\|(\R^n-\II)V_n\|^2=\langle-\D^n V_n,z-\I_n z\rangle-\big\langle\nabla V_n,\nabla(z-\I_n z)\big\rangle-\frac{1}{\alpha}\bg\big\langle(\R^n-\II)V_n,z\big\rangle.$$
Since both $\alpha$ and $\bg$ are positive, \eqref{auxil1} follows by  $\big\langle(\R^n-\II)V_n,z\big\rangle=\|\nabla z\|^2\ge 0$; cf.\ \eqref{dualprob1}.
\end{proof}
Since we use finite element spaces that are allowed to change from $t_{n-1}$ to $t_n$, we will need to work with quantities of the form $\|(\R^n-\II)V_n-(\R^{n-1}-\II)V_{n-1}\|$ for $V_n\in\V^n$ and $V_{n-1}\in\V^{n-1}.$ To estimate such a quantity, we consider the elliptic problem
$$\langle\nabla\hat z,\nabla\phi\rangle=\big\langle(\R^n-\II)V_n-(\R^{n-1}-\II)V_{n-1},\phi\big\rangle,\quad\forall\phi\in H_0^1(\Om)$$
with solution $\hat z$ and we denote by $\widehat\I_n\hat z$ its Cl\'{e}ment-type interpolant onto $\widehat\V_n.$
\begin{lemma}
For $V_n\in\V^n$ and $V_{n-1}\in\V^{n-1}$ we have that
\begin{equation}
\label{auxil2}
\begin{aligned}
\|(\R^n-\II)V_n-(\R^{n-1}-\II)V_{n-1}\|^2\le\big|&\langle\D^n V_n,\hat z-\widehat\I_n\hat z\rangle-\big\langle\nabla V_n,\nabla(\hat z-\widehat\I_n\hat z)\big\rangle\\
                                                                               &+\langle\D^{n-1}V_{n-1},\hat z-\widehat\I_n\hat z\rangle+\big\langle\nabla V_{n-1},\nabla(\hat z-\widehat\I_n\hat z)\big\rangle\big|.
\end{aligned}
\end{equation}
\end{lemma}
\begin{proof}
The proof is similar to the proof of Lemma~\ref{auxiliary}.
\end{proof}

To estimate a posteriori the errors $\|(\R^n-\II)V_n\|$ and $\|(\R^n-\II)V_n-(\R^{n-1}-\II)V_{n-1}\|,$ we use residual-type error estimators. To this end, for a given $V_n\in\V^n,\,0\le n\le N,$ we define the following $L^2-$elliptic estimator:
\begin{equation}
\label{EE1}
\begin{aligned}
\h(V_n):=\bigg\{\sum_{K\in\T_n}\Big(\|h_K^2(\D-\D^n)V_n\|^2_{L^2(K)}+\|h_K^{\frac  32}J[\nabla V_n]\|^2_{L^2(\p K)}\Big)\bigg\}^{\frac 12}.
\end{aligned}
\end{equation}
In case $d=1,$ the term with the discontinuities in \eqref{EE1} vanishes. For $V_n\in\V^n$ and $V_{n-1}\in\V^{n-1},\,1\le n\le N,$ we also define
\begin{equation}
\label{EE2}
\begin{aligned}
\het(V_n,V_{n-1}):=\bigg\{\sum_{K\in\widehat\T_n}\Big(\|h_K^2\big[(\D-\D^n)V_n&-(\D-\D^{n-1})V_{n-1}\big]\|^2_{L^2(K)}
\\&+\|h_{K}^{\frac 32}J[\nabla V_n-\nabla V_{n-1}]\|^2_{L^2(\check{\varSigma}_K^n)}\Big)\bigg\}^{\frac 12}.
\end{aligned}
\end{equation}
In view of the definition of $\h$ and of \eqref{auxil1}, the Lemma below is standard. Its proof is based on duality arguments and the elliptic regularity estimate for the Laplace operator. For details on the proof we refer, for example, to \cite{MN1, LM}.
%For its proof we refer, for example, to \cite[Section 4]{MN1} or \cite[Appendix B]{LM}.
%
\begin{lemma}
For all $V_n\in\V_n,\,0\le n\le N,$ it holds
\begin{equation}
\label{ellipticerror1}
\|(\R^n-\II)V_n\|\le C\h(V_n),
\end{equation}
where the constant $C$ depends only on the domain $\Om$ and the shape regularity of the family of triangulations.\qed
\end{lemma}
%

%Moreover, 
Similarly, by \eqref{auxil2}  the estimate \eqref{ellipticerror2} in the next lemma holds. For a detailed proof,  we refer to \cite{LM, BKM}.
\begin{lemma}
For $V_n\in\V^n$ and $V_{n-1}\in\V^{n-1},\,1\le n\le N,$ we have
\begin{equation}
\label{ellipticerror2}
\|(\R^n-\II)V_n-(\R^{n-1}-\II)V_{n-1}\|\le\widehat{C}\het(V_n,V_{n-1}),
\end{equation}
where the constant $\widehat{C}$ depends only on the domain $\Om$, the shape regularity of the triangulations, and the number of bisections necessary to pass from $\T_{n-1}$ to $\T_n$.\qed
\end{lemma}

%\begin{remark}
%{\color{blue} MAYBE TO SAY SOMETHING HERE ABOUT THE ORDER OF MAGNITUDE OF CONSTANTS $C, \widehat{C}$???????}
%\end{remark}
%

%--------------------------------------------------------------------------------
\subsection{Space and time-space reconstructions}
%-------------------------------------------------------------------------------
%
We first define the continuous, piecewise linear interpolant $U:[0,T]\to H_0^1(\Om)$ between the nodal values $U^{n-1}$ and $U^n$, i.e.,
\begin{equation}
\label{LIN}
U(t):=\ell_0^n(t)U^{n-1}+\ell_1^n(t)U^n,\quad t\in I_n,
\end{equation} 
with
$\displaystyle\ell_0^n(t):=\frac{t^n-t}{k_n}$  and $ \displaystyle\ell_1^n(t):=\frac{t-t_{n-1}}{k_n},\, t\in I_n.$
The space reconstruction of $U$, that was used in \cite{LM} to obtain of optimal order a posteriori error estimates for the backward Euler-Galerkin fully discrete scheme is given via
$$\omega(t):=\ell_0^n(t)\R^{n-1}U^{n-1}+\ell_1^n(t)\R^nU^n,\quad t\in I_n.$$
However, as the authors note in \cite{AMN1,LPP} to obtain optimal order in time a posteriori error estimates for the Crank-Nicolson method, a reconstruction in time is also needed. Here,  with the aid of the new elliptic reconstruction \eqref{ellipticrec}, we propose a  two-point time-space reconstruction for  linear Schr\"odinger equations and the method \eqref{CNGS}.% \TK{ A similar time-space reconstruction is used in \cite{BKM} in the case of heat equation.}
%Appropriate time-space reconstructions for Crank-Nicolson-Galerkin schemes were proposed by B\"{a}nsch, Karakatsani \& Makridakis in \cite{BKM} for the heat equation. 
%Here,  with the aim of the new elliptic reconstruction \eqref{ellipticrec}, we extend the definition of the two-point time-space reconstruction, introduced in \cite{BKM} for method \eqref{CNGS} for linear Schr\"odinger equations.
%
\begin{definition}[time-space reconstruction]
For $1\le n\le N,$ we define the two-point time-space reconstruction $\widehat{U}: I_n\to H_0^1(\Om)$ of the CNFE scheme \eqref{CNGS} as
\begin{equation}
\label{STR1}
\begin{aligned}
\widehat{U}(t):=\R^{n-1}U^{n-1}+\frac{t-t_{n-1}}{k_n}\big(\R^n&\varPi^nU^{n-1}-\R^{n-1}U^{n-1}\big)-\iu\alpha\int_{t_{n-1}}^t\R^n\Theta(s)\,ds
\\&-\iu\int_{t_{n-1}}^t\R^n\Pro^n G_U(s)\,ds+\int_{t_{n-1}}^t\R^n\Pro^nF(s)\,ds,\quad t\in I_n,
\end{aligned}
\end{equation}
where 
\begin{equation}
\label{data1}
G_U(t):=g(t_{n-\frac 12})U^{n-\frac 12}+\frac{2}{k_n}(t-t_{n-\frac 12})\Big[g(t_{n-\frac 12})U^{n-\frac 12}-g(t_{n-1})U^{n-1}\Big]
\end{equation}
and
\begin{equation}
\label{data2}
F(t):=f(t_{n-\frac 12})+\frac{2}{k_n}(t-t_{n-\frac 12})\Big[f(t_{n-\frac 12})-f(t_{n-1})\Big],
\end{equation}
denote the linear interpolants of $gU$ and $f$, respectively, at the nodes $t_{n-1}$ and $t_{n-\frac 12},$ and 
\begin{equation}
\label{theta}
\Theta(t):=\ell_0^n(t)\varPi^n(-\D^{n-1})U^{n-1}+\ell_1^n(t)(-\D^n)U^n.
\end{equation}
\end{definition}

In order to write compactly method \eqref{CNGS} and the reconstruction $\widehat U$, we introduce  the notation
\begin{equation}
\label{compnot}
W(t):=\Big(\iu\alpha\Theta+\iu\Pro^nG_U-\Pro^nF\Big)(t),\quad t\in I_n.
\end{equation}
With this notation, the reconstruction $\widehat U$ is rewritten as 
\begin{equation}
\label{STR2}
\widehat U(t)=\R^{n-1}U^{n-1}+\frac{t-t_{n-1}}{k_n}\big(\R^n\varPi^nU^{n-1}-\R^{n-1}U^{n-1}\big)-\int_{t_{n-1}}^t\R^nW(s)\,ds,\quad t\in I_n,
\end{equation}
and method \eqref{CNGS} as
\begin{equation}
\label{CNGS2}
\frac{U^n-\varPi^n U^{n-1}}{k_n}+W(t_{n-\frac{1}{2}})=0,\quad 1\le n \le N.
\end{equation}
Note that in each $[t_{n-1},t_n],$ $W$ is a linear polynomial between the values $\big(t_{n-1}, W(t_{n-1})\big)$ and $\big(t_{n-\frac{1}{2}}, W(t_{n-\frac{1}{2}})\big).$ Thus, it is straightforward to see that
\begin{equation}
\label{wdiff}
W(t)-W(t_{n-\frac 12})=(t-t_{n-\frac 12})\p_t W(t),\quad t\in I_n.
\end{equation}
\begin{proposition}
For $1\le n\le N,$ there holds
$$\widehat U(t_{n-1}^+)=\R^{n-1}U^{n-1} \quad \text{ and } \quad \widehat U(t_n)=\R^n U^n.$$
In particular, $\widehat U$ is continuous in time. Furthermore, it satisfies
\begin{equation}
\label{erreq1}
\begin{aligned}
\p_t\widehat U+\iu\alpha\R^n\Theta+\iu\R^n\Pro^nG_U=\R^n\Pro^n F+\frac{\R^n\varPi^n U^{n-1}-\R^{n-1}U^{n-1}}{k_n}\ \text{ in }\ I_n.
\end{aligned}
\end{equation}
\end{proposition}
\begin{proof}
That $\widehat U(t_{n-1}^+)=\R^{n-1}U^{n-1}$ is obvious from the definition of $\widehat U.$ Moreover,
$$\widehat U(t_n)=\R^n\varPi^n U^{n-1}-\R^n\int_{I_n}W(t)\,dt.$$
Since $W$ is a linear polynomial in time in $I_n$, we have that $\int_{I_n}W(t)\,dt=k_nW(t_{n-\frac 12})$ and that $\widehat U(t_n)=\R^nU^n$  follows invoking \eqref{CNGS2}. Finally, \eqref{erreq1} is an immediate consequence of differentiation in time of \eqref{STR1}.  
\end{proof}
We conclude the section by computing the difference $\widehat U-\omega.$ For this, we introduce, for $1\le n\le N,$ the notation
\begin{equation}
\label{disctd}
\dtd:=\frac 2{k_n}\Big[W(t_{n-\frac{1}{2}})-W(t_{n-1})\Big].
\end{equation}
\begin{lemma} [the difference $\hat U-\omega$]
The difference $\widehat U-\omega$ satisfies
\begin{equation}
\label{diffrec1}
(\widehat U-\omega)(t)=\frac 12(t_n-t)(t-t_{n-1})\R^n\dtd,\quad t\in I_n.
\end{equation}
\end{lemma}
\begin{proof}
Using the definitions of $\widehat U$ and $\omega$ and the method in the form \eqref{CNGS2} we obtain
$$\p_t(\widehat U-\omega)(t)=-\R^n\Big(W(t)-W(t_{n-\frac 12})\Big).$$
Thus, using \eqref{wdiff} and the fact that $\int_{t_{n-1}}^t(s-t_{n-\frac 12})\,ds=\frac{1}{2}(t-t_{n-1})(t-t_n),$ we obtain
\begin{equation}
\label{diffrec2}
(\widehat U-\omega)(t)=\frac 12(t_n-t)(t-t_{n-1})\R^n\p_t W(t),\quad t\in I_n.
\end{equation}
Equality \eqref{diffrec1} follows now from \eqref{diffrec2}, by noting that $\p_t W(t)=\dtd,\,t\in I_n$; cf.\ \eqref{disctd} and the definition \eqref{compnot} of $W(t)$.
\end{proof}
%----------------------------------------------------------------------------------------------------
\section{A Posteriori Error Estimates in the $L^\infty(L^2)-$norm}\label{apost}
%----------------------------------------------------------------------------------------------------
%--------------------------------------------------------------------------------
\subsection{Main Ideas}
%----------------------------------------------------------------------------------
In this section, we establish a posteriori error estimates in the $L^\infty(L^2)-$norm for problem \eqref{LS1}, using the tools  developed in the previous section. To this end, we denote by $e:=u-U$ the error, where recall that $U$ is the piecewise linear interpolant between the nodal values $U^{n-1}$ and $U^n$; cf.\ \eqref{LIN}. To achieve proving optimal order a posteriori error estimates in the $L^\infty(L^2)-$norm for \eqref{LS1} we split the error as
$$e:=\hat\rho+\sigma+\epsilon,$$
with $\hat\rho:=u-\widehat U$, $\sigma:=\widehat U-\omega$ and $\epsilon:=\omega-U.$ %We follow the notation and terminology of \cite{BKM}, i.e.,
 We refer to $\hat\rho$ as the main error, to $\sigma$ as the time-reconstruction error and to $\epsilon$ as the elliptic-reconstruction error. The term $\sigma$ measures the error due to the reconstruction in time. This term is of optimal order in time, cf.\ \eqref{diffrec1}, but not yet an a posteriori quantity. It can be estimated a posteriori using the residual-type error estimators. The residual estimators will also be used for the direct estimation of the elliptic-reconstruction error.

Finally, as we shall see, the main error $\hat\rho$ satisfies a perturbation of the original PDE and it will be bounded by the perturbed terms using energy techniques. The perturbed terms are either a posteriori quantities of optimal order, or can be estimated a posteriori by estimators of optimal order. These terms will include quantities that measure the time and space errors, the effect of mesh changes and the variation of the data $f$ and $g$.
We now proceed with the estimation of $\sigma$ and $\epsilon$ in Propositions \ref{timerecerror} and \ref{ellipterror}, respectively.
\begin{proposition}[estimation of the time-reconstruction error]\label{timerecerror}
For $1\le m\le N$, the following estimate is valid for the time reconstruction error $\sigma=\widehat U-\omega$:
\begin{equation}
\label{timerecer}
\max_{0\le t\le t_m}\|\sigma(t)\|\le\E_m^{\TT,0}\quad \text{ with }\quad \E_m^{\TT,0}:=\max_{1\le n\le m}\frac{k_n^2} 8\Big[\|\dtd \|+C\h(\dtd)\Big].
\end{equation}
\end{proposition}
\begin{proof}
We write $\R^n\dtd =\dtd +(\R^n-\II)\dtd$  and the desirable result now follows using \eqref{ellipticerror1} and \eqref{diffrec1}.
\end{proof}
\begin{proposition}[estimation of the elliptic error]\label{ellipterror}
For the elliptic error $\epsilon=\omega-U$ we have, for $1\le m\le N$:
\begin{equation}
\label{ellipticerr}
\max_{0\le t\le t_m}\|\epsilon(t)\|\le C\E_m^{\Ss,0}\quad \text{ with }\quad \E_m^{\Ss,0}:=\max_{0\le n\le m}\h(U^n).
\end{equation}
\end{proposition}
\begin{proof}
For $t\in I_n,$ $\epsilon=\ell_0^n(t)(\R^{n-1}-\II)U^{n-1}+\ell_1^n(t)(\R^n-\II)U^n.$ Hence, 
$$\|\epsilon(t)\|\le\max\Big\{\|(\R^{n-1}-\II)U^{n-1}\|,\|(\R^n-\II)U^n\|\Big\},\quad t\in I_n,$$
from where we immediately conclude \eqref{ellipticerr}, in view of \eqref{ellipticerror1}. 
\end{proof}

%-----------------------------------------------------------------------------------
\subsection{Estimation of the main error}
%-----------------------------------------------------------------------------------
%
In view of \eqref{erreq1} we see that the reconstruction $\widehat U$ satisfies, for $t\in I_n,$ the equation
\begin{equation}
\label{erreq2}
\begin{aligned}
\langle\p_t\widehat U(t),\phi\rangle+\iu\alpha\langle\nabla\widehat U(t),\nabla\phi\rangle+\iu\langle g(t)\widehat U(t),\phi\rangle=\langle R(t),\phi\rangle,\quad\forall\phi\in H_0^1(\Om),
\end{aligned}
\end{equation}
with
\begin{equation}
\label{residual}
\Rr(t):=-\R^nW(t)+\frac{\R^n\varPi^nU^{n-1}-\R^{n-1}U^{n-1}}{k_n}+\iu\big(-\alpha\D+g(t)\big)(\omega+\sigma)(t),\quad t\in I_n.
\end{equation}
\begin{proposition}[error equation for $\hat\rho$]\label{errormain}
The main error $\hat\rho=u-\widehat U$ satisfies, for $t\in I_n,$ the equation
\begin{equation}
\label{erreq3}
\begin{aligned}
\langle\p_t\hat \rho(t),\phi\rangle+\iu\alpha\langle\nabla\hat \rho(t),\nabla\phi\rangle+\iu\langle g(t)\hat\rho(t),\phi\rangle=\sum_{j=1}^4\langle R_j(t),\phi\rangle,\quad\forall\phi\in H_0^1(\Om),
\end{aligned}
\end{equation}
where the residuals $\Rr_j,\,1\le j\le 4,$ are given by
\begin{equation}
\label{residual1}
\Rr_1(t):=(\R^n-\II)W(t)-\frac{\R^n\varPi^nU^{n-1}}{k_n}+\iu\alpha\ell_0^n(t)(\II-\varPi^n)\D^{n-1}U^{n-1},
\end{equation}
\begin{equation}
\label{residual2}
\Rr_2(t):=\frac\iu 2(t_n-t)(t-t_{n-1})\Big[\big(-\alpha\D^n+g(t)\big)\dtd+\big(g(t)-\bg)(\R^n-\II)\dtd\Big],
\end{equation}
\begin{equation}
\label{residual3}
\Rr_3(t):=\iu\big(g(t)-\bg\big)\Big[\ell_0^n(t)(\II-\R^{n-1})U^{n-1}+\ell_1^n(t)(\II-\R^n)U^n\Big],
\end{equation}
and
\begin{equation}
\label{residual4}
\Rr_4(t):=\iu\big(\Pro^nG_U(t)-(gU)(t)\big)+\big(f(t)-\Pro^n F(t)\big).
\end{equation}
\end{proposition}
\begin{proof}
Subtracting \eqref{erreq2} from \eqref{LS2} we obtain, for $t\in I_n,$
\begin{equation}
\label{erreq4}
\begin{aligned}
\langle\p_t\hat \rho(t),\phi\rangle+\iu\alpha\langle\nabla\hat \rho(t),\nabla\phi\rangle+\iu\langle g(t)\hat\rho(t),\phi\rangle=\langle f(t),\phi\rangle-\langle R(t),\phi\rangle,\quad \forall\phi\in H_0^1(\Om).
\end{aligned}
\end{equation}
We further write
$$\big(-\alpha\D+g(t)\big)\omega(t)=\big(-\alpha\D+\bg)\omega(t)+\big(g(t)-\bg)\omega(t),$$
where we recall that $\omega(t)=\ell_0^n(t)\R^{n-1}U^{n-1}+\ell_1^n(t)\R^nU^n,\,t\in I_n.$ Thus \eqref{ellipticrec}, \eqref{theta} yield
\begin{equation}
\label{erraux1}
\begin{aligned}
\big\langle\big(-\alpha\D+g(t)\big)\omega(t),\phi\big\rangle=&\alpha\langle\Theta(t),\phi\rangle+\alpha\ell_0^n(t)\langle(\varPi^n-\II)\D^{n-1}U^{n-1},\phi\rangle+\langle (gU)(t),\phi\rangle\\
                                                                                                          &+\big\langle\big(g(t)-\bg)\big[\ell_0^n(t)(\R^{n-1}-\II)U^{n-1}+\ell_1^n(t)(\R^n-\II)U^n\big],\phi\big\rangle.
\end{aligned}
\end{equation}
Similarly, in view of \eqref{diffrec1}, we obtain
\begin{equation}
\label{erraux2}
\begin{aligned}
\big\langle(-\alpha\D+g(t)\big)\sigma(t),\phi\rangle=&\frac 12 (t_n-t)(t-t_{n-1})\times\\
                                                                                           &\big\langle\big(-\alpha\D^n+g(t)\big)\dtd+\big(g(t)-\bg)(\R^n-\II)\dtd,\phi\big\rangle.
\end{aligned}
\end{equation}
Combining \eqref{erreq4}, \eqref{residual} with \eqref{erraux1}, \eqref{erraux2} and using  \eqref{compnot} we arrive at \eqref{erreq3}.
\end{proof}
Next, we prove the following auxiliary lemma.
\begin{lemma}\label{resrew}
The residual $\Rr_1$ in \eqref{residual1} can be rewritten as 
\begin{equation}
\label{residual11}
\begin{aligned}
\Rr_1(t)=(t-t_{n-\frac 12})(\R^n-\II)\dtd&-\frac{(\R^n-\II)U^n-(\R^{n-1}-\II)U^{n-1}}{k_n}
\\&+(\II-\varPi^n)\big(\iu\alpha\ell_0^n(t)\D^{n-1}U^{n-1}+\frac{U^{n-1}}{k_n}\big),\quad t\in I_n.
\end{aligned}
\end{equation}
\end{lemma}
\begin{proof}
We just note, using the method in the form \eqref{CNGS2}, that
\begin{equation*}
\begin{aligned}
(\R^n-\II)W(t_{n-\frac 12})&-\frac{\R^n\varPi^nU^{n-1}-\R^{n-1}U^{n-1}}{k_n}=\frac{(\II-\R^n)U^n-(\II-\R^{n-1})U^{n-1}}{k_n}+
(\II-\varPi^n)\frac{U^{n-1}}{k_n}.
\end{aligned}
\end{equation*}
The result follows in light of \eqref{wdiff}, because $\p_tW(t)=\dtd$ for $t\in I_n.$
\end{proof}
Proposition~\ref{errormain} and Lemma~\ref{resrew} together with energy methods, lead to the following a posteriori estimation in the $L^\infty(L^2)-$norm for the main error $\hat\rho.$
\begin{proposition}[estimation of the main error]\label{mainerror}
Let $p_n:=\sup_{\Om\times I_n}|g(x,t)-\bg|,\,1\le n\le N.$ Then, for the main error $\hat\rho=u-\widehat U$ and $1\le m\le N$, it holds that
\begin{equation}
\label{estmainerr}
\max_{0\le t\le t_m}\|\hat\rho(t)\|\le\|u_0-\R^0U^0\|+\E_m^{\TT,1}+C(\E_m^{\Ss,1}+\E_m^{\Ss,2})+\widehat C\E_m^{\Ss,3}+\E_m^{\mathrm{C}}+\E_m^{\mathrm{D}},
\end{equation}
where the time estimator $\E_m^{\TT,1}$ is given by
\begin{equation}
\label{timeest}
\begin{aligned}
\E_m^{\TT,1}:=&\sum_{n=1}^m\int_{t_{n-1}}^{t_n}\frac{(t_n-t)(t-t_{n-1})}{2}\|\big(-\alpha\D^n+g(t)\big)\dtd\|\,dt
+C\sum_{n=1}^m\frac{k_n^3}{24}p_n\h(\dtd),
\end{aligned}
\end{equation}
the space estimators $\E_m^{\Ss,j},\,1\le j\le 3,$ are given by
\begin{equation}
\label{spaceest}
\begin{aligned}
\E_m^{\Ss,1}:=\sum_{n=1}^m\frac{k_n^2}{4}&\h(\dtd),\qquad \E_m^{\Ss,2}:=\sum_{n=1}^m\frac{k_n}{2}p_n\big(\hh(U^{n-1})+\h(U^n)\big), 
\\&\text{and }\quad \E_m^{\Ss,3}:=\sum_{n=1}^mk_n\het(\frac{U^n}{k_n},\frac{U^{n-1}}{k_n}),
\end{aligned}
\end{equation}
and the coarsening and data estimators $\E_m^{\mathrm{C}}$ and $\E_m^{\mathrm{D}}$ are
\begin{equation}
\label{coarsening}
\E_m^{\mathrm{C}}:=\sum_{n=1}^m\int_{t_{n-1}}^{t_n}\|(\II-\varPi^n)\big(\frac{U^{n-1}}{k_n}+\iu\alpha\ell_0^n(t)\D^{n-1}U^{n-1}\big)\|\,dt,
\end{equation}
and
\begin{equation}
\label{data}
\E_m^{\mathrm{D}}:=\sum_{n=1}^m\int_{t_{n-1}}^{t_n}\Big[\|\Pro^nG_U(t)-(gU)(t)\|+\|f(t)-\Pro^nF(t)\|\Big]\,dt,
\end{equation}
respectively.
\end{proposition}
\begin{proof}
Setting $\phi=\hat\rho$ in \eqref{erreq3} and taking real parts yields
$$\frac 12\frac{d}{dt}\|\hat\rho(t)\|^2=\Rea\big\langle\sum_{j=1}^4\Rr_j(t),\hat\rho(t)\big\rangle\le\sum_{j=1}^4\|\Rr_j(t)\|\,\|\hat\rho(t)\|,\quad t\in I_n,$$
or,
\begin{equation}
\label{energy1}
\max_{0\le t\le t_m}\|\hat\rho(t)\|\le\|\hat\rho(0)\|+\sum_{j=1}^4\int_0^{t_m}\|\Rr_j(t)\|\,dt.
\end{equation}
Then, it is easily seen that
\begin{equation}
\label{estres1}
\int_0^{t_m}\|\Rr_1(t)\|\,dt\le \E_m^{\Ss,1}+\E_m^{\Ss,3}+\E_m^{\mathrm{C}};
\end{equation}
cf.\ \eqref{residual11}, and
\begin{equation}
\label{estres2}
\int_0^{t_m}\|\Rr_2(t)\|\,dt\le\E_m^{\TT,1},\quad \int_0^{t_m}\|\Rr_3(t)\|\,dt\le\E_m^{\Ss,2},\quad \int_0^{t_m}\|\Rr_4(t)\|\,dt\le\E_m^{\mathrm{D}};
\end{equation}
cf.\ \eqref{residual2}--\eqref{residual4}. Going back to \eqref{energy1} and plugging in \eqref{estres1}--\eqref{estres2} we readily obtain \eqref{estmainerr}.
\end{proof}
\begin{remark}[optimal order of the estimators in \eqref{estmainerr}]
\rm
It is clear that the space estimators $\E_m^{\Ss,j},\,1\le j\le 3,$ are expected to be of optimal order of accuracy in space. In fact, estimator $\E_m^{\Ss,1}$ is expected to be of optimal order in space and of order one in time, i.e., it is a superconvergent term. As far as the first part of the time estimator $\E_m^{\TT,1}$ is concerned, we note that
$$\int_{t_{n-1}}^{t_n}\frac{(t_n-t)(t-t_{n-1})}{2}\big\|\big(-\alpha\D^n+g(t)\big)\dtd\big\|\,dt\le\frac{k_n^3}{12}\sup_{t\in I_n}\big\|\big(-\alpha\D^n+g(t)\big)\dtd\big\|.$$ 
So, it is expected to be of optimal order  of accuracy in time. Numerically, this term can be computed by invoking a quadrature in time, which is at least  second order accurate (i.e., at least as accurate as the accuracy of the discretization method in time). The second part of $\E_m^{\TT,1}$ is expected to be of optimal order in both time and space. On the other hand, note that estimator $\E_m^{\mathrm{C}}$ is not identically zero, only during the coarsening procedure. Finally, for the estimators related to the data of the problem we have
$\|u_0-\R^0U^0\|\le\|u_0-U^0\|+C\eta_{\V^0}(U^0)$
and
$\|\Pro^nG_{U}(t)-(gU)(t)\|\le\|(\II-\Pro^n)G_U(t)\|+\|(G_U-gU)(t)\|.$
The term $\|f(t)-\Pro^nF(t)\|$ is handled similarly. Thus, it is straightforward to see that $\E_m^{\mathrm{D}}$ can be split into optimal order estimators in time and space, while $\|u_0-\R^0U^0\|$ is easily estimated a posteriori via optimal order estimators in space.
\end{remark}
\begin{remark}[the constants $p_n$]
\rm
For the constants $p_n$ we note that $p_n\le p_{n,1}+p_{n,2}$ with $p_{n,1}:=\sup_{\Om\times I_n}|g(x,t)-g(x,t_{n-\frac 12})|$ and $p_{n,2}=\frac 12 \big[\sup_{x\in\Om}|g(x,t_{n-\frac 12})|-\inf_{x\in\Om}|g(x,t_{n-\frac 12})|\big].$ Therefore, $p_{n,1}=\mathcal{O}(k_n),$ while $p_{n,2}$ is relatively small, provided that $g$ does not change much, with respect to the spatial variable. More precisely, $p_{n,2}\equiv 0$ when $g$ is constant in space, while the estimators that are multiplied by $p_n$ in \eqref{estmainerr} vanish for constant potentials. This particular behavior of the estimators is natural from physical point of view.
\end{remark}

We conclude with the main theorem of the paper.
\begin{theorem}[a posteriori error estimate in the $L^\infty(L^2)-$norm]
Let $u$ be the exact solution of \eqref{LS1} and let $U$ be the continuous approximation \eqref{LIN} of $u$ related to the modified Crank-Nicolson-Galerkin method \eqref{CNGS}. Then, the following  estimate is valid for $1\le m\le N$:
\begin{equation}
\label{apostest}
\begin{aligned}
\max_{0\le t\le t_m}\|(u-U)(t)\|\le\|u_0-\R^0U^0\|+\E_m^{\TT,0}+\E_m^{\TT,1}+C\sum_{j=0}^2\E_m^{\Ss,j}+\widehat C\E_m^{\Ss,3}+\E_m^{\mathrm{C}}+\E_m^{\mathrm{D}},
\end{aligned}
\end{equation}
where $\E_m^{\TT,1},\, \E_m^{\Ss,j},\,1\le j\le3,\,\E_m^{\mathrm{C}},\,\E_m^{\mathrm{D}}$ are given by \eqref{timeest}, \eqref{spaceest},    \eqref{coarsening} and \eqref{data} and $\E_m^{\TT,0},\,\E_m^{\Ss,0}$ are as in \eqref{timerecer} and \eqref{ellipticerr}, respectively.
\end{theorem}
\begin{proof}
We write $u-U=\hat\rho+\sigma+\epsilon$, whence, for $1\le m\le N,$ 
$$\max_{0\le t\le t_m}\|(u-U)(t)\|\le\max_{0\le t\le t_m}\|\hat\rho(t)\|+\max_{0\le t\le t_m}\|\sigma(t)\|+\max_{0\le t\le t_m}\|\epsilon(t)\|.$$
Estimate \eqref{apostest} is now an immediate consequence of Propositions~\ref{ellipterror},~\ref{timerecerror} and~\ref{mainerror}.
\end{proof}

%-----------------------------------------------------------------------------------
\section{Numerical Experiments: Uniform Partition}\label{unif}
%-----------------------------------------------------------------------------------
In this section, we perform various numerical experiments for the one-dimensional linear semiclassical Schr\"odinger equation:
\begin{equation}
\label{1dsemiclassical}
\p_t u-\mathrm{i}\frac{\ep}{2}\p_{xx}u+\frac{\mathrm{i}}{\ep}V(x,t)u=0\quad\text{ in }(a,b)\times(0,T],
\end{equation}
using uniform partitions. 
Our experiments, not only illustrate and complement our theoretical results, but also give important information in several other interesting aspects, like the behavior of the estimators with respect to the parameter $\ep$. At the moment, the particular behavior can only be proven formally; cf.\ Subsection~\ref{sensitivity}. In all of the numerical experiments, the initial data is of the well known semiclassical WKB form:
\begin{equation}
\label{incond}
u_0(x)=\sqrt{n_0(x)}\mathrm{e}^{\mathrm{i}\frac{S_0(x)}{\ep}}.
\end{equation}
In \eqref{incond}, $n_0$ and $S_0$ are real and smooth functions on $[a,b]$. In addition, $n_0$ is positive on $(a,b)$ and vanishes (numerically) at the endpoints $a$ and $b$.

The modified Galerkin-Crank-Nicolson method \eqref{CNGS} and the corresponding a posteriori error estimators for problem \eqref{1dsemiclassical}-\eqref{incond} with homogeneous Dirichlet  boundary conditions, were implemented in a double precision C-code, using  B-splines of degree $r,\, r\in\mathbb{N},$ as a basis for the finite element space $\V^n,\, 0\le n\le N$. The involved projections $\varPi^n$ and $\Pro^n$ in \eqref{CNGS} are taken to be the $L^2-$projection onto $\V^n.$ 
%\IK{(Here, maybe it is a good idea to give some more information on the compiler and the OS we used.)}

In what follows, we present some characteristic examples that allow us to verify the correct order of convergence of the estimators in time and space, and their dependence on the Planck constant $\ep$. We also report on the relation between the time and space mesh sizes with respect to $\ep$ in order to have convergence. %In the last subsection we investigate wether these restrictive conditions (cf.\ \eqref{econdk},\eqref{econdh}) can be relaxed using an adaptive algorithm.
%-----------------------------------------------------------------------------------
%----------------------------------------------------------------------------------
\subsection{$\text{EOC}$  of the estimators}
%----------------------------------------------------------------------------------
%
We proceed by studying two different cases.  The first one concerns time-independent potentials, while in the second one we consider a time-dependent potential.

\textsc{{Experiment 1}} (Time-independent potentials). Here, we consider three well-known types of potential: a constant potential, a harmonic oscillator and a double-well potential (\cite{S1,FFS,MM}). In all three examples, the Planck constant is taken to be of order 1. More precisely, we study the following cases:
\begin{enumerate}[a.]
\item $V(x)=100$, $\sqrt{n_0(x)}=\mathrm{e}^{-\frac{25}2 x^2}$, $S_0(x)=\frac{x^2}{2},$  and $\ep=1$;
\item $V(x)=\frac{x^2}{2}$, $\sqrt{n_0(x)}=\mathrm{e}^{-25(x-0.5)^2}$, $S_0(x)=1+x,$ and $\ep=0.5$;
\item $V(x)=(x^2-0.25)^2=x^4-\frac 1 2 x^2+\frac 1 {16}$, $\sqrt{n_0(x)}=\mathrm{e}^{-\frac{25}2 x^2}$, $S_0(x)=-\frac{1}{5}\ln\Big(\mathrm{e}^{5(x-0.5)}+\mathrm{e}^{-5(x-0.5)}\Big),$ and  $\ep=0.25.$ 
\end{enumerate}
All  computations are performed in $[a,b]\times[0,T]=[-2,2]\times[0,1].$ Our purpose  is to compute the \emph{experimental order of convergence ($\text{EOC}$)} of the a posteriori error estimators at the final time  $T=1$. For this, we consider uniform partitions in both time and space. If we denote by $r$ the degree of B-splines used for the discretization in space, then in each implementation, the relation between the mesh size $h$ and the time step $k$ is taken to be
\begin{equation}
\label{hkrelation}
h\approx k^{\frac{2}{r+1}}
\end{equation}
with equality, whenever possible. We also denote by $M=\frac{b-a}{h}.$ Then, for each space estimator $\E_N^{\Ss,j},\,0\le j\le 3,$ the $\text{EOC}$ is computed as
\begin{equation}
\label{EOCS}
\text{EOC}:=\frac{\log\Big(\E_N^{\Ss,j}(\ell)/ \E_N^{\Ss,j}(\ell+1)\Big)}{\log\Big(M(\ell+1)/M(\ell)\Big)},
\end{equation}
where $\E_N^{\Ss,j}(\ell)$ and $\E_N^{\Ss,j}(\ell+1)$ denote the value of the estimators in two consecutive implementations with mesh sizes $h(\ell)=\frac{b-a}{M(\ell)}$ and $h(\ell+1)=\frac{b-a}{M(\ell+1)}$, respectively. Note that $\E_N^{\Ss,1}$ is expected to be of optimal order in space and of order $1$ in time, i.e., it is a superconvergent term. Therefore, the $\text{EOC}$ we expect to observe is $h^{r+1}\cdot h^{\frac{r+1}{2}}=h^{\frac{3}{2}(r+1)},$ due to \eqref{hkrelation} and \eqref{EOCS}. Similarly, for the time estimators $\E_N^{\TT,j},\,0\le j\le1,$ the $\text{EOC}$ is computed as 
\begin{equation}
\label{EOCT}
 \text{EOC}:=\frac{\log\Big(\E_N^{\TT,j}(\ell)/ \E_N^{\TT,j}(\ell+1)\Big)}{\log\Big(k(\ell)/k(\ell+1)\Big)}. 
 \end{equation}

 We are also interested in computing the \emph{effectivity index}, defined as the ratio between the total a posteriori error estimator and the corresponding norm of the exact error. Since we do not have at our disposal the exact solution for the three examples, we compute a reference solution $u_{\mathrm{ref}}$ instead, by taking very fine mesh and time step. In particular, we take as $k_{\mathrm{ref}}^{-1}=40960$, while in space we discretize by B-splines of degree 5 and take as $h_{\mathrm{ref}}^{-1}=120.$ Then, the reference error is defined as $\mathrm{Eref}:=\displaystyle\max_{0\le n\le N}\|u_{\mathrm{ref}}(t_n)-U^n\|.$ In addition, we define 
 $$\E_N^{\text{total}}:=\|u_0-U^0\|+\eta_{\V^0}(U^0)+\E_N^{\TT,0}+\E_N^{\TT,1}+\sum_{j=0}^3\E_N^{\Ss,j}+\E_N^{\mathrm{D}},$$
 and we compute the effectivity index $ei$ as $ei:=\E_N^{\text{total}}/\mathrm{Eref}.$ Note that for uniform partitions, the coarsening estimator $\E_N^{\mathrm{C}}$ is identically zero. Our findings are reported in Tables~\ref{tbl1}--\ref{tbl6}.
 
 In the case of constant potential $V(x)=100$, we discretize in space by linear B-splines. We recall that in this case $\E_N^{\Ss,2}$ is identically zero and does not appear in Table~\ref{tbl1}. As we see in Tables~\ref{tbl1}, \ref{tbl2}, all estimators decrease with the correct order. 
\begin{table}[htb!]
\begin{center}
\renewcommand{\arraystretch}{1.0}
\begin{tabular}{|l||cc|cc|ccc|}\hline
$M$ &  $ \E_N^{\Ss,0}$ & $\text{EOC}$&$\E_N^{\Ss,1}$ & $\text{EOC}$&$\E_N^{\Ss,3}$ &$\text{EOC}$&
\\ \hline\hline
$640$  &$5.0445$e$-04$ &-- &$1.3289$e$-02$ &-- &$6.1493$e$-02$ & --& \\%*[-1pt]
$1280$  &$1.2609$e$-04$ &$2.0003$ & $1.7796$e$-03$& $2.9006$&$1.6361$e$-02$  &$1.9102$ &   \\
$2560$  &$3.1522$e$-05$  &$2.0000$ & $2.2677$e$-04$& $2.9722$&$4.1610$e$-03$ &$1.9753$ & \\
$5120$  &$7.8804$e$-06$ & $2.0000$&$2.8488$e$-05$ & $2.9928$ &$1.0448$e$-03$&$1.9937$ &  \\
$10240$  & $1.9701$e$-06$& 2.0000&$3.5665$e$-06$ &$2.9978$ & $2.6150$e$-04$ &$1.9983$ & \\ \hline
\end{tabular}\\[2ex]
\caption{Space estimators $\E_N^{\Ss,j},\, j=0,1,3,$ and $\text{EOC}$ for Experiment 1a.} \label{tbl1}
\end{center}
\end{table}
\begin{table}[ht]
\begin{center}
\renewcommand{\arraystretch}{1.0}
\begin{tabular}{|l||cc|cc||cccc|}\hline
$k^{-1}$ &  $ \E_N^{\TT,0}$ & $\text{EOC}$ & $\E_N^{\TT,1}$ & $\text{EOC}$ &  $\mathrm{Eref}$ & $\E^{\text{total}}_N$ & $ei$ &
\\ \hline\hline
$160$  &$3.1810$e$-02$ &-- & $2.3471$&-- & $1.1329$ & $2.4547$&$2.1668$ & \\
$320$  &$8.2836$e$-03$ &$1.9411$ &$6.1356$e$-01$ &$1.9356$ &$5.4529$e$-01$ &$6.4024$e$-01$ &$1.1741$ &   \\
$640$  &$2.0935$e$-03$  &$1.9843$ & $1.5524$e$-01$&$1.9827$ &$1.5026$e$-01$ &$1.6178$e$-01$ &$1.0767$ & \\
$1280$  & $5.2481$e$-04$& $1.9960$&$3.8928$e$-02$ &$1.9956$  &$3.8853$e$-02$ & $4.0541$e$-02$&$1.0434$ & \\
$2560$ &$1.3129$e$-04$&$2.0090$ &$9.7395$e$-03$ &$1.9989$  & $9.6973$e$-03$& $1.0140$e$-02$& $1.0456$&  \\ \hline
\end{tabular}\\[2ex]
\caption{Time estimators $\E_N^{\TT,j},\,j=0,1,$ and $\text{EOC}$, total estimator  $\E^{\text{total}}_N$, reference error $\mathrm{Eref},$ and effectivity index $ei$  for Experiment 1a.} \label{tbl2}
\end{center}
\end{table}
We observe that the total error is mainly due to the time estimator $\E_N^{\TT,1}$, while the effectivity index is around 1.04, i.e., the total estimator $\E_N^{\text{total}}$ is very close to the reference error. However constant potentials are the simplest; actually, from physical point of view, having a constant potential is like having no potential at all. 

In Tables~\ref{tbl3}, \ref{tbl4} the results for the harmonic oscillator (1b) are presented. We use  quadratic B-splines for the discretization in space. The correct order of convergence is observed for all  estimators.  The dominant estimator for the harmonic oscillator is $\E_N^{\Ss,3}$, while the effectivity index tends asymptotically to the constant value $4.5$.
\begin{table}[htb!]
\begin{center}
\renewcommand{\arraystretch}{1.0}
\begin{tabular}{|l||cc|cc|cc|ccc|}\hline
$M$ &  $ \E_N^{\Ss,0}$ & $\text{EOC}$&$\E_N^{\Ss,1}$ & $\text{EOC}$& $\E_N^{\Ss,2}$& $\text{EOC}$&$\E_N^{\Ss,3}$ &$\text{EOC}$&
\\ \hline\hline
$75$  &$1.3042$e$-02$ &-- &$1.5619$e$-01$ &-- &$2.4735$e$-02$ & --&$6.3657$e$-01$ & --&\\
$120$  &$3.1817$e$-03$ &$3.0016$ &$2.1398$e$-02$ &$4.2293$ &$6.0463$e$-03$  &$2.9974$ &$1.6747$e$-01$  &$2.8410$ & \\
$185$  &$8.6805$e$-04$  &$3.0008$ &$3.0282$e$-03$ &$4.5172$ &$1.6509$e$-03$ &$2.9989$ &$4.6712$e$-02$ &$2.9497$ &\\
$295$  &$2.1405$e$-04$ &$3.0004$ &$3.7707$e$-04$ & $4.4646$ &$4.0737$e$-04$&$2.9989$ &$1.1585$e$-02$ &$2.9881$ & \\
$470$  &$5.2927$e$-05$  &$3.0000$ &$4.6730$e$-05$ &$4.4831$ &$1.0077$e$-04$  &$2.9992$ &$2.8684$e$-03$ &$2.9972$ &\\
$750$  &$1.3025$e$-05$   &$3.0000$ &$5.7532$e$-06$ &$4.4820$ &$2.4807$e$-05$ &$2.9993$ &$7.0610$e$- 04$& $2.9994$& \\ \hline
\end{tabular}\\[2ex]
\caption{Space estimators $\E_N^{\Ss,j},\, 0\le j\le 3,$ and $\text{EOC}$ for Experiment 1b.} \label{tbl3}
\end{center}
\end{table}
\begin{table}[ht]
\begin{center}
\renewcommand{\arraystretch}{1.0}
\begin{tabular}{|l||cc|cc||cccc|}\hline
$k^{-1}$ &  $ \E_N^{\TT,0}$ & $\text{EOC}$ & $\E_N^{\TT,1}$ & $\text{EOC}$ &  $\mathrm{Eref}$ & $\E^{\text{total}}_N$ & $ei$ &
\\ \hline\hline
$80$  &$5.7695$e$-03$ &-- &$2.0831$e$-01$ &-- & $1.0412$e$-01$&$1.0596$ &$10.1767$ &   \\
$160$  & $1.3258$e$-03$&$2.1216$ & $5.6917$e$-02$&$1.8718$ &$4.3944$e$-02$  & $2.6041$e$-01$&$5.9259$ &  \\
$320$  &$3.2430$e$-04$  &$2.0314$ &$1.4648$e$-02$ &$1.9581$ &$1.3199$e$-02$ &$6.8543$e$-02$ &$5.1930$ &  \\
$640$  &$8.0510$e$-05$ &$2.0101$ &$3.6910$e$ -03$&$1.9886$  &$3.5667$e$-03$&$1.6784$e$-02$ &$4.7057$ & \\
$1280$ &$2.0093$e$ -05$ & $2.0025$&$9.2463$e$-04$ &$1.9971$  &$9.2127$e$-04$ &$4.1723$ e$-03$&$4.5289$ &  \\ 
$2560$  &$5.0210$e$-06$ &$2.0006$ &$2.3128$e$-04$ &$1.9992$  &$2.3004$e$-04$ &$1.0513$e$-03$ &$4.5701$ &\\\hline
\end{tabular}\\[2ex]
\caption{Time estimators $\E_N^{\TT,j},\, j=0,1,$ and $\text{EOC}$, total estimator $\E^{\text{total}}_N$, reference error $\mathrm{Eref}$, and effectivity index $ei$  for Experiment 1b.} \label{tbl4}
\end{center}
\end{table}
%
%The dominant estimator for the harmonic oscillator is $\E_N^{\Ss,3}$, while the effectivity index tends asymptotically to the constant value $4.5$.

%Finally, for the double-well potential (1c), we discretize in space by cubic B-splines. The results are listed  in Tables~\ref{tbl5}, \ref{tbl6}.
%
\begin{table}[htb!]
\begin{center}
\renewcommand{\arraystretch}{1.0}
\begin{tabular}{|l||cc|cc|cc|ccc|}\hline
$M$ &  $ \E_N^{\Ss,0}$ & $\text{EOC}$&$\E_N^{\Ss,1}$ & $\text{EOC}$& $\E_N^{\Ss,2}$& $\text{EOC}$&$\E_N^{\Ss,3}$ &$\text{EOC}$&
\\ \hline\hline
$35$  &$2.2902$e$-02$ &-- &$3.0041$e$-02$ &-- &$3.7853$e$-01$ & --& $3.0978$e$-01$ & --&\\
$50$  &$5.1709$e$-03$ &$ 4.1724$&$3.5161$e$-03$ &$6.0145$ &$8.7250$e$-02$ &$4.1144$ &$7.1970$e$-02$ &$4.0923$ & \\
$70$  &$1.2925$e$-03$ &$4.1206$ &$4.3823$e$-04$ &$6.1888$ &$2.2013$e$-02$ &$4.0929$ &$1.8017$e$-02$ &$4.1160$&\\
$100$  &$3.0294$e$-04$&$4.0676$ &$5.0898$e$-05$ &$6.0361$ &$5.1836$e$-03$ &$4.0545$ &$4.2076$e$-03$ &$4.0777$ & \\
$145$  &$6.7657$e$-05$&$4.0345$  &$5.6498$e$-06$ &$5.9161$  & $1.1609$e$-03$&$4.0270$ &$9.3722$e$-04$ &$4.0416$ &\\
$200$  &$1.8587$e$-05$&$4.0176$ &$7.7427$e$-07$ &$6.1802$ &$3.1941$e$-04$ &$4.0129$ & $2.5721$e$-04$&$4.0208$ & \\  
\hline
\end{tabular}\\[2ex]
\caption{Space estimators $\E_N^{\Ss,j},\, 0\le j\le 3,$ and $\text{EOC}$ for Experiment 1c.} \label{tbl5}
\end{center}
\end{table}
\begin{table}[htb!]
\begin{center}
\renewcommand{\arraystretch}{1.0}
\begin{tabular}{|l||cc|cc||cccc|}\hline
$k^{-1}$ &  $ \E_N^{\TT,0}$ & $\text{EOC}$ & $\E_N^{\TT,1}$ & $\text{EOC}$ &  $\mathrm{Eref}$ & $\E^{\text{total}}_N$ & $ei$ &
\\ \hline\hline
$80$  &$1.2555$e$-03$ &-- &$1.6010$e$-02$ &-- &  $1.2414$e$-02$&$7.8510$e$-01$ &$63.2431$ &   \\
$160$  &$2.5490$e$-04$ &$2.3003$ &$3.6767$e$-03$ &$2.1225$ &$3.4441$e$-03$ & $1.7773$e$-01$&$51.6042$ &  \\
$320$  &$6.0463$e$-05$ &$2.0758$ &$9.0229$e$-04$ & $2.0267$& $9.7835$e$-04$&$4.4277$e$-02$ &$45.2568$ &  \\
$640$  &$ 1.4911$e$-05$ &$2.0197$ &$2.2452$e$-04$ &$2.0067$ &$2.1739$e$-04$ &$1.0409$e$-02$ & $47.8817$& \\
$1280$ &$3.7157$e$-06$ &$2.0047$&$5.6068$e$-05$&$2.0016$ &$5.5890$e$-05$ &$2.3622$e$-03$ &$42.2652$ &  \\ 
$2560$  &$9.2832$e$-07$ &$2.0009$ &$1.4014$e$-05$ &$2.0003$ &$1.3946$e$-05$ &$6.6728$e$-04$ &$47.8474$ &\\\hline
\end{tabular}\\[2ex]
\caption{Time estimators $\E_N^{\TT,j},\, j=0,1,$ and $\text{EOC}$, total estimator $\E^{\text{total}}_N$, reference error $\mathrm{Eref}$, and effectivity index $ei$  for Experiment 1c.} \label{tbl6}
\end{center}
\end{table}

Finally, for the double-well potential (1c), we discretize in space by cubic B-splines. The results are listed  in Tables~\ref{tbl5}, \ref{tbl6}. For this example, the effectivity index seems to be asymptotically constant (around $47.8$), but it is certainly larger compared to the previous two examples. This is maybe an indicator that the presented analysis can be improved, in order to end-up with better effectivity indices. Effectivity indices of this size were also observed in experiments for the two-dimensional heat equation, for backward Euler finite element schemes (\cite{LM}) and for the corresponding to \eqref{CNGS} method (\cite{BKM2}). 

\textsc{{Experiment 2}} (A time-dependent potential). In the second experiment we consider the time-dependent potential $V(x,t)=(1+t)^2\frac{x^2}{2}.$ Such potentials were studied for example in \cite{Leach, CGMPS}. In order to have an example where we can evaluate the exact error, instead of solving numerically problem \eqref{1dsemiclassical}--\eqref{incond} with zero Dirichlet boundary conditions, we replace \eqref{1dsemiclassical} by
\begin{equation}
\label{1deq2}
\p_t u-\frac{\mathrm{i}}{2}\p_{xx}u+\mathrm{i}V(x,t)u=f(x,t)
\end{equation}
(for this experiment, $\ep=1$). We consider as exact solution $u(x,t)=\mathrm{e}^{-25(x-t)^2}\mathrm{e}^{\mathrm{i}(1+t)(1+x)}$ and we calculate $f$ through \eqref{1deq2}.

We take again $[a,b]\times[0,T]=[-2,2]\times[0,1]$ and we  perform the same computations as in Experiment 1. In space, we discretize by quadratic B-splines. The numerical results are reported in Tables~\ref{tbl7}, \ref{tbl8}. 
 %splines degree$=2$, $h= k^{2/3}$ $u(x,t)=\mathrm{e}^{-25(x-t)^2}\mathrm{e}^{{\mathrm{i}}(1+t)(1+x)}$ $\ep=1$ $V(x,t)=(1+t)^2\displaystyle\frac{x^2}{2}$
%
\begin{table}[htb!]
\begin{center}
\renewcommand{\arraystretch}{1.0}
\begin{tabular}{|l||cc|cc|cc|ccc|}\hline
$M$ &  $ \E_N^{\Ss,0}$ & $\text{EOC}$&$\E_N^{\Ss,1}$ &$\text{EOC}$& $\E_N^{\Ss,2}$& $\text{EOC}$&$\E_N^{\Ss,3}$ &$\text{EOC}$&
\\ \hline\hline
$75$  &$1.3090$e$-02$  &-- &$7.2318$e$-03$  &-- &$2.8864$e$-02$  & --&$1.4480$e$-01$ & --& \\
$120$  &$3.1864$e$-03$ &$3.0063$ &$7.9989$e$-04$ &$4.6846$ &$7.0052$e$-03$ &$3.0126$ &$3.5234$e$-02$ &$3.0071$ & \\
$185$  &$8.6868$e$-04$ &$3.0025$ &$1.0649$e$-04$ &$4.6583$ &$1.9068$e$-03$ &$3.0061$ &$9.6015$e$-03$ &$3.0035$ &   \\
$295$ &$2.1412$e$-04$ &$3.0012$ &$1.3072$e$-05$ &$4.5036$ &$4.6970$e$-04$ &$3.0026$ &$2.3665$e$-03$ & $3.0069$&\\
$470$  &$5.2935$e$-05$ &$3.0004$ &$1.6152$e$-06$ &$4.4895$ &$1.1608$e$-04$ &$3.0012$ &$5.8502$e$-04$  &$3.0005$ &\\
$750$  & $1.3026$e$-05$&$3.0002$ &$1.9878$e$-07$ &$4.4826$ &$2.8560$e$-05$ &$3.0005$ &$1.4396$e$-04$ &$3.0002$ &\\
$1190$ &$3.2610$e$-06$ &$3.0000$ &$2.4920$e$-08$ &$4.4982$ &$7.1492$e$-06$ &$3.0002$ &$3.6039$e$-05$ &$3.0000$ &\\ 
$1885$ &$8.2045$e$-07$ &$3.0000$ &$3.7518$e$-09$ &$4.1164$ &$1.7986$e$-06$ &$3.0001$ &$9.0671$e$-06$ &$3.0001$ &\\ \hline
\end{tabular}\\[2ex]
\caption{Space estimators $\E_N^{\Ss,j},\, 0\le j\le 3,$ and $\text{EOC}$ for Experiment 2.} \label{tbl7}
\end{center}
\end{table}
\begin{table}[ht]
%\begin{center}
\centering
\renewcommand{\arraystretch}{1.0}
\begin{tabular}{|l||cc|cc||cccc|}\hline
$k^{-1}$ &  $ \E_N^{\TT,0}$ & $\text{EOC}$ & $\E_N^{\TT,1}$ & $\text{EOC}$ &  $\mathrm{Eex}$ & $\E^{\text{total}}_N$ & $ei$ &
\\ \hline\hline
$80$  &$6.9252$e$-04$ &-- &$3.3241$e$-02$ &-- &$6.6552$e$-04$ &$2.6595$e$-01$ &$399.612$ & \\
$160$  &$1.6443$e$-04$ &$2.0744$ &$7.7801$e$-03$ &$2.0951$  &$1.6474$e$-04$  &$6.2511$e$-02$ &$379.4525$ & \\
$320$  &$4.0878$e$-05$ &$2.0081$ &$1.9170$e$-03$ & $2.0209$ &$4.1787$e$-05$  &$1.6624$e$-02$ &$397.8271$ &   \\
$640$  &$1.0204$e$-05$ &$2.0022$  &$4.7815$e$-04$ & $2.0033$ &$1.0658$e$-05$  &$4.0799$e$ -03$&$382.8016$ & \\
$1280$  & $2.5513$e$-06$&$1.9998$ & $1.1953$e$-04$&$2.0000$  &$2.6883$ e$-06$  &$1.0073$e $-03$ & $374.6978$& \\
$2560$ &$6.3801$e$-07$ &$1.9996$ &$2.9888$e$-05$ &$1.9997$  &$ 6.7413$e$-07$  &$2.4807$e$-04$ &$367.9854$ & \\ 
$5120$  &$1.5988$e$-07$ &$1.9968$ &$7.4898$e$-06$ &$1.9966$  &$1.6935$e$-07$  &$6.2075$e$-05$ &$366.5486$ & \\ 
$10240$  &$3.9973$e$-08$ &$2.0000$ &$1.8728$e$-06$ & $1.9997$ &$4.2433$e$-08$  &$1.5602$e$-05$ &$367.6855$ & \\ \hline
\end{tabular}\\[2ex]
\caption{Time estimators $\E_N^{\TT,j},\, j=0,1,$ and $\text{EOC}$, total estimator $\E^{\text{total}}_N$, exact error $\mathrm{Eex}$,  and effectivity index  $ei$ for Experiment 2.} \label{tbl8}
%\end{center}
\end{table}
The correct order of convergence is observed for the estimators. The effectivity index tends asymptotically to a constant value, which is around $368$, which is a strong indication that there maybe room for improvement of the analysis. We point out though, that no a posteriori error bounds of optimal order exist in the literature for time-dependent potentials and any numerical method. It is the first time that a complete a posteriori error analysis is provided and numerically verified for operators of the form $\mathrm{i}(-\D+V(x,t)).$

%---------------------------------------------------------------------------------
\subsection{$\ep-$sensitivity of the estimators}\label{sensitivity}
%---------------------------------------------------------------------------------
%
%\begin{equation}
%\label{Apriori}
%\begin{aligned}
%\max_{0\le n\le N}||u^n&-U^n||\le C\bigg\{\Big[||u_0||_r+\big(1+\frac{T}{\ep}\sup_{x\in\Omega}V(x)\big)\max_{0\le t\le T}||u(t)||_r+\max_{0\le t\le T}||\p_tu(t)||_r\Big]h^r\\&
%+T\Big[\max_{0\le t\le T}||\p_t^3u(t)||+\ep\max_{0\le t\le T}
%||\p_x^2\p^2_t u(t)||+\frac{1}{\ep}\sup_{x\in\Omega}V(x)\max_{0\le t\le T}||\p_t^2u(t)||\Big]k^2\bigg\},
%\end{aligned}
%\end{equation}
%%
In the case of WKB initial data for the problem  \eqref{1dsemiclassical}-\eqref{incond} one can show that 
\begin{equation*}
%\label{etsderivatives}
\sup_{0\le t\le T}\|\frac{\p^mu}{\p t^m} (t)\|=\mathcal{O}(\frac{1}{\ep^m}) \ \text{ and }	\ \sup_{0\le t\le T}\|\frac{\p^mu}{\p x^m} (t)\|=\mathcal{O}(\frac{1}{\ep^m}),\quad m\in\mathbb{N}_0,
\end{equation*}
provided $n_0$, $S_0$ and  $V$ are regular enough; \cite{BJM}. %that $\max_{0\le n\le N}||u^n-U^n||=\mathcal{O}(\frac{h^r}{\ep^{r+1}}+\frac{k^2}{\ep^3})$
In that respect, and assuming that  $U^n,\, 0\le n\le N,$ are reasonably good approximations to $u$ at the nodes $t_n$, we expect the following behavior of the  a posteriori error estimators with respect to the parameter $\ep$:
%
%\begin{equation}
%\label{econdh}
%\E_N^{\Ss,0}=\mathcal{O}(\frac{h^{r+1}}{\ep^{r+1}}),\quad \E_N^{\Ss,1}=\mathcal{O}(\frac{h^{r+1}}{\ep^{r+2}}\frac{k}{\ep}), \quad \E_N^{\Ss,2}=\mathcal{O}(\frac{h^{r+1}}{\ep^{r+2}}),\quad \E_N^{S,3}=\mathcal{O}(\frac{h^{r+1}}{\ep^{r+2}}),
%\end{equation}
%%
%and
%%
%\begin{equation}
%\label{econdk}
%\E_N^{\TT,0}=\mathcal{O}\Big(\frac{k^2}{\ep^2}(1+\frac{h^{r+1}}{\ep^{r+1}})\Big), \quad \E_N^{\TT,1}=\mathcal{O}\Big(\frac{k^2}{\ep^3}(1+\frac{h^{r+1}}{\ep^{r+2}})\Big).
%\end{equation}
%
\begin{align}
\E_N^{\Ss,0}=\mathcal{O}(\frac{h^{r+1}}{\ep^{r+1}}),\quad \E_N^{\Ss,1}=\mathcal{O}(\frac{h^{r+1}}{\ep^{r+2}}\frac{k}{\ep}), & \quad \E_N^{\Ss,2}=\mathcal{O}(\frac{h^{r+1}}{\ep^{r+2}}),\quad \E_N^{S,3}=\mathcal{O}(\frac{h^{r+1}}{\ep^{r+2}}), \label{econdh} \\
 \E_N^{\TT,0}=\mathcal{O}\Big(\frac{k^2}{\ep^2}(1+\frac{h^{r+1}}{\ep^{r+1}})\Big),&  \quad \E_N^{\TT,1}=\mathcal{O}\Big(\frac{k^2}{\ep^3}(1+\frac{h^{r+1}}{\ep^{r+2}})\Big). \label{econdk}
\end{align}
Relations \eqref{econdh}--\eqref{econdk} give us an idea on how we have to choose the time and space steps so that the estimators converge. The suggested choice seems to be restrictive; however it is the expected one. Indeed the a priori error analysis for CNFE schemes gives that
\begin{equation}
\label{exact}
\max_{0\le n\le N}\|u(t_n)-U^n\|=\mathcal{O}(\frac{h^{r+1}}{\ep^{r+2}}+\frac{k^2}{\ep^3}),
\end{equation}
cf.\ \cite{BJM}, and naturally, conditions \eqref{econdh}-\eqref{econdk} were not expected to be more relaxed.  Next, we verify numerically \eqref{econdh}--\eqref{econdk}. To this end, we consider $\sqrt{n_0(x)}=\mathrm{e}^{-25(x-0.5)^2}$, $S_0(x)=-\frac{1}{5}\ln\Big(\mathrm{e}^{5(x-0.5)}+\mathrm{e}^{-5(x-0.5)}\Big),$ and the constant  potential $V(x)=10.$ We solve numerically problem \eqref{1dsemiclassical}--\eqref{incond} in $(a,b)\times(0,T]=(-1,2)\times(0,0.54],$ for $\ep=0.005$ and $\ep=0.001$, using B-splines of degree $1$ or $3$. Since the potential is taken to be constant, estimator $\E_N^{\Ss,2}$ is identically zero. The particular  example has been considered earlier in \cite{BJM} (see also \cite{MPPS}) and it is interesting because caustics are formed before the final time $T=0.54$. 

First, we consider the case $\ep=0.005$. We discretize by B-splines of degree $1$ and we consider uniform partitions in both time and space with $k=h$. The behavior of the space and time a posteriori error estimators are reported in Table~\ref{tbl9}.
\begin{table}[htb!]
%\begin{center}
\centering
\renewcommand{\arraystretch}{1.0}
\begin{tabular}{|l||ccc|ccc|}\hline
$k=h$ &  $ \E_N^{\Ss,0}$ &$\E_N^{\Ss,1}$ &$\E_N^{\Ss,3}$ &$\E_N^{\TT,0}$&$\E_N^{\TT,1}$& 
\\ \hline\hline
$10^{-2}$  &$5.9846$e$-01$ &$3.3162$e$+02$ &$6.4332$e$+01$ &$5.5855$ &$1.8291$e$+03$ & \\
$10^{-3}$ &$5.1220$e$-03$ &$2.0283$ &$3.9587$ &$1.8122$e$-01$ &$1.3029$e$+02$ &  \\
$5\times 10^{-4}$  &$1.2789$e$-03$&$3.2267$e$-01$ &$1.2595$ &$5.7015$e$-02$ &$4.1315$e$+01$ & \\
$10^{-4}$  &$5.1137$e$-05$ &$2.8842$e$-03$ &$5.6296$e$-02$ &$2.5351$e$-03$ &$1.8417$ & \\
$5\times 10^{-5}$  &$1.2784$e$-05$  &$3.6194$e$-04$ &$1.4128$e$-02$ &$6.3615$e$-04$ &$4.6218$e$-01$ & \\
$10^{-5}$&$5.1136$e$-07$ &$3.2460$e$-06$ &$5.6582$e$-04$ &$2.5476$e$-05$ &$1.8510$e$-02$ &\\ \hline
\end{tabular}\\[2ex]
\caption{Space estimators $\E_N^{\Ss,j},\, j=0,1,3,$ and time  estimators $\E_N^{\TT,j},\, j=0,1,$  for $\ep=0.005$.} \label{tbl9}
%\end{center}
\end{table}
As %it is expected from 
\eqref{econdh} suggests, estimator $\E_N^{\Ss,1}$ has the expected behavior for $k=h\le 5\times 10^{-4}$, while $\E_N^{\Ss,3}$ for $h\le 10^{-4}.$ Similar results, verifying \eqref{econdk}, are observed for the time estimators $\E_N^{\TT,0}$ and $\E_N^{\TT,1}.$ In particular, note that for $k\ge 10^{-4},$ $\E_N^{\TT,1}$ is not reasonable, %something completely expected, 
something we expect, 
provided that \eqref{econdk} is true and $\ep=0.005$. Note however, that estimator $\E_N^{\Ss,0}$ behaves better than expected, since for $h=10^{-2}$ it already decays with optimal order.

Next, we consider the case $\ep=0.001$. We discretize in space by cubic B-splines. To verify numerically \eqref{econdh}, we take constant time step,  $k=5\times 10^{-3}$  so that $\frac{k}{\ep}=\mathcal{O}(1)$, and thus be able to see only the effect of the space discretization with respect to $\ep$ for $\E_N^{\Ss,1}.$ As before, in Table~\ref{tbl10}, the stated relation \eqref{econdh} between $h$ and $\ep$ is observed for $\E_N^{\Ss,1}$ and $\E_N^{\Ss,3}$. We also verify the corresponding relation between $h$ and $\ep$ in \eqref{econdh} for $\E_N^{\Ss,0}$. Indeed, despite the fact that $\E_N^{\Ss,0}$ is small, even for $M=600$ ($h=5\times 10^{-3}$),  it does not decay with optimal order. The correct behavior is initiated for  $M=1500$ ($h=2\times 10^{-3}$), and verified for $M=3000$ ($h=10^{-3}$). For the time estimators, \eqref{econdk} is verified with constant mesh size $h=5\times 10^{-4}$ ($M=6000$). Our choice of $h$ is so that $\frac{h^4}{\ep^5}$ is controlled, and allow us to exploit the behavior of $k$ with respect to $\ep$. Our findings are shown in Table~\ref{tbl10}.
%
%%
%\newpage
%
%\begin{minipage}{1.0\textwidth}
\begin{table}[!htb]
%\begin{center}
\centering
\renewcommand{\arraystretch}{1.0}
\begin{tabular}{|l||ccc|||c||ccc|}\hline
$M$ &  $ \E_N^{\Ss,0}$ &$ \E_N^{\Ss,1}$  & $ \E_N^{\Ss,3}$ &  $k$ & $ \E_N^{\TT,0}$ & $ \E_N^{\TT,1}$ &
\\ \hline\hline
$600$  &$3.6649$e$-01$ &$2.4658$e$+02$ &$1.9454$e$+03$ &$10^{-3}$ &$1.2389$ &$4.4981$e$+03$ &   \\
$1500$  &$4.6616$e$-02$ &$3.2401$e$+01$ & $2.5136$e$+02$&$5\times 10^{-4}$ &$5.8702$e$-01$ &$2.1314$e$+03$ &  \\
$3000$  &$2.6006$e$-03$ &$1.8065$ &$1.4019$e$+01$ &$10^{-4}$ &$5.6881$e$-02$ & $2.0655$e$+02$&  \\
$4500$  &$4.9896$e$-04$ &$3.4653$e$-01$ &$2.6894$ &$5\times10^{-5}$ &$1.5444$e$-02$ &$5.6083$e$+01$ &  \\
$6000$ &$1.5618$e$-04$ &$1.0846$e$-01$ &$8.4179$e$-01$ &$10^{-5}$ &$6.3630$e$-04$ &$2.3107$ &  \\ 
$7500$& $6.3646$e$-05$&$4.4198$e$-02$ & $3.4304$e$-01$&$5\times 10^{-6}$ &$1.5923$e$-04$ &$5.7822$e$-01$ &\\
$9000$  &$3.0608$e$-05$ &$2.1254$e$-02$ &$1.6497$e$-01$ &$2.5\times 10^{-6}$ &$3.9816$e$-05$ &$1.4459$e$-01$ &\\
\hline
\end{tabular}\\[2ex]
\caption{Space estimators $\E_N^{\Ss,j},\, j=0,1,3,$ with $k=5\times 10^{-5},$ and time estimators  $\E_N^{\TT,j},\, j=0,1,$ with $M=6000$, for $\ep=0.001$.} \label{tbl10}
%\end{center}
\end{table}
%\end{minipage}
%
%\setlength{\textfloatsep}{-3.95cm}

%
%\vfill
%

%
%-----------------------------------------------------------------------------------
\section{Numerical Experiments: Adaptivity}\label{adapt}
%-----------------------------------------------------------------------------------
%
In this section, we adjust and further develop a time-space adaptive algorithm for linear Schr\"odinger equations, using the a posteriori error estimators derived earlier. Our goal is to study  numerically the behavior of the estimators under this adaptive algorithm, and investigate the benefits, in terms of computational cost and accuracy, of time-space adaptivity.

To this end, we consider, as in the previous section,  the one-dimensional linear Schr\"odinger  equation in the semiclassical regime, cf.\ \eqref{1dsemiclassical}, along with the WKB initial condition \eqref{incond}. The presented numerical experiments indicate that adaptivity through the a posteriori error bounds is indeed advantageous, especially for relatively small values of the Planck constant $\ep$, for both time-independent and time-dependent potentials.  
Furthermore,  by appropriately  modifying  the adaptive algorithm we are able to construct efficient approximations, not only to the exact solution $u$, but also to observables \eqref{posden} and \eqref{curden} of problem \eqref{1dsemiclassical}-\eqref{incond}.
%
%Observables, such as the \emph{position density},
%%
%\begin{equation}
%\label{posden}
%|u(x,t)|^2,
%\end{equation}
%%
%and the \emph{current density},
%%
%\begin{equation}
%\label{curden}
%J(x,t):= \Ima\big(\overline {u(x,t)}\nabla u(x,t)\big),
%\end{equation}
%%
%are physical quantities defined through the wave function $u$. Such quantities are important from applications point of view, 
As already mentioned, for small values of $\ep$ it is very difficult to approximate correctly \eqref{posden} and \eqref{curden}, unless very fine mesh sizes are used.  The problem becomes harder in cases where caustics develop. This is a hard and delicate issue and adaptivity can play an important role to resolve it. 

%-----------------------------------------------------------------------------------
\subsection{The adaptive algorithm}
%-----------------------------------------------------------------------------------
%
We consider, modify  and further develop the time-space algorithm of \cite{ALBERT}, introduced first in \cite{NSV}. 
%The algorithm is based on the maximum strategy for the adaptivity in space and on implicit strategy A for the adaptivity in time. 
%
We stress out once more that we do not claim that the particular adaptive algorithm is an optimal one. However, it appears to perform well for the problem under consideration and the estimators at hand. In that respect, it is possible to check the efficiency and robustness of the estimators.

%For linear Schr\"odinger equations, an adaptive algorithm has also been proposed by D\"orfler in \cite{D}. For parabolic problems,   a number of adaptive algorithms exists in the literature; cf., e.g., \cite{CF,ALBERT} and the references therein.
% It is to be emphasized though that in the literature exists only one proven convergent time-space adaptive algorithm 
%for evolution problems and can be found in \cite{KMSS}. This algorithm is appropriate for the heat equation and backward Euler finite element schemes and it is not clear how to generalize it to other problems and higher order in time methods.
%

We next briefly describe  the algorithm we use. To this end, we use $\Gn$ to indicate the spatial grid at $t=t_n$. We also use the notation
$\zeta_0^{\mathrm{I}}:=\|u_0-U^0\|+\eta_{\mathbb{V}^0}(U^0)$. In addition, we can write
%%
%$$\E_m^{\TT,0}:=\max_{1\le n\le m}\zeta_n^{\TT,0}, \qquad \E_m^{\Ss,0}:=\max_{0\le n\le m}\zeta_n^{\Ss,0}$$
%%
%and
%%
%$$\E_m^{\TT,1}\le\max_{1\le n\le m}\zeta_n^{\TT,1}, \quad \E_m^{\Ss,j}\le\max_{1\le n\le m}\zeta_n^{\Ss,j},\,1\le j\le 3\quad \E_m^{\mathrm{C}}\le\max_{1\le n\le m}\zeta_n^{\mathrm{C}},\quad \E_m^{\mathrm{D}}\le\max_{1\le n\le m}\zeta_n^{\mathrm{D}},$$
\begin{align*}
\E_m^{\TT,0}:=\max_{1\le n\le m}\zeta_n^{\TT,0},  & \qquad \E_m^{\Ss,0}:=\max_{0\le n\le m}\zeta_n^{\Ss,0},\\
\E_m^{\TT,1}\le\max_{1\le n\le m}\zeta_n^{\TT,1}, \quad \E_m^{\Ss,j}\le\max_{1\le n\le m}\zeta_n^{\Ss,j},\, &1\le j\le 3,\quad \E_m^{\mathrm{C}}\le\max_{1\le n\le m}\zeta_n^{\mathrm{C}},\quad \E_m^{\mathrm{D}}\le\max_{1\le n\le m}\zeta_n^{\mathrm{D}},
\end{align*}
where $\zeta_n^{\TT,j},\, 0\le j\le 1,$ $\zeta_n^{\Ss,j},\, 0\le j\le 3$, $\zeta_n^{\mathrm{C}}$ and $\zeta_n^{\mathrm{D}}$ can readily be obtained from \eqref{timerecer}, \eqref{ellipticerr}, and \eqref{timeest}--\eqref{data}. In all  computations the constant $C$, cf., \eqref{timeest}, is taken equal to $1$ and the involved local time integrals are computed using the midpoint quadrature rule. For $1\le n\le m\le N,$ we further define
$$\zeta^{\TT}_n:=\zeta_n^{\TT,0}+\zeta_n^{\TT,1}\ \text{ and }\ \zeta_n^{\Ss}:=\sum_{j=0}^3\zeta_n^{\Ss,j}+\zeta_n^{\mathrm{C}}+\zeta_n^{\mathrm{D}},$$
and let $\mathrm{tol}_{\Ss}$  and $\mathrm{tol}_{\TT}$ denote the tolerances for the local time and space  estimators $\zeta_n^{\TT}$ and $\zeta_n^{\Ss}$, respectively.   The main steps of the adaptive algorithm are summarized schematically in the pseudocode below. 
\vspace{-0.025cm}
%More precisely, the adaptive algorithm starts by advancing the solution and computing the local space and time estimators. Next, before starting the process of adapting the spatial mesh,  we perform a time-step refinement, if necessary, based on the local time estimator.  We proceed
%on the spatial adaptation part of algorithm based on the local space estimator  : we first mark the elements for refinement and/or coarsening and
%we adapt the grid appropriately, we recompute the solution and the local space and time estimators. Next we perform another timestep refinement, if necessary
%based on the local time estimator and then we loop back to the space estimator check. One step of the adaptive algorithm then concludes
%by a timestep coarsening step.
%$$\E_0^\mathrm{I}:=\|u_0-U^0\|+\eta_{\mathbb{V}^0}(U^0), \quad \E_n^{\TT}:=\E_n^{\TT,0}+\E_n^{\TT,1},\quad \E_n^{\Ss}:=\sum_{j=0}^3\E_n^{\Ss,j}+\E_n^{\mathrm{C}}+\E_n^{\mathrm{D}},$$ 
%
%
%and let $tol_{\Ss}$  and $tol_{\TT}$ denote the tolerances for the total space and time estimators $\E_n^{\TT}$ and $\E_n^{\Ss}$, respectively. 
%\IK{Describe the algorithm} We also provide an algorithmic description of the time-space algorithm.

%
\begin{algorithm}[htb!]
 \NoCaptionOfAlgo
\DontPrintSemicolon
\SetAlgoLined 
	{\bf Choose Parameters} : $\mathrm{tol}_{\Ss},\ \mathrm{tol}_{\TT}, \ \delta_{1}\in(0,1), \ \delta_{2}>1, \ \theta_{1}\in(0,1), \ \theta_{2}\in(0,\theta_{1})$\;
	{\bf Initialization}:\;
	\Indp Given an initial grid $\Gc_{0}$ compute $U^0, \zeta_0^{\mathrm{I}},\zeta^{\Ss,0}_0$ \;
		$\Gc_{0} :=$AdaptInitialGrid($U^0, \zeta_0^{\mathrm{I}},\zeta_0^{\Ss,0}$), $t=0$\;

\Indm	
\BlankLine
\While{$t<T$}{
	At $t_{n-1}$ given  $(\Gnm, k_{n-1}, U^{n-1})$ set $\Gn:=\Gnm, \ k_n:=k_{n-1}, t_n:=t_{n-1}+k_n$ \;
	Solve the discrete problem: $(\Gnm, U^{n-1}) \to (\Gn,U^n)  $\;
	Compute Estimators $\zeta_n^{\Ss}, \zeta_{n}^{\TT}$ on $\Gn$ \;
	
	\BlankLine
	
	\While{$\zeta^{\TT}_n > \theta_{1}\mathrm{tol}_{\TT}$}{
		$k_n:=\delta_{1}k_{n-1}$\;
		$t_n:=t_{n-1}+k_n$\;
		Solve the discrete problem: $(\Gnm, U^{n-1}) \to (\Gn,U^n)  $\;
		Compute Estimators $\zeta_n^{\Ss}, \zeta_n^{\TT}$ on $\Gn$  \;	
	}
	
	\BlankLine
	
	\While{$\zeta_n^{\Ss}>\mathrm{tol}_{\Ss}$}{
		Mark Elements for Refinement and/or Coarsening \;
	
		\If{elements are marked}{
			Adapt grid $\Gn$ \;
			Solve the discrete problem: $(\Gnm, U^{n-1}) \to (\Gn,U^n)  $\;
			Compute Estimators $\zeta_n^{\Ss}, \zeta_n^{\TT}$ on $\Gn$ \;
		}
		
		\BlankLine
		
		\While{$\zeta^{\TT}_n> \theta_{1}\mathrm{tol}_{\TT}$}{
			$k_n:=\delta_{1}k_{n-1}$\;
			$t_n:=t_{n-1}+k_n$\;
			Solve the discrete problem: $(\Gnm, U^{n-1}) \to (\Gn,U^n)  $\;
			Compute Estimators $\zeta_n^{\Ss}, \zeta_n^{\TT}$ on $\Gn$  \;	
		}	
	}
	
	\BlankLine
	\If{$\zeta_n^{\TT} \le \theta_{2}\mathrm{tol}_{\TT}$}{
		$k_n:=\delta_{2}k_n$
	}
	$t:=t_n$
}

\caption{\textbf{Time-Space Adaptive Algorithm}}
\end{algorithm}
%
%
%\vspace{-14cm}
%
%\newpage
%
More precisely, the adaptive algorithm starts by advancing the solution and computing the local space and time estimators. Next, before starting the process of adapting the spatial grid,  we perform a time-step refinement, if necessary, based on the local time estimator.  We proceed on the spatial adaptation part of algorithm based on the local space estimator:  we first mark the elements for refinement and/or coarsening and we adapt the grid appropriately, we recompute the solution and the local space and time estimators. Next we perform another time-step refinement, if necessary, based on the local time estimator and then we loop back to the space estimator check. One step of the adaptive algorithm then concludes by a time-step coarsening step.

Reasonable choices for the parameters $\theta_1$ and $\theta_2$ are $\theta_1=0.9$ and $\theta_2=0.2,$ while for $\delta_1$ and $\delta_2$ we take $\delta_1=0.75$ and $\delta_2=1.25$. In all of the experiments, the coarsening percentage is taken to be $10\%$. 
For the mesh refinement percentage, we take $1\%$ for the time-dependent potentials and $5\%$ for all the other cases.
In the sequel, we denote by $\tilde\E_m^{\TT}$ and $\tilde\E_m^{\Ss}$ the following global time and space estimators:
$$\tilde\E_m^{\TT}:=\E_m^{\TT,0}+\max_{1\le n\le m}\zeta_n^{\TT,1}\ \text{ and }\ \tilde\E_m^{\Ss}:=\zeta_0^{\mathrm{I}}+\E_m^{\Ss,0}+\sum_{j=1}^3\max_{1\le n\le m}\zeta_n^{\Ss,j}+\max_{0\le n\le m}\zeta_n^{\mathrm{C}}+\max_{0\le n\le m}\zeta_n^{\mathrm{D}},$$
respectively.
Finally, we define the total degrees of freedom of the adaptive algorithm at the final time $T$ as
\begin{equation*}
\text{Total DoF's}
:=\Big[\sum_{n=1}^Nk_nM_n\Big]+1,
\end{equation*}
where $\Big[\cdot\Big]$ denotes the integral part of a real number and $M_n$ denotes the degrees of freedom at time-level $t_n$.
%
%-----------------------------------------------------------------------------------
\subsection{Time-independent potentials}
%-----------------------------------------------------------------------------------
%
For the first set of the numerical experiments with adaptivity, we consider two characteristic cases of time-independent potentials: a constant  potential and a harmonic oscillator. In both cases, we consider the WKB initial data \eqref{incond} with
\begin{equation}
\label{initial2}
\sqrt{n_0(x)}=\mathrm{e}^{-\lambda^2 (x-0.5)^2},\quad  S_0(x)=-\frac{1}{\lambda}\ln\Big(\mathrm{e}^{\lambda(x-0.5)}+\mathrm{e}^{-\lambda(x-0.5)}\Big).
\end{equation}
 In particular we consider:\\
\textsc{Case 1}: $[a,b]\times[0,T]=[0,1]\times[0,0.1]$, $V(x)\equiv 10$, $\ep=10^{-4}$ and $\lambda=30$.\\
\textsc{Case 2}: $[a,b]\times[0,T]=[-1,2]\times[0,0.54]$, $V(x)=\frac{x^2}2$, $\ep=10^{-3}$ and $\lambda=5$.

For the first case, we discretize in space by B-splines of degree $4$. The particular example is interesting, because caustics are formed before the final time $T=0.1$. We first apply the time-space adaptive algorithm. As  expected, we observe adaptivity in space but we  do not observe adaptivity in time. However, in this case, we emphasize  that regardless of the initial choice of the time-step, the adaptive algorithm is able to produce the required time-step for the desirable tolerance of the error. For this example, the given initial time-step was $10^{-3}$ and adapted by the algorithm to $1.34\times 10^{-7}$, which is in agreement with \eqref{econdk} (see also \eqref{exact}). Next, we perform the same experiment, but using uniform partitions and the same degrees of freedom as in the adaptive algorithm. The estimators  are plotted in Figure~\ref{tmind1} in logarithmic scale, for both the adaptive algorithm and the uniform partition. We observe that the total estimator computed with the uniform partition is two orders of magnitude larger compared to the corresponding one using adaptivity. Since the time-step, after its initial adaptation remains fixed, the evolution of the total time estimator is the same for both the adaptive algorithm and the corresponding uniform partition. The total space estimator dominates the time estimator in the uniform partition, and this is the reason that $\tilde\E_m^{\Ss}$ coincides with the total estimator on the right plot of Figure~\ref{tmind1}.

\begin{figure}[htb!]%\label{tmind1}
%\centering
%
{
%\hspace{-1.5cm}
%\includegraphics%[scale=0.38]
%[height=6.5cm, width=9cm]{EXAMPLES/EX1/T_vs_Est_adapt1.pdf}}
%%
% {\hspace{-0.3cm}
%\includegraphics%[scale=0.38]
%[height=6.5cm, width=9cm]{EXAMPLES/EX1/T_vs_Est_unif1.pdf} \\
%
\hspace{-1.2cm}
\includegraphics%[scale=0.38]
[height=6.5cm, width=9cm]{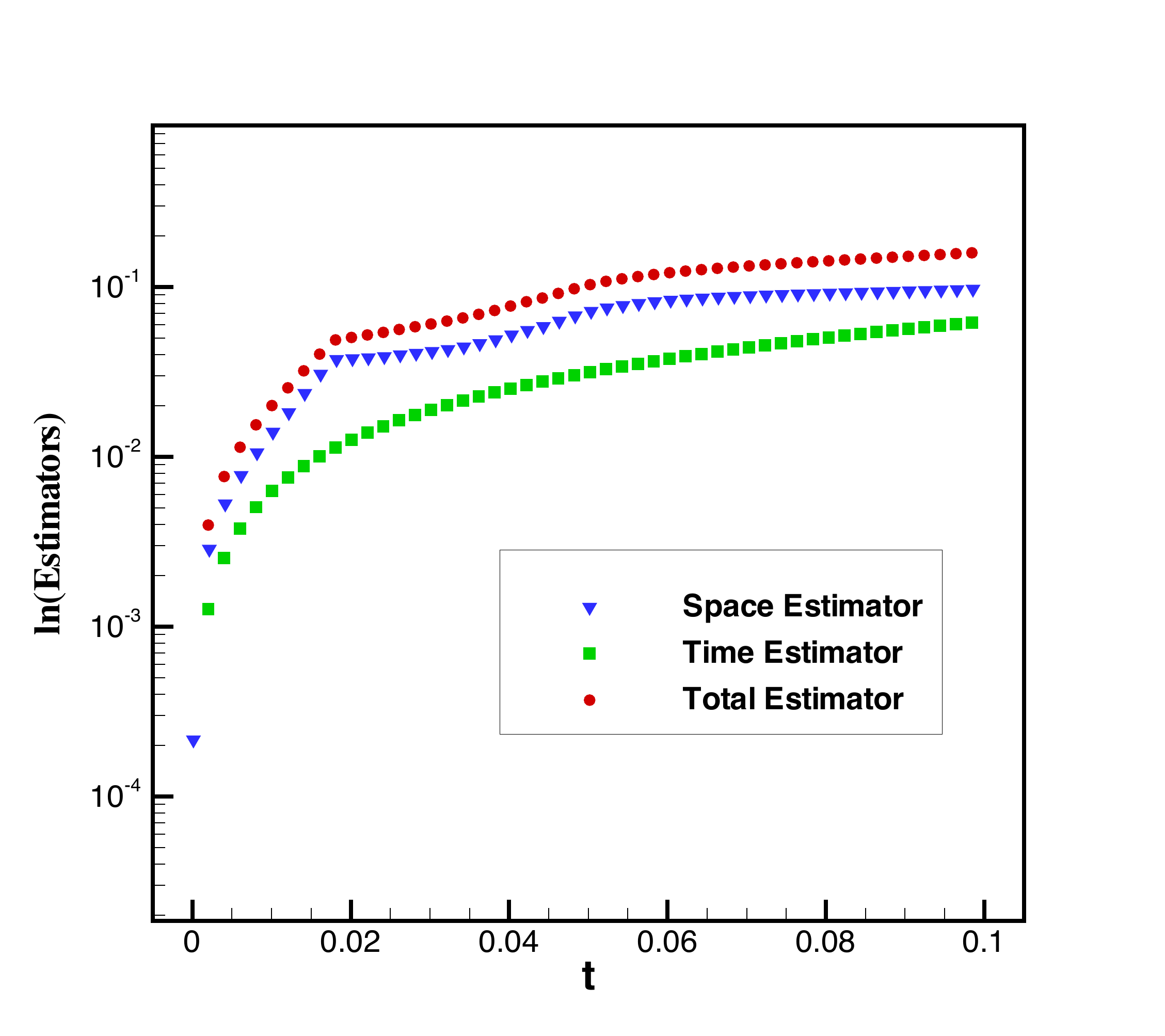}}
 {\hspace{-0.3cm}
\includegraphics%[scale=0.38]
[height=6.5cm, width=9cm]{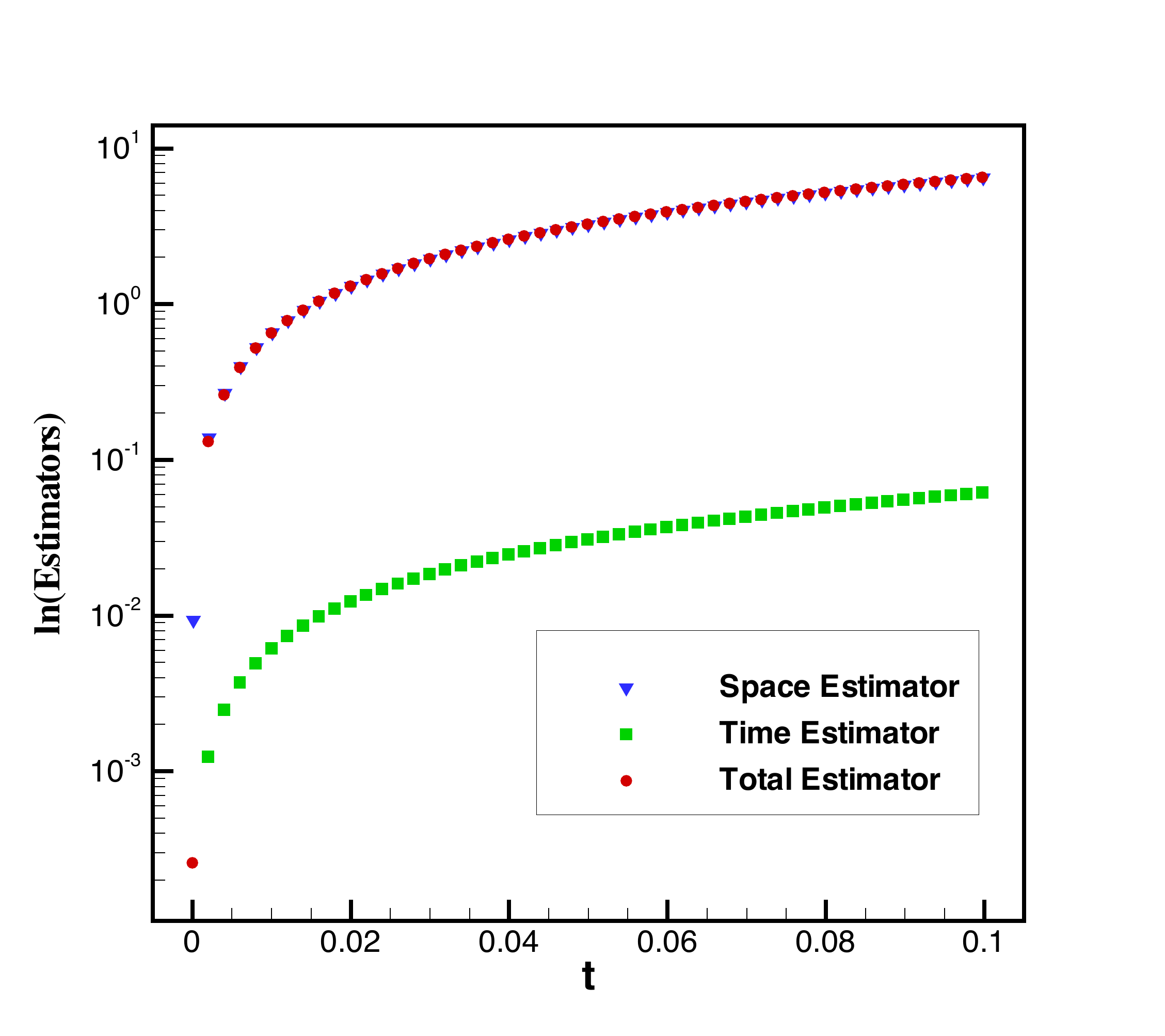}
%
%\vspace{-0.1cm}
%
%
\caption{Evolution of $\tilde\E_m^{\TT},\tilde\E_m^{\Ss}$ and total estimator  in logarithmic scale, using adaptivity (left) and uniform partitions with the same degrees of freedom (right) for the case $V(x)\equiv 10$. \label{tmind1}}
}
\end{figure}

For the second case, we discretize in space by cubic B-splines and we apply again the adaptive algorithm. As initial time-step we take again $10^{-3}$ and adapted to $7.5\times 10^{-5}$, which is larger than the expected one. This is because both \eqref{econdk}, \eqref{exact} are sufficient, but not always necessary conditions for convergence for problem \eqref{1dsemiclassical}--\eqref{incond}. The fact that the adaptive algorithm is able to compute the correct time-step size can be considered an advantage, since for the linear Schr\"odinger equation in the semiclassical regime such a choice is crucial and delicate from the  point of view of accuracy and stability of the approximations, as well as from the point of view of computational cost.  In each time-slot, the mesh size varies from $7.32\times 10^{-6}$ to $2.4\times 10^{-1}$, which proves that  conditions \eqref{econdh} can be relaxed through adaptivity in space; very fine mesh sizes are needed only in certain areas of $[-1,2]$.
 In Figure~\ref{tmind2}, we plot the evolution of  time, space and total estimators  in logarithmic scale and the position density at the beginning and at the final time $T=0.54$. As we observe from the plot of $|U|^2$  at $T=0.54,$ caustics are formed for this problem as well. The a priori knowledge of such information requires very technical and tedious  calculations. However this information can be obtained through the a posteriori error analysis and adaptivity.

%----------------------------------------------------------------------------------------------------------
\begin{figure}[htb!]%\label{tmind2}
%\centering
%
{\hspace{-1.5cm}
\includegraphics[scale=0.38]{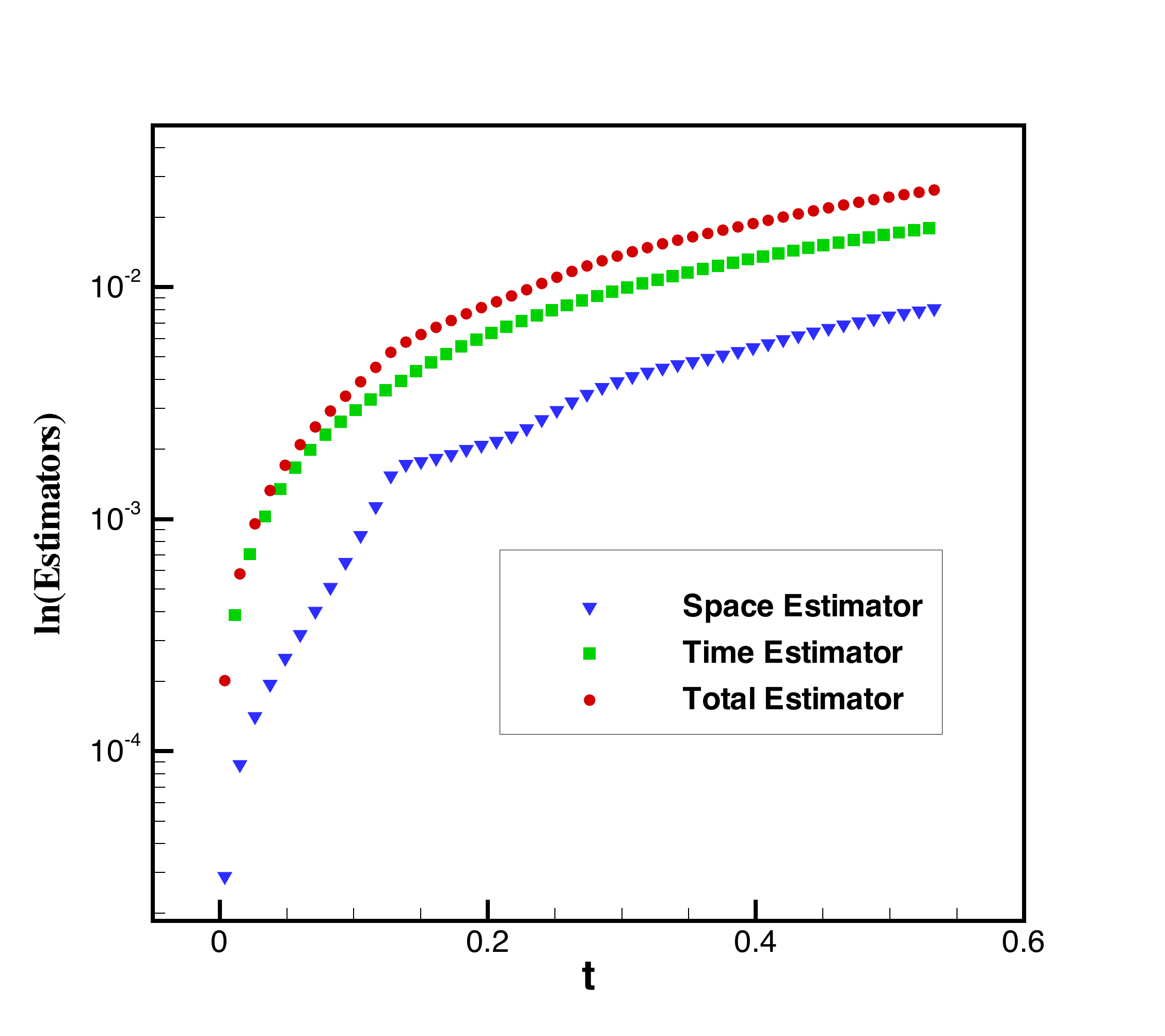}}
 {\hspace{-0.3cm}
\includegraphics[scale=0.32]{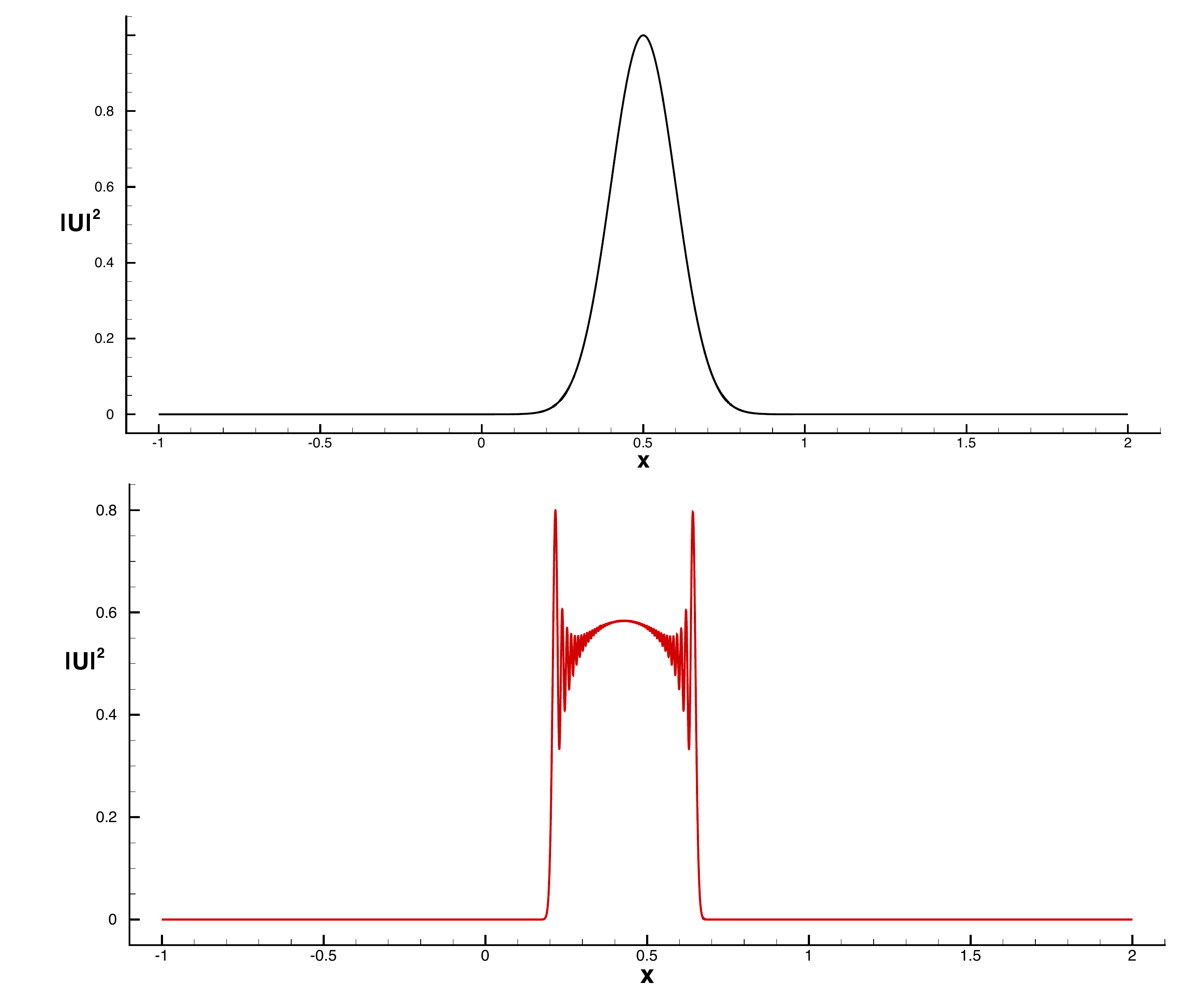} \\
%
%\vspace{-0.1cm}
%
%
\caption{Evolution of $\tilde\E_m^{\TT},\tilde\E_m^{\Ss}$ and  total estimator  in logarithmic scale, using the adaptive algorithm (left) and position density for $t_0=0$ (upper right) and at the final time $T=0.54$ (lower right) for the harmonic oscillator. \label{tmind2}}}
\end{figure}
%
%-----------------------------------------------------------------------------------
\subsection{Time-dependent potentials}
%-----------------------------------------------------------------------------------
%
The simplest time-dependent potentials are of the form $V(x,t)=\frac{x^2}2\omega(t)$, where $\omega$ denotes a smooth function in time; \cite{CGMPS, Leach}. To check the efficiency of the estimators during time adaptivity, we choose two time-dependent potentials of this form, which change relatively fast with time. 

For the first experiment, we solve in $[a,b]\times[0,T]=[1,2]\times[0,3]$ and we take $V(x,t)=\displaystyle\frac{x^2}2\cdot\frac1{10t+0.05} $ and $\ep=10^{-2}.$ As $\sqrt{n_0(x)}$ we take the one in \eqref{initial2}, while we choose $S_0(x)=5(x^2-x)$, and we define the initial condition through \eqref{incond}. We use quadratic B-splines and we apply the time-space adaptive algorithm. In Figure~\ref{tmd1}, we plot the evolution of the estimators  in a logarithmic scale, as well as the variation of the time-steps $k_n$ during time adaptivity. The considered potential changes faster with  time in the subinterval $[0,1]$, compared to $[1,3]$, and this is the reason the required time-step is considerably smaller in this area.  For this experiment, in each time-slot, the mesh size varies from $1.17\times 10^{-4}$ to $1.2\times 10^{-1}$. 
\begin{figure}[htb!] 
%\centering
%
{\hspace{-1.5cm}
\includegraphics[scale=0.38]{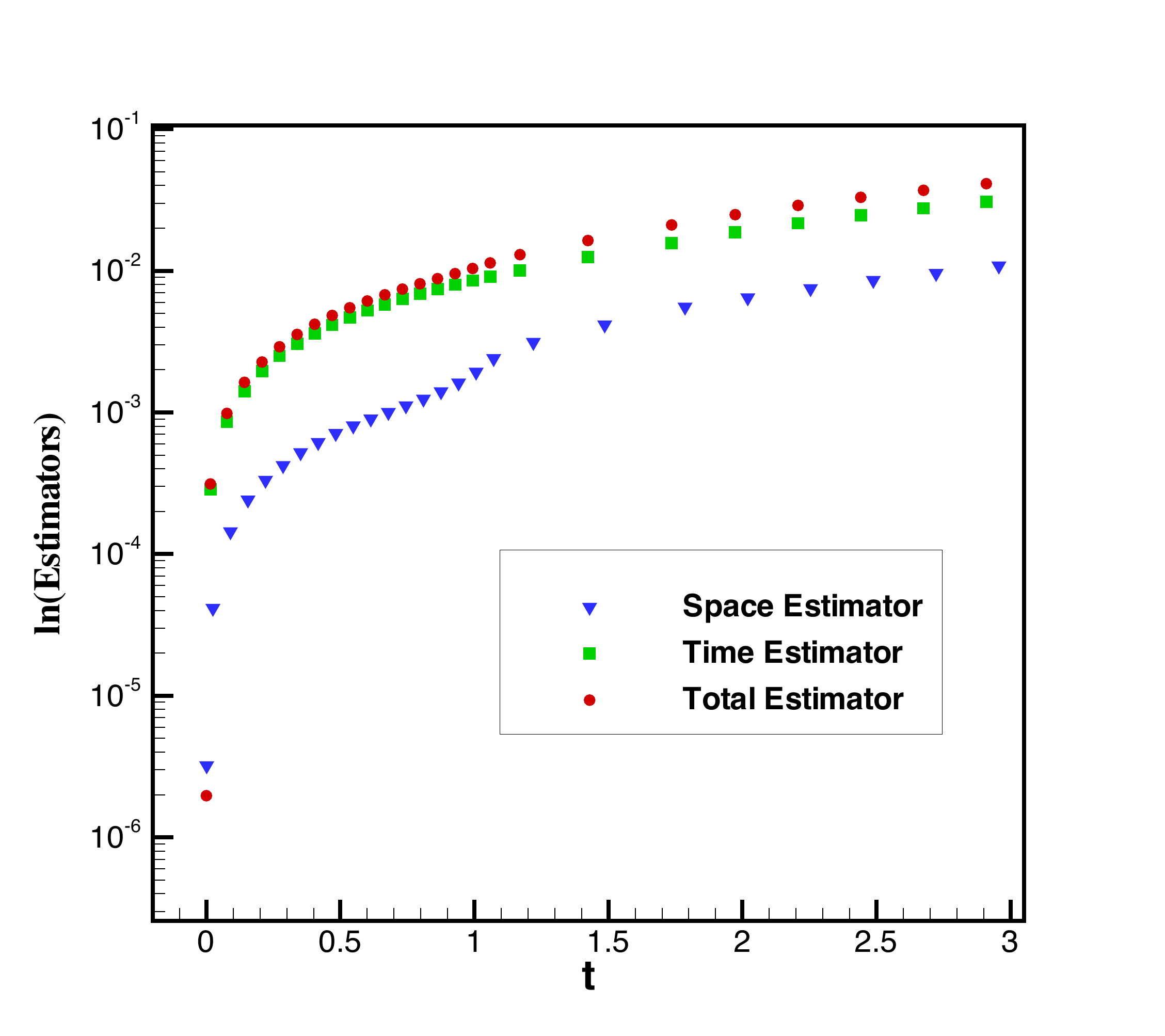}}
 {\hspace{0.3cm}
\includegraphics[scale=0.38]{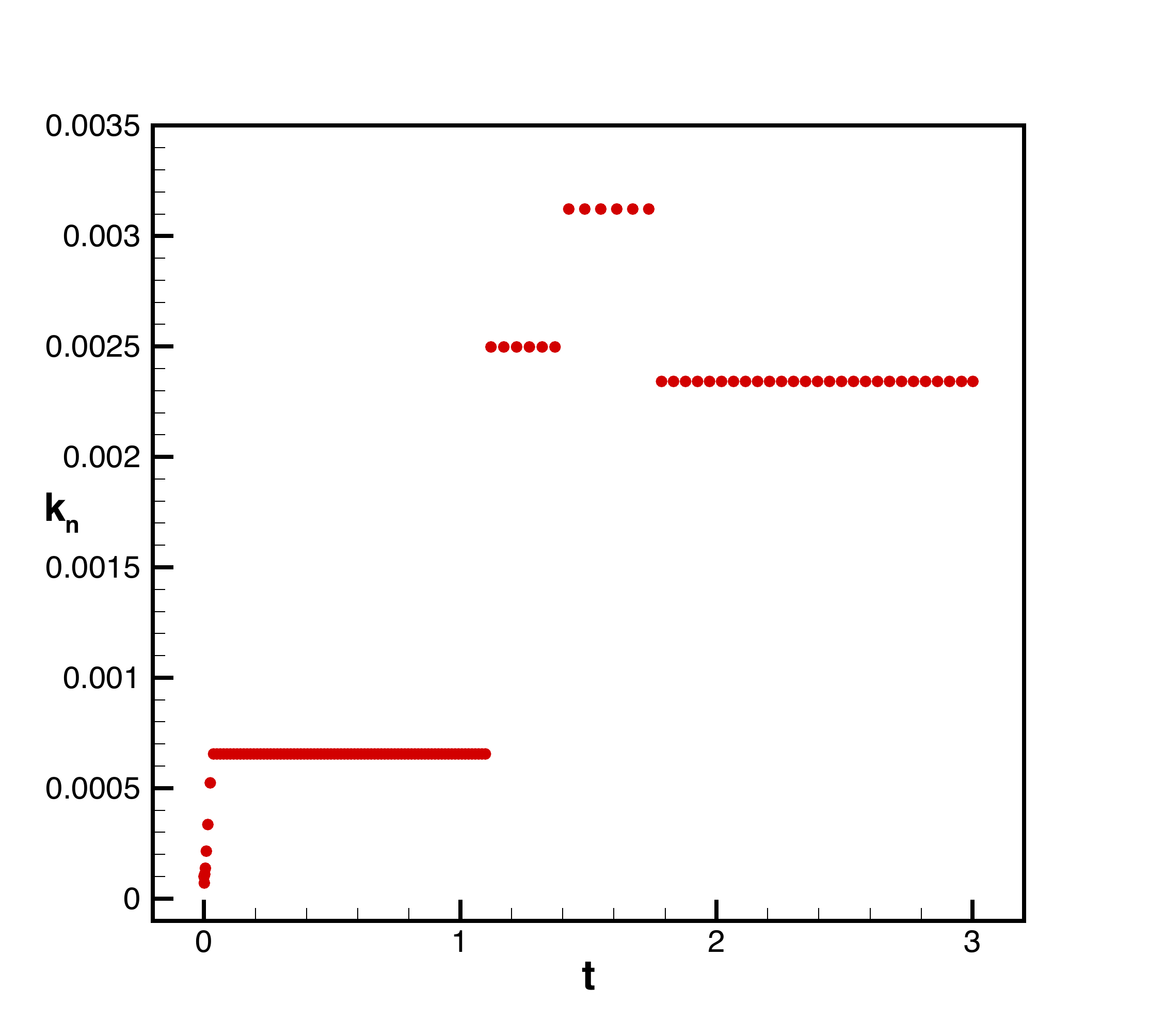} 
%
%\vspace{-0.1cm}
%
%
\caption{Evolution of $\tilde\E_m^{\TT}, \tilde\E_m^{\Ss}$ and the total estimator in logarithmic scale (left), and variation of time-steps $k_n$ during adaptivity (right), for $V(x,t)=\frac{x}2\cdot\frac{1}{10t+0.05}.$\label{tmd1}}
}
\end{figure}

For the second experiment, we solve in $[a,b]\times[0,T]=[-1,2]\times[0,1]$ and we take $V(x,t)=\displaystyle\frac {x^2}2\cdot\frac {1}{t+0.05}$ and $\ep=2.5\times 10^{-3}$. We take the same initial condition as in the previous experiment and cubic B-splines. In Figure~\ref{tmd2},  we plot the evolution of the estimators in logarithmic scale and the variation in time of the time-steps and of the degrees of freedom. This is a characteristic example where intensive adaptivity is observed, in both time and space. 
%
%------------------------------------------------------------------------------------
\begin{figure}[htb!]
%\centering
%
{\hspace{-1.5cm}
\includegraphics[scale=0.38]{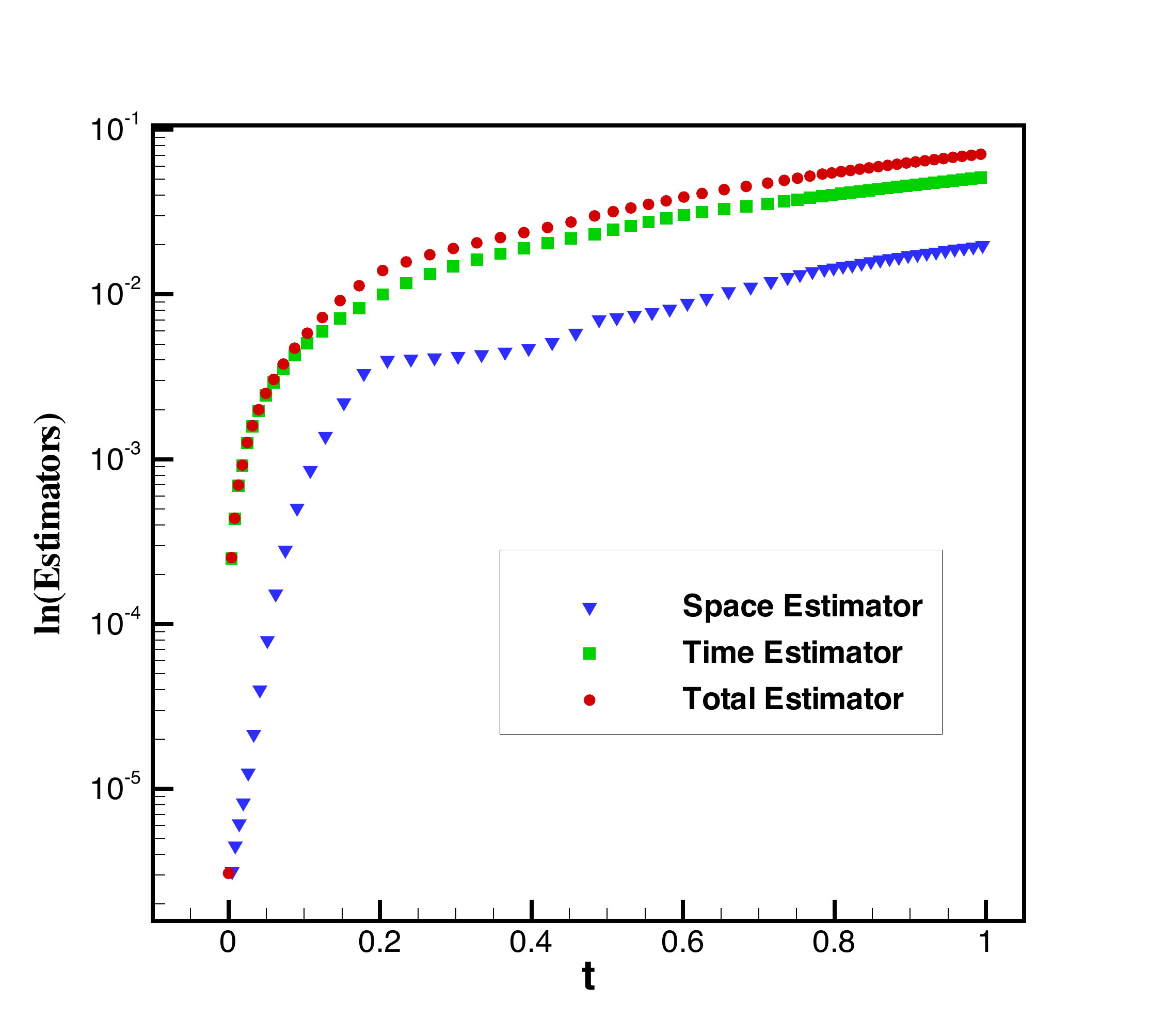}}
 {\hspace{-0.7cm}
\includegraphics[scale=0.35]{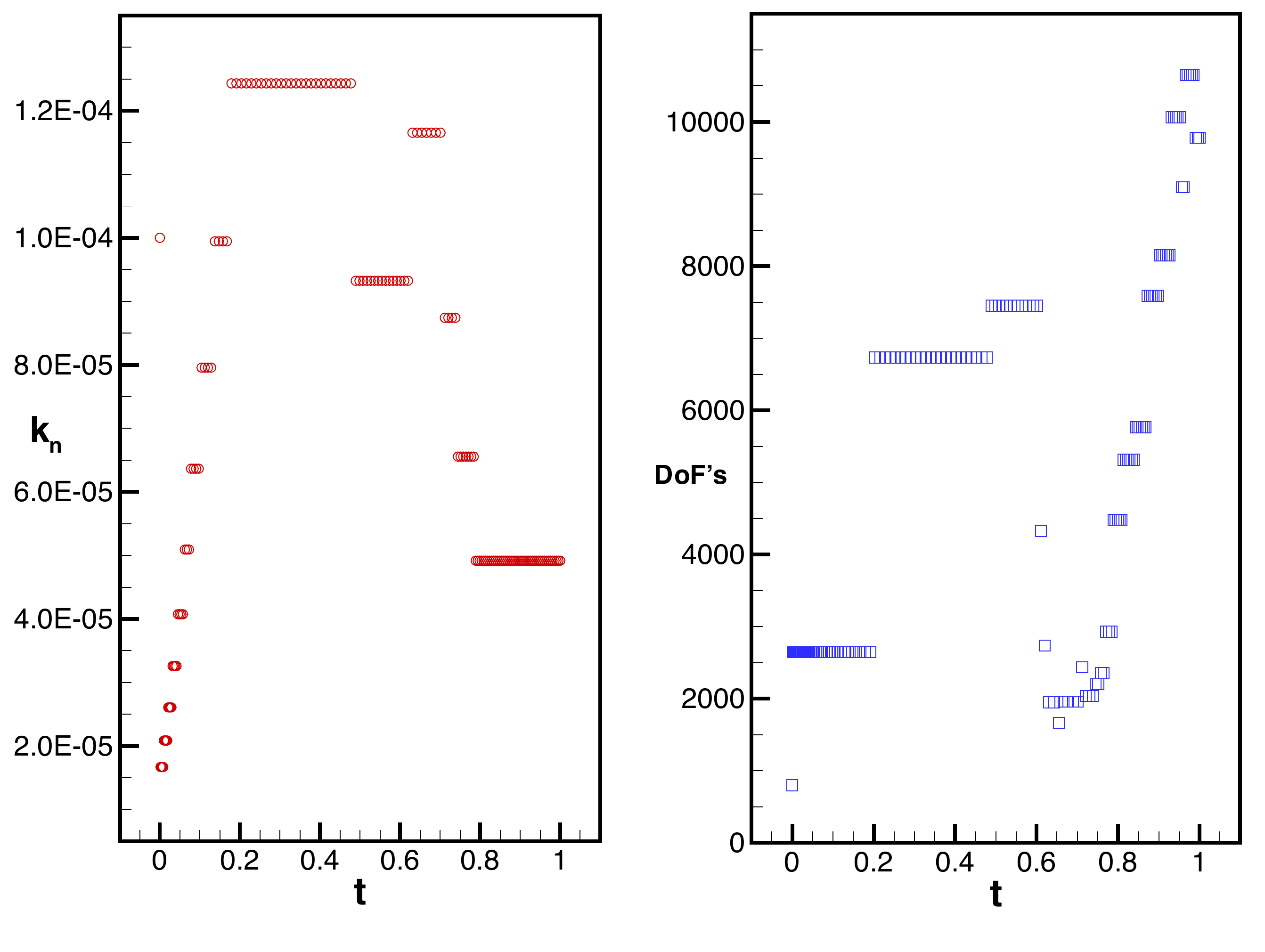} 
%
%\vspace{-0.1cm}
%
%
\caption{Evolution of estimators  in logarithmic scale (left) and variation of the time-steps $k_n$ and the DoF's versus $t$ (right) during  adaptivity for  $V(x,t)=\frac {x^2}2\cdot\frac {1}{t+0.05}$. \label{tmd2}}
}
\end{figure}
%------------------------------------------------------------------------------------

%For the second experiment, we solve in $[a,b]\times[0,T]=[-1,2]\times[0,1]$ and we take $V(x,t)=\displaystyle\frac {x^2}2\dot\frac 1{t+0.05}$ and $\ep=2.5\times 10^{-3}$. We take the same initial condition as in the previous experiment. We discrete by cubic B-splines and we apply again the adaptive algorithm. In Figure~\ref{ tmd2}
%------------------------------------------------------------------------------------

\begin{figure}[htb!]
%\centering
%
{\hspace{-1.5cm}
\includegraphics[scale=0.44]{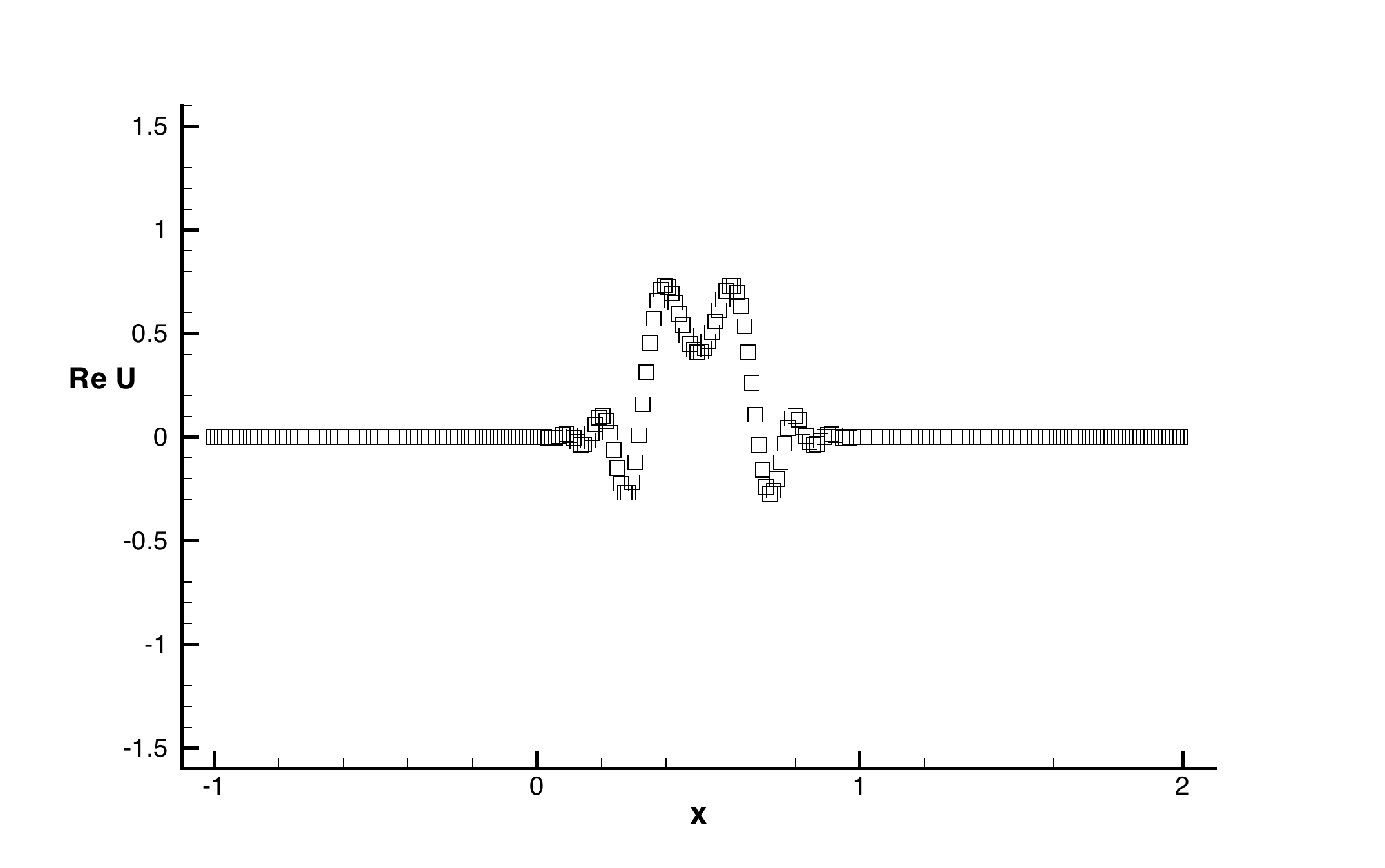}}
 {\hspace{-2.7cm}
\includegraphics[scale=0.44]{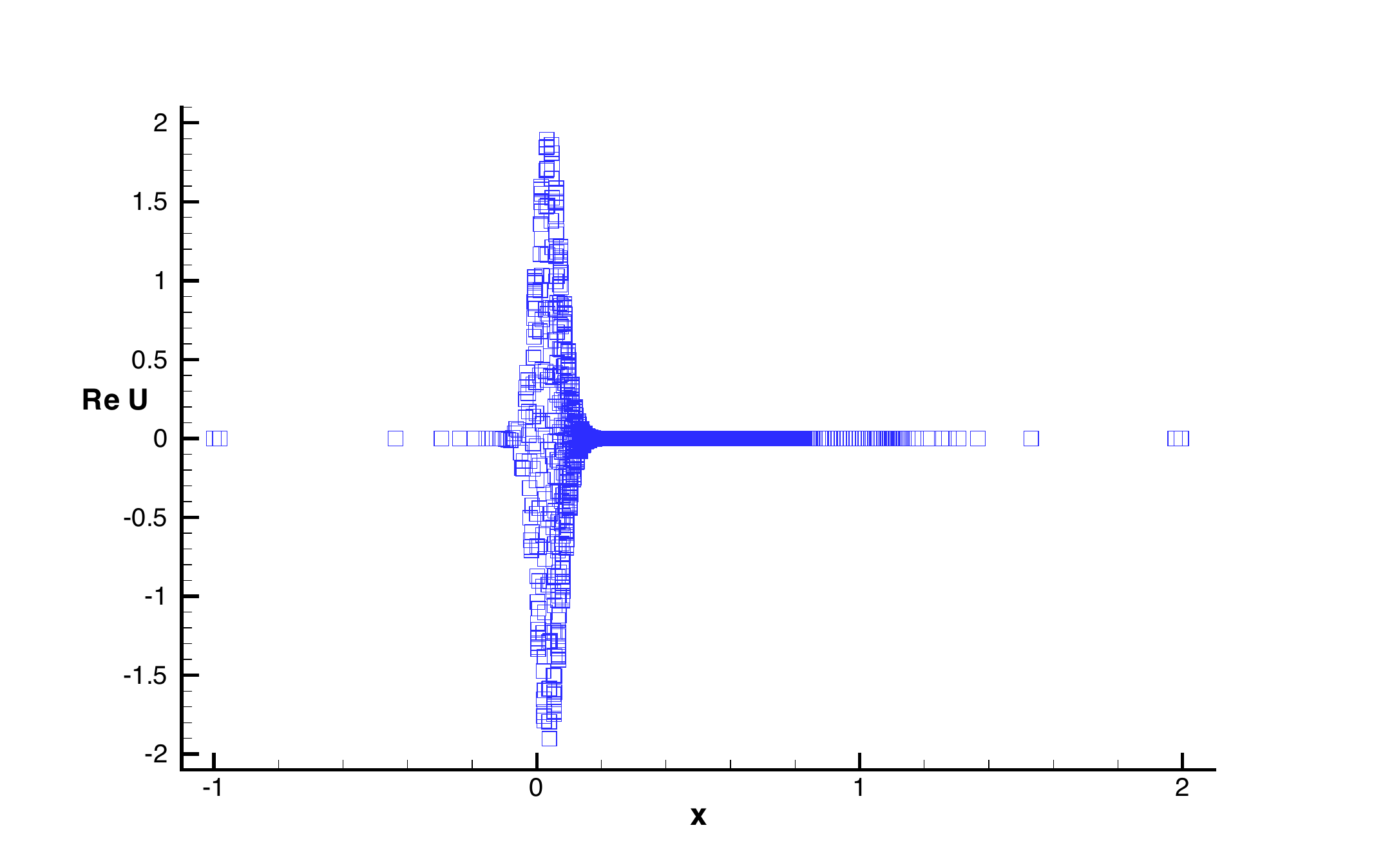}} \\

%
%\vspace{-0.1cm}
%
%
{\hspace{-1.5cm}
\includegraphics[scale=0.44]{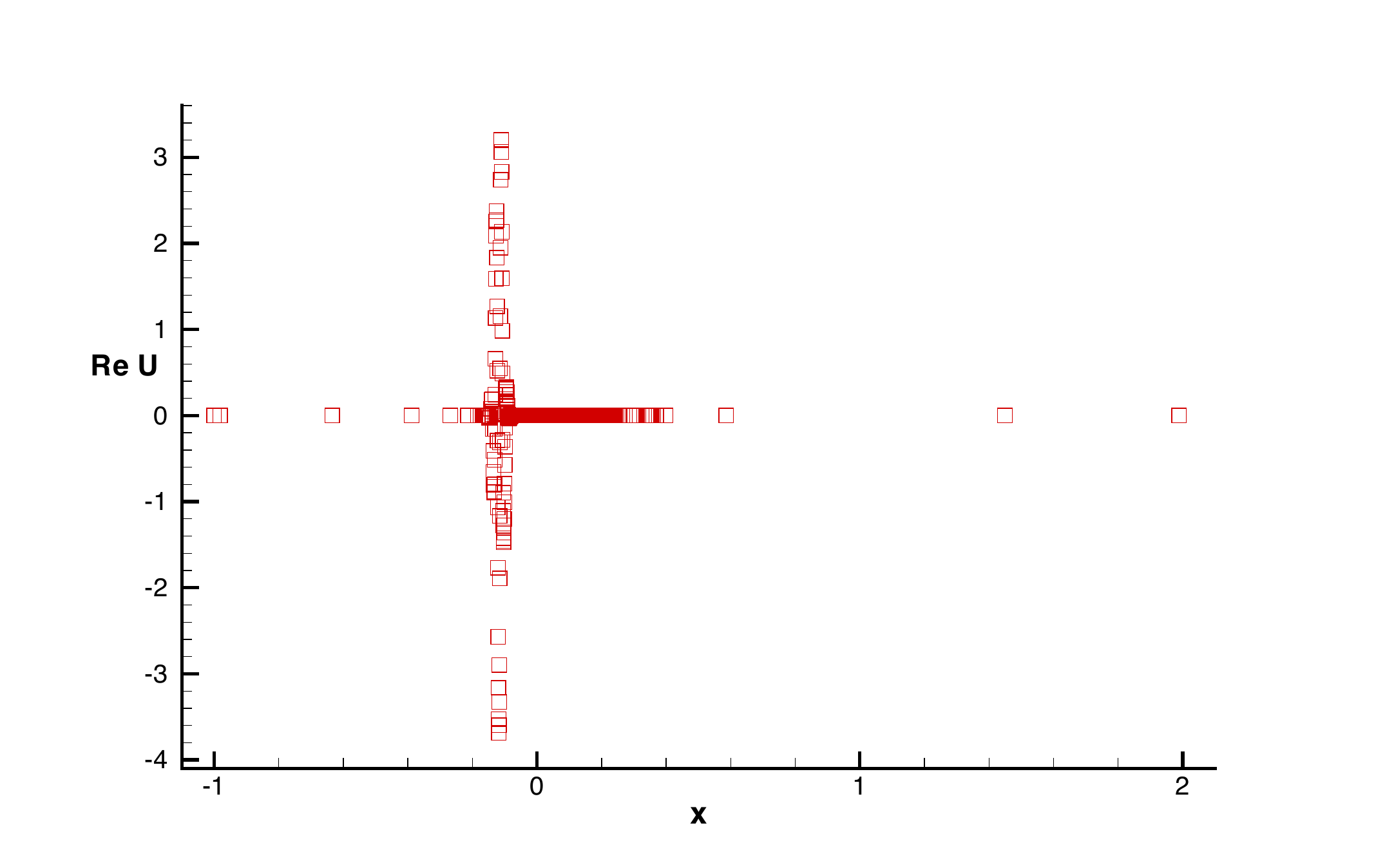}}
 {\hspace{-2.7cm}
\includegraphics[scale=0.44]{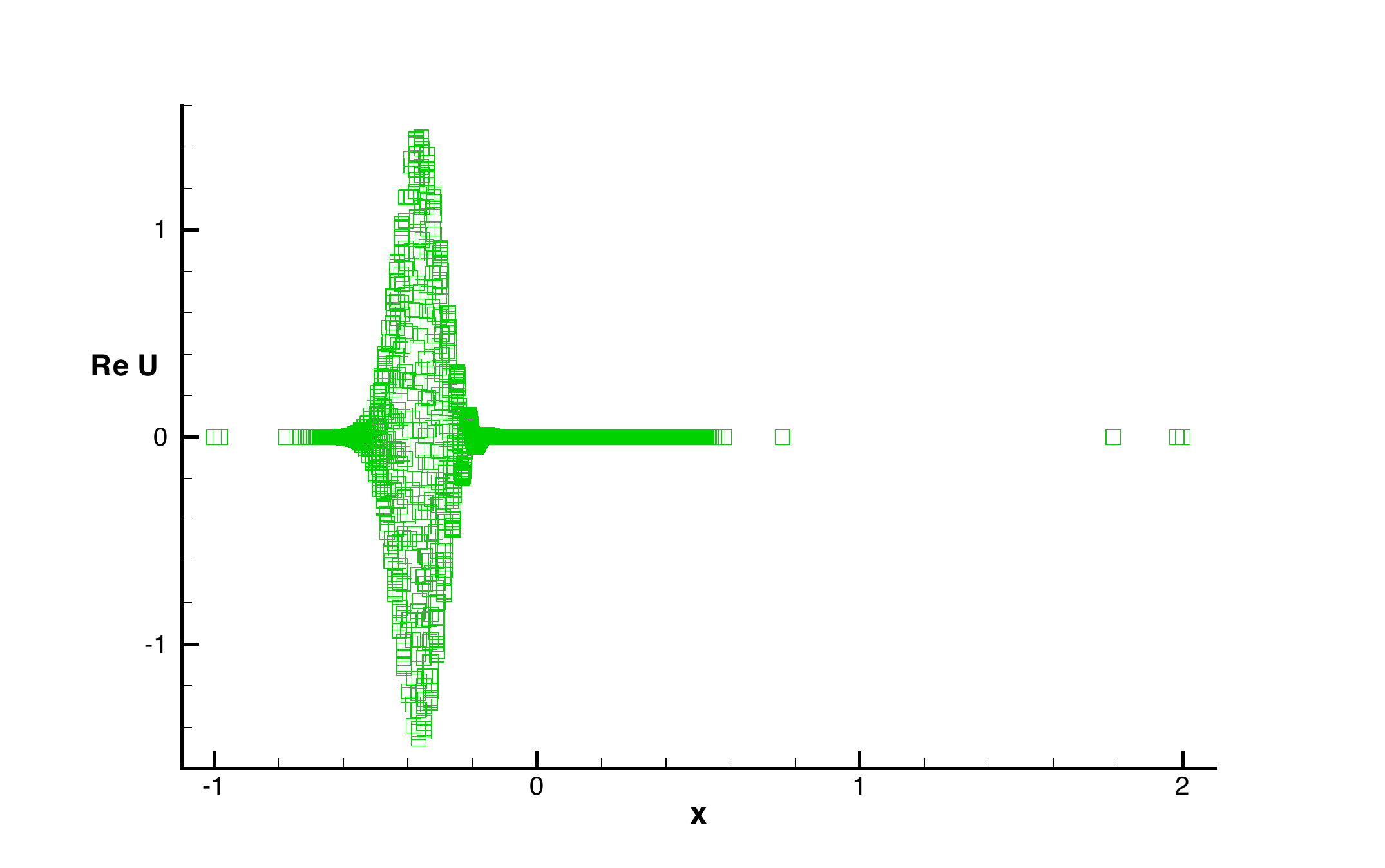}
\caption{Snapshots of the real part of the approximation and distribution of the grid points for the case $V(x,t)=\frac {x^2}2\cdot\frac {1}{t+0.05}$.\label{tmd3}}
}
\end{figure}

In Figure~\ref{tmd3}, we plot four snapshots: at the beginning, at the final time and in two intermediate times. From the plots we can also see the distribution of the grid points. At $t=0$, we start with uniform partition. For the remaining three snapshots, we observe that the points are dense close to rough changes of the approximation. Especially, in the third snapshot (left plot from below), almost all the points are concentrated close to the peak, while in areas where the solution doesn't change much, the grid is very sparse. This is an indicator of the robustness of  the  adaptive algorithm which  can provide reliable results with considerably less computational cost, compared to uniform grids.
%\begin{figure}[h!]\label{tmd4}
%%\centering
%%
%{\hspace{-1.5cm}
%\includegraphics[scale=0.35]{EXAMPLES/EX4/obs1.pdf}}
%%
% {%\hspace{-2.7cm}
%\includegraphics[scale=0.35]{EXAMPLES/EX4/obs2.pdf}
%
%%
%%\vspace{-0.1cm}
%%
%\caption{}
%}
%\end{figure}

%------------------------------------------------------------------------------------

%-----------------------------------------------------------------------------------
\subsection{Approximation of the observables}\label{observables}
%-----------------------------------------------------------------------------------
We focus next on the approximation of the observables \eqref{posden}, \eqref{curden}. In particular, we propose a modification of the adaptive algorithm and we verify numerically the advantages of the modified algorithm for the approximation of the observables, in terms of computational cost and accuracy.

For CNFE schemes, it is well known that the restrictive conditions between mesh sizes and the parameter $\ep$ needed for the efficient error control of the exact solution of \eqref{1dsemiclassical}--\eqref{incond} can be relaxed for the error control of the corresponding observables. More precisely, as it was proven in \cite{MPP, MPPS}, a sufficient and necessary condition for approximating well the observables is $\frac h\ep+\frac k\ep\to 0$. Moreover, the $L^\infty(L^2)$ approximation of the exact solution implies the $L^\infty$ approximation of observables' mean value; \cite{BJM}. In view of all these, we modify the adaptive algorithm as follows: We multiply all estimators but $\E_m^{\Ss,0}$ and $\E_m^{\TT,0}$ by $\ep$, so that the new estimators will converge provided that $\frac h\ep+\frac k\ep\to 0$,  cf., \eqref{econdh},\eqref{econdk}. Then, we apply the same algorithm, but with respect to these new estimators. 

We then perform various numerical experiments to verify whether this partially  heuristic idea can be advantageous to the approximation of the observables. More precisely, we consider the constant potential $V(x)\equiv 10$ and the WKB initial condition \eqref{incond} with $\sqrt{n_0}$ and $S_0$ as in \eqref{initial2}. We perform the experiments with adaptivity only in space.
%
%\vspace{-6cm}
For the first two tests, we take $[a,b]\times[0,T]=[-1,2]\times[0,0.54]$, $\lambda=5$ and $\ep=10^{-3}$ or $\ep=2.5\times10^{-4}$. Recall that the particular example, considered earlier in \cite{BJM}, is interesting because caustics are formed before the final time. For the case $\ep=10^{-3}$, we take $k=10^{-5}$ and discretize by quadratic  B-splines, whereas for $\ep=2.5\times 10^{-4}$, we take $k=3\times 10^{-6}$ and discretize by B-splines of degree $4$. In Figures~\ref{obs11}, \ref{obs21}, we plot the position density  using the adaptive algorithm (left plot) and uniform grid with the same degrees of freedom (right plot). The solid line corresponds to the semiclassical limit of the exact observable which is possible to compute for constant potentials.  The dotted lines correspond to the approximate observable. As we observe from these plots, the approximation using adaptivity is very good, while the one using uniform partition misses completely the angles and peaks. Similar comments can be made for the plots referring to the current density. These plots can be viewed in Figures~\ref{obs12} and \ref{obs22} for $\ep=10^{-3}$ and $2.5\times 10^{-4}$, respectively. In the plots concerning the 
approximations with space adaptivity, we also see the distribution of the grid points. It is remarkable that most of the points are concentrated close to the angles and peaks. On the contrary, very few points are placed around the endpoints, where the observables remain constant. The total number of degrees of freedom in adaptivity corresponds to $1458$ DoF's in each time-slot for $\ep=10^{-3}$ and to $3186$ for the case $\ep=2.5\times 10^{-4}$. The required degrees of freedom in each time-slot with uniform partition are more than $3000$ for $\ep=10^{-3}$ and more than $12 000$ for $\ep=2.5\times 10^{-4}$.
%
%The first two tests indicate that the smaller the value of $\ep$ using adaptivity is very advantageous. To make this indication stronger, we perform a final test in which $[a,b]\times[0,T]=[0,1]\times[0,0.1]$, $\lambda=30$ and $\ep=5\times 10^{-5}$. This is another example where caustics are formed. We use cubic B-splines and $k=5\times 10^{-7}.$
 \vspace{-0.05cm} 
\begin{figure}[htb!] 
%\centering
%
{\hspace{-1.5cm}
\includegraphics[scale=0.36]{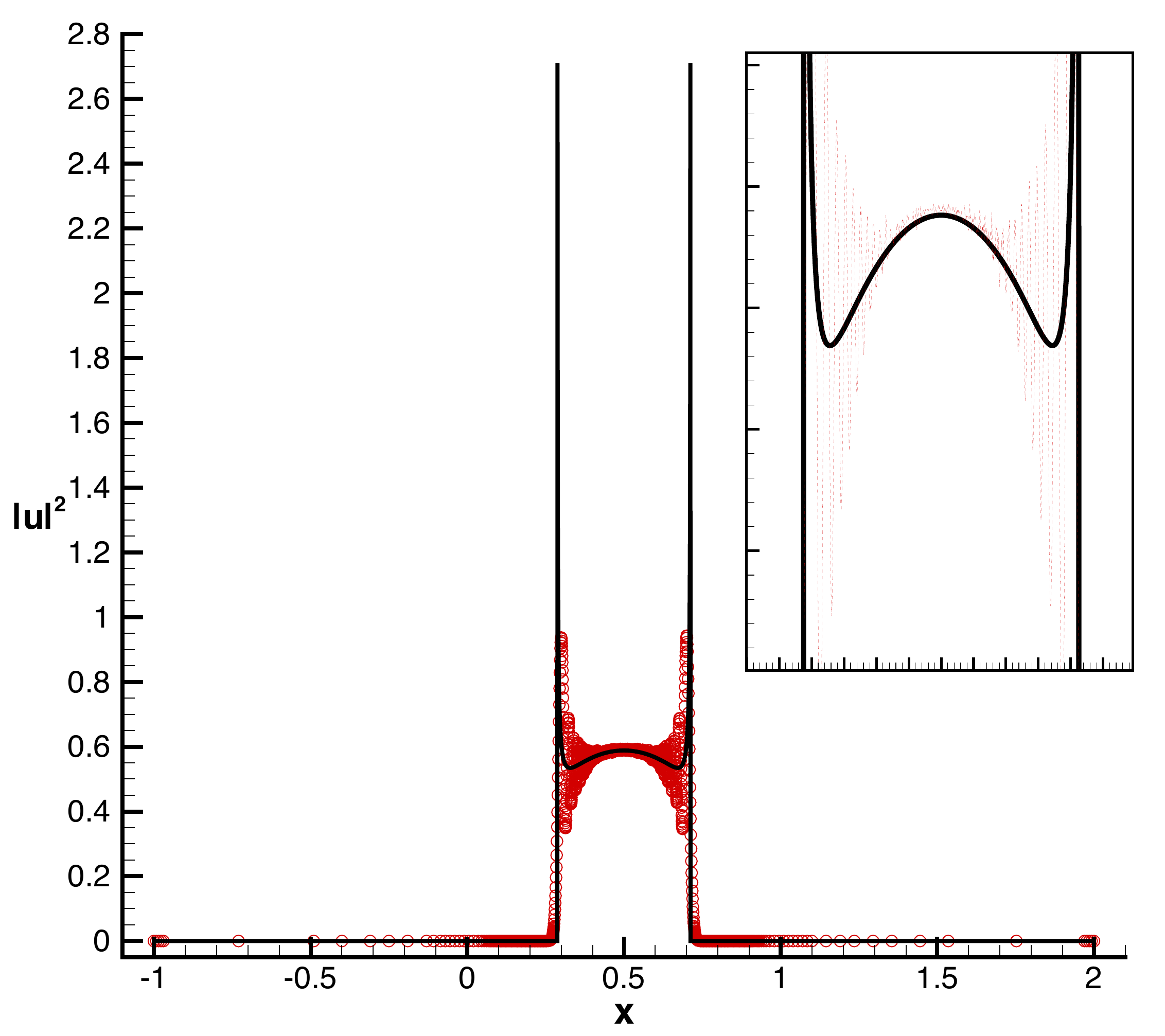}}
 {\hspace{0.5cm}
\includegraphics[scale=0.36]{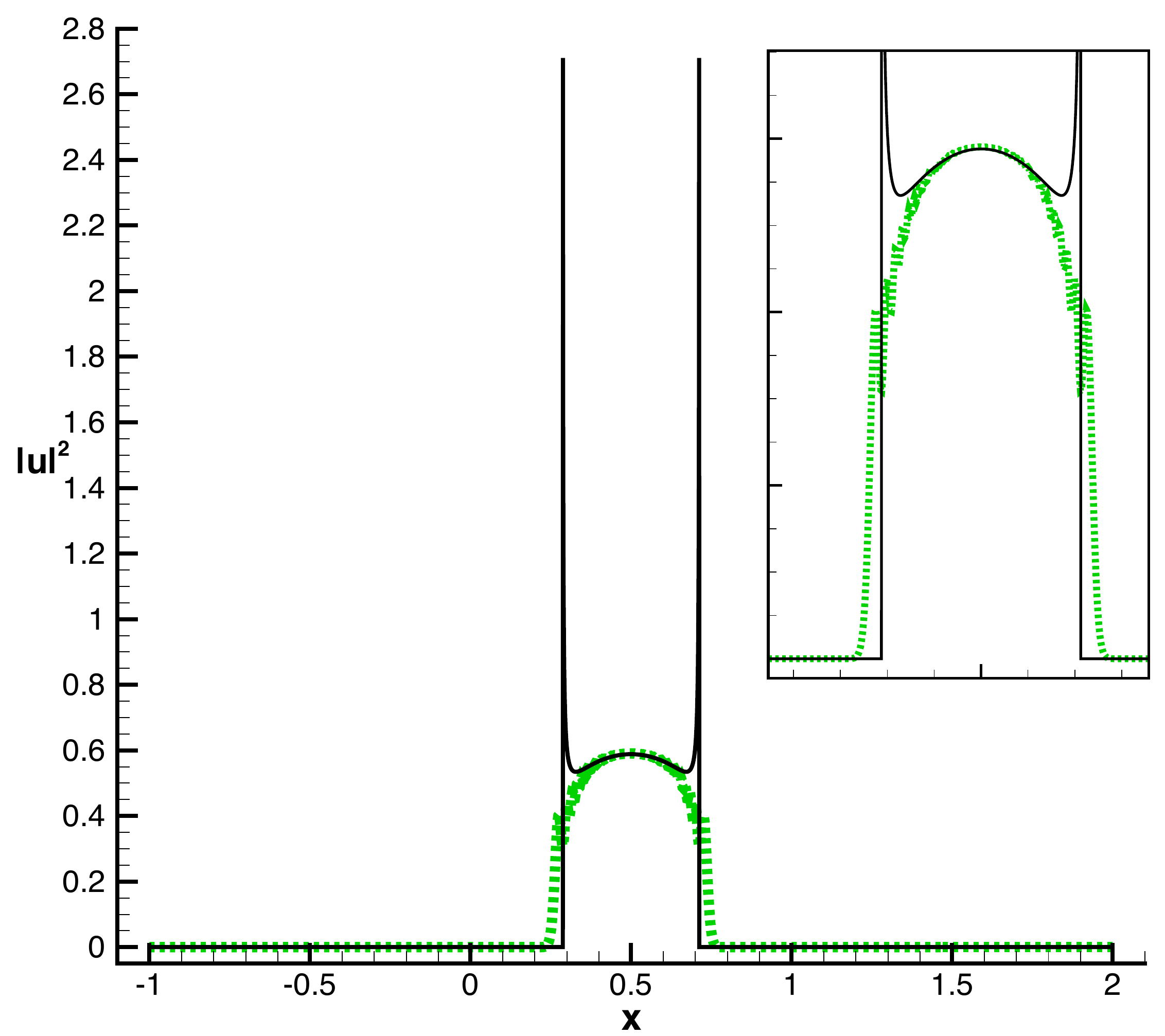} 
%
%\vspace{-0.1cm}
%
%
\caption{Position density  at the final time $T=0.54$ in case $\ep=10^{-3}.$ Solid line represents the semiclassical limit of the  exact observable, while dotted line represents the approximation using adaptivity (left) and uniform partition with the same DoF's (right).\label{obs11}}
}
\end{figure}
%-----------------------------------------------------------------------------------------------
\begin{figure}[htb!]
%\centering
%
{\hspace{-1.5cm}
\includegraphics[scale=0.36]{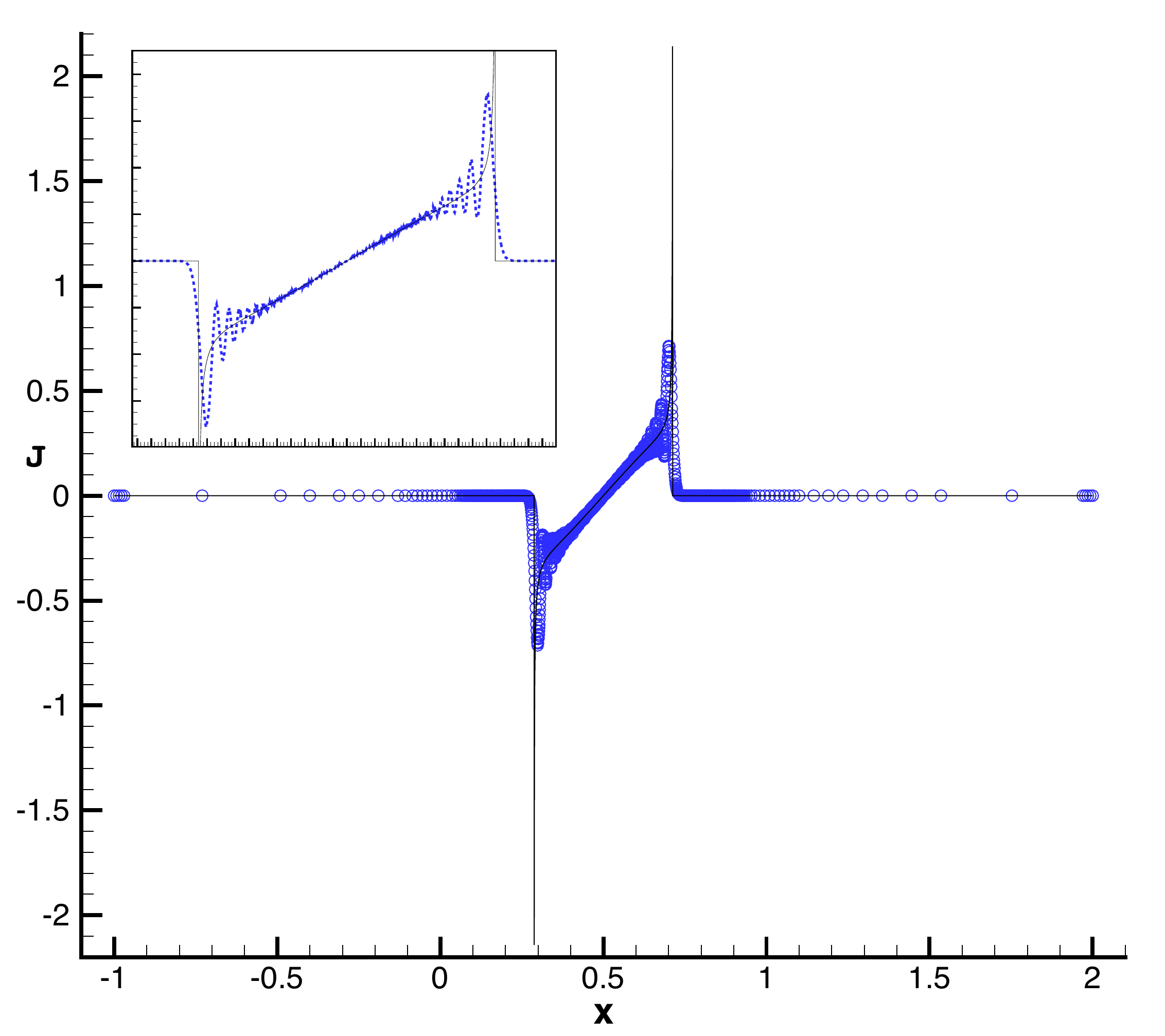}}
 {\hspace{0.5cm}
\includegraphics[scale=0.36]{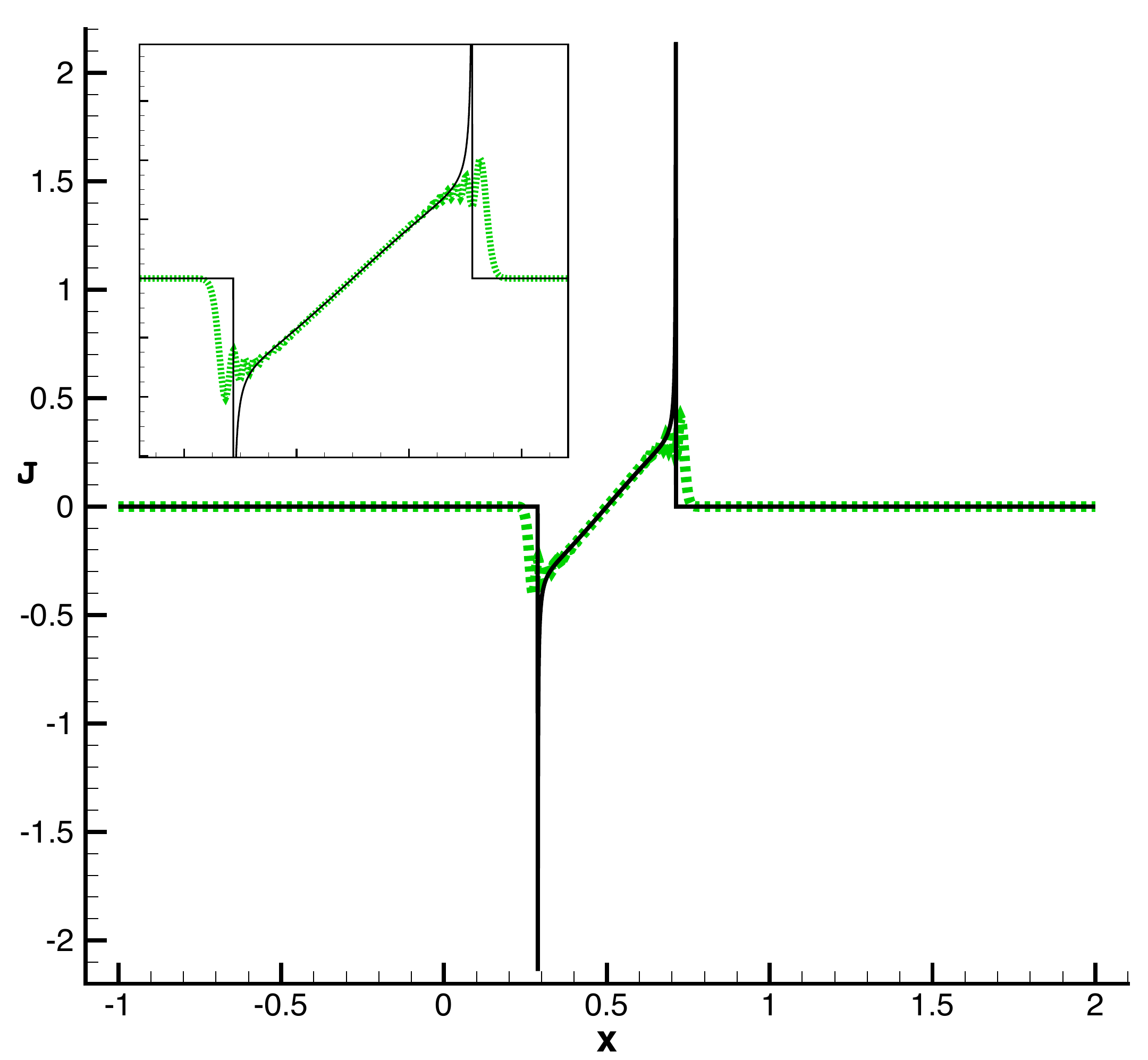} 
%
%\vspace{-0.1cm}
%
%
\caption{Current density at the final time $T=0.54$ in case $\ep=10^{-3}.$ Solid line represents the semiclassical limit of the exact observable, while dotted line represents the approximation using adaptivity (left) and uniform partition with the same DoF's (right).\label{obs12}}
}
\end{figure}
%---------------------------------------------------------------------------------------------------------------
%

\begin{figure}[htb!]
%\centering
%
{\hspace{-1.5cm}
\includegraphics[scale=0.36]{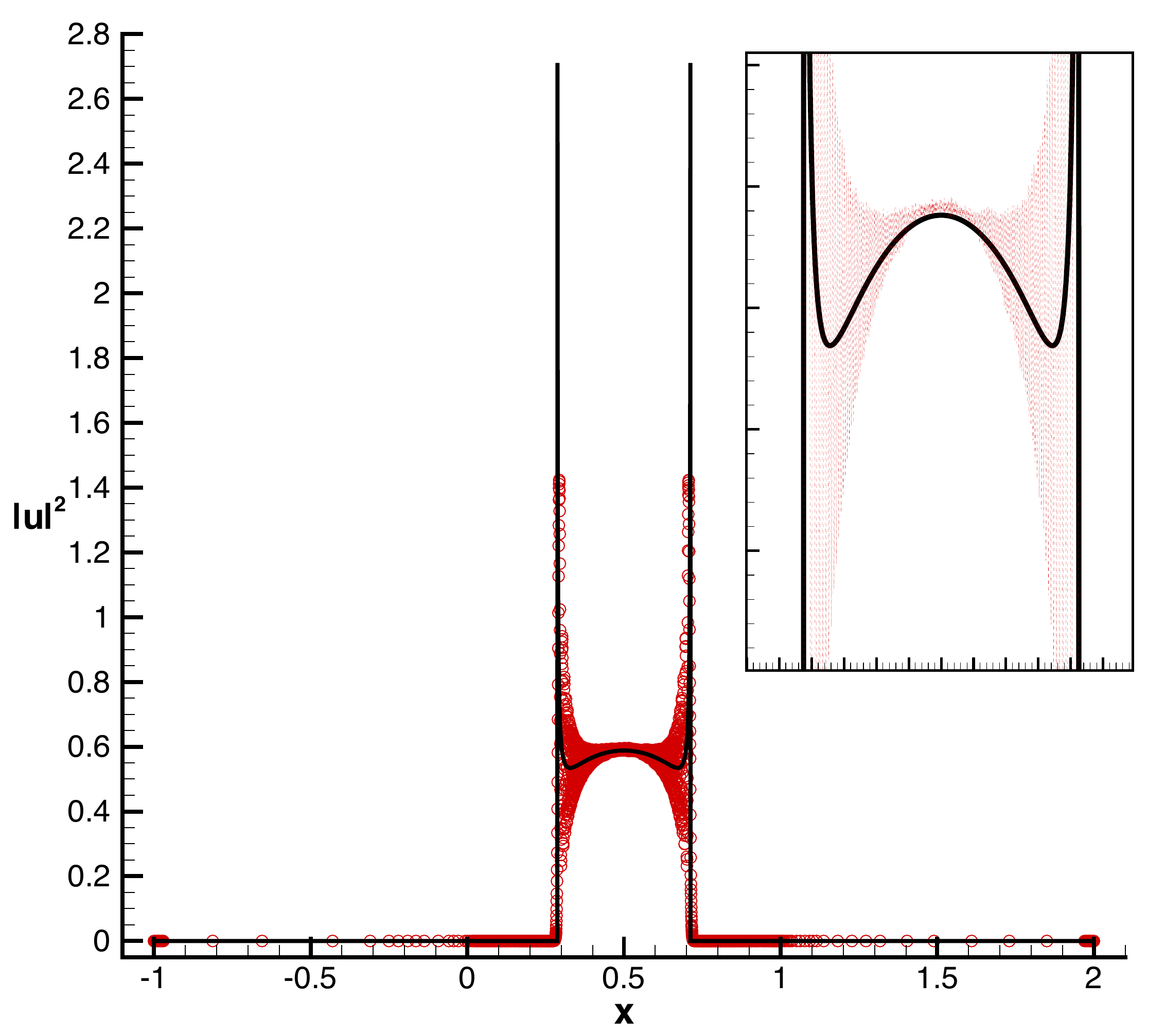}}
 {\hspace{0.5cm}
\includegraphics[scale=0.36]{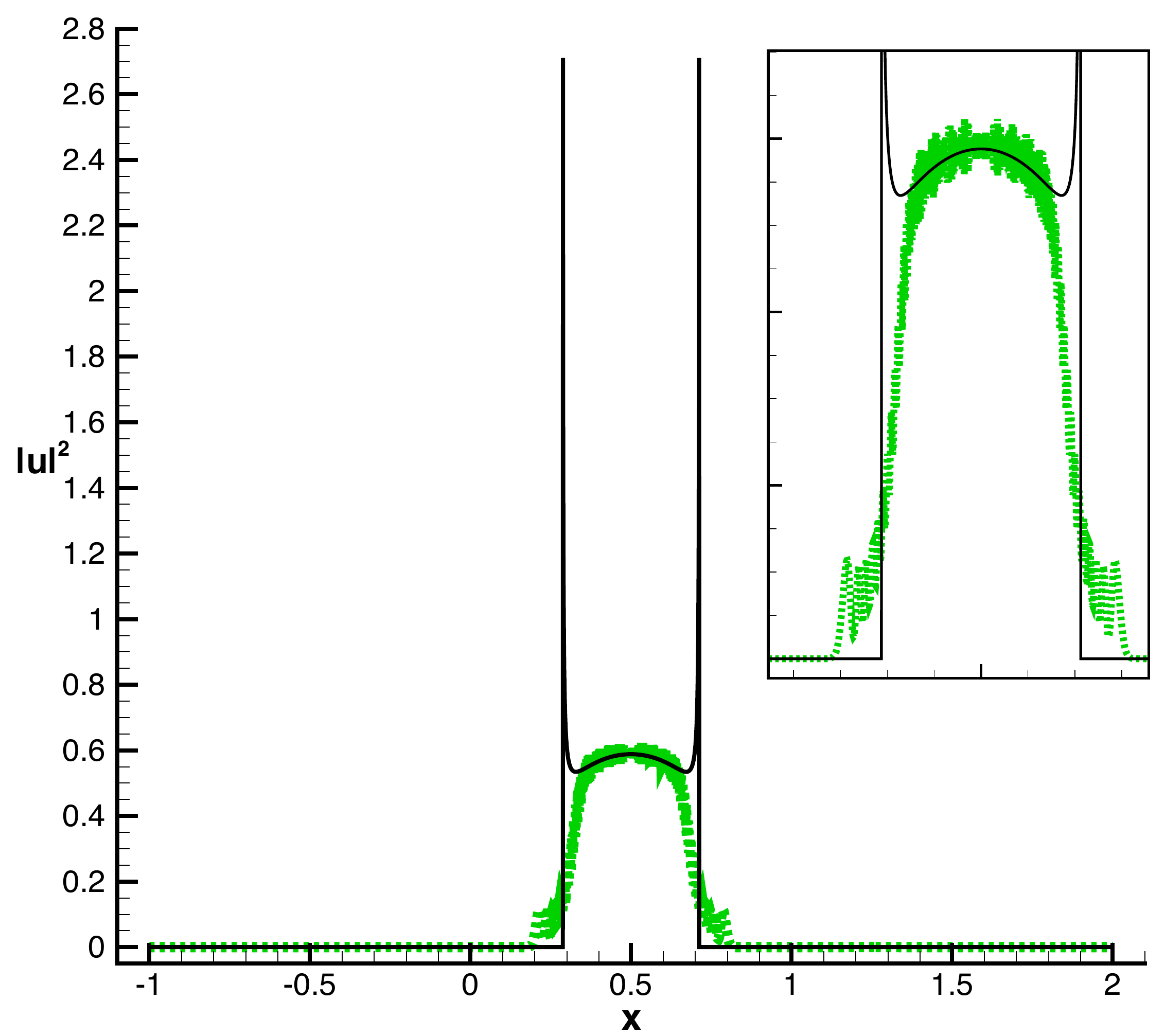} 
%
%\vspace{-0.1cm}
%
%
\caption{Position density at the final time $T=0.54$ in case $\ep=2.5\times 10^{-4}.$ Solid line represents the semiclassical limit of the exact observable, while dotted line represents the approximation using adaptivity (left) and uniform partition with the same DoF's (right).\label{obs21}}
}
\end{figure}
%-------------------------------------------------------------------------------------------------------------------
\begin{figure}[htb!]
%\centering
%
{\hspace{-1.5cm}
\includegraphics[scale=0.36]{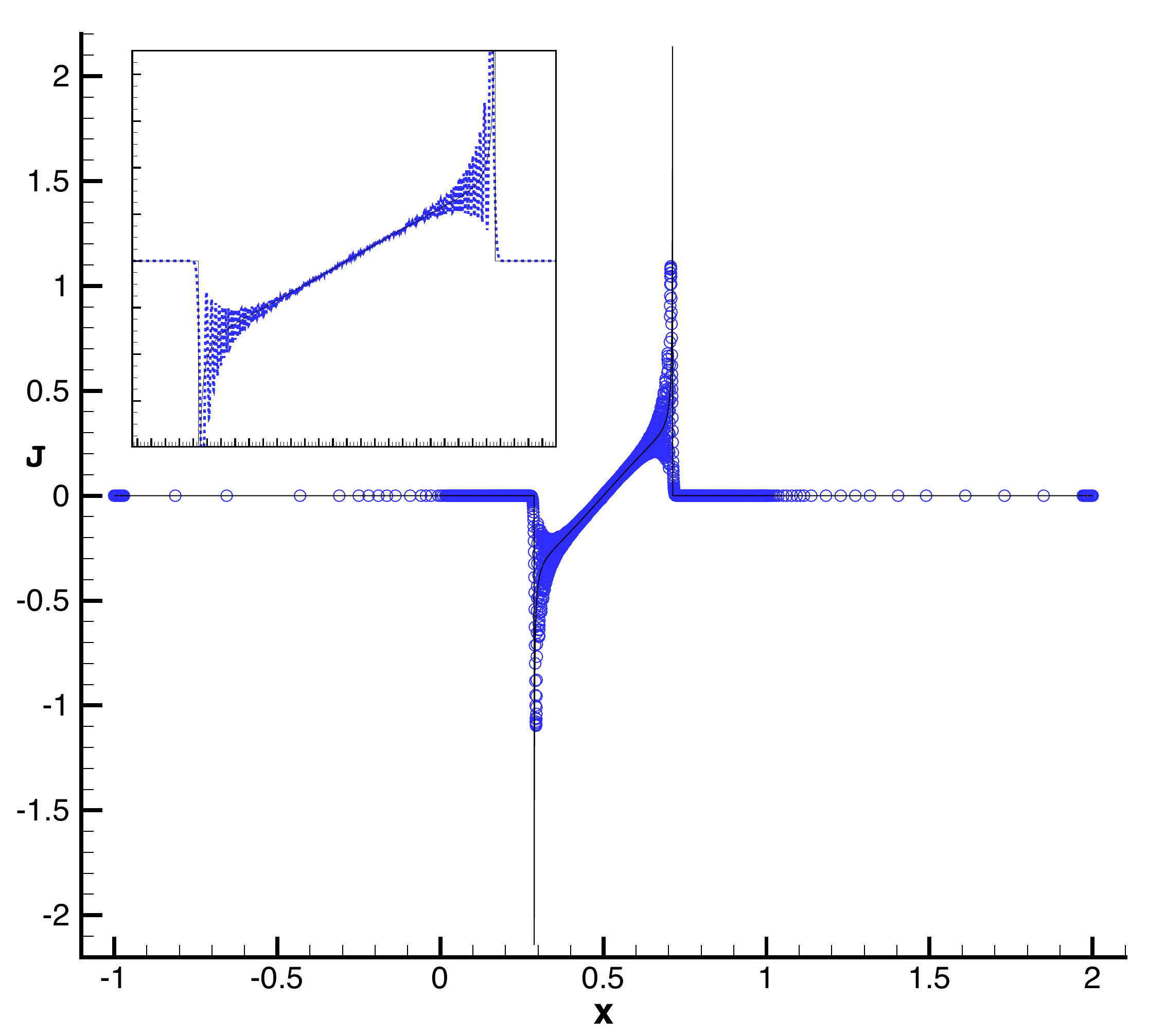}}
 {\hspace{0.5cm}
\includegraphics[scale=0.36]{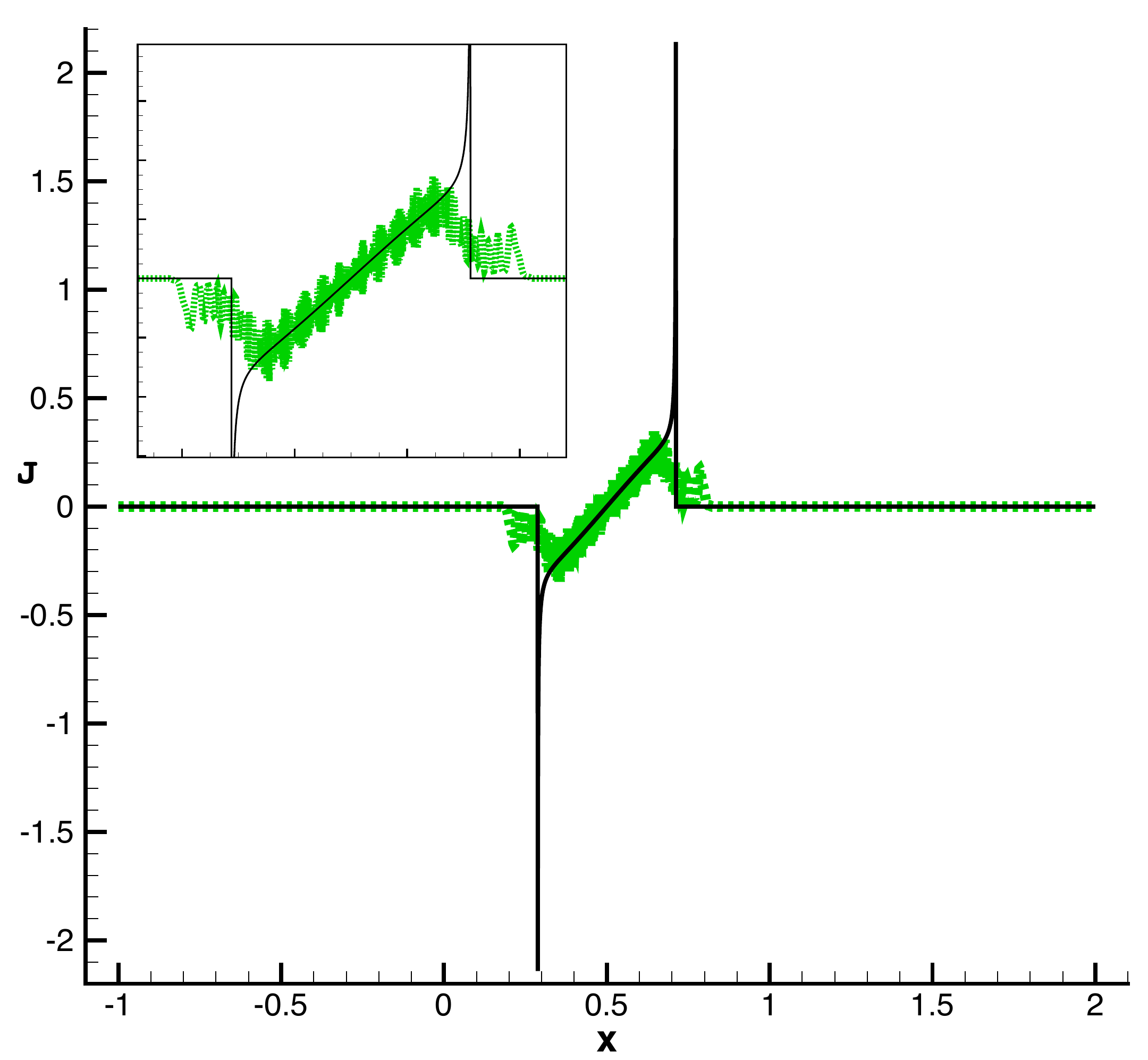} 
%
%\vspace{-0.1cm}
%
%
\caption{Current  density at the final time $T=0.54$ in case $\ep=2.5\times10^{-4}.$ Solid line represents the semiclassical limit of the exact observable, while dotted line represents the approximation using adaptivity (left) and uniform partition with the same DoF's (right).\label{obs22}}
}
\end{figure}
\begin{figure}[htb!]\label{obs41}
%\centering
%
{\hspace{-1.5cm}
\includegraphics[scale=0.36]{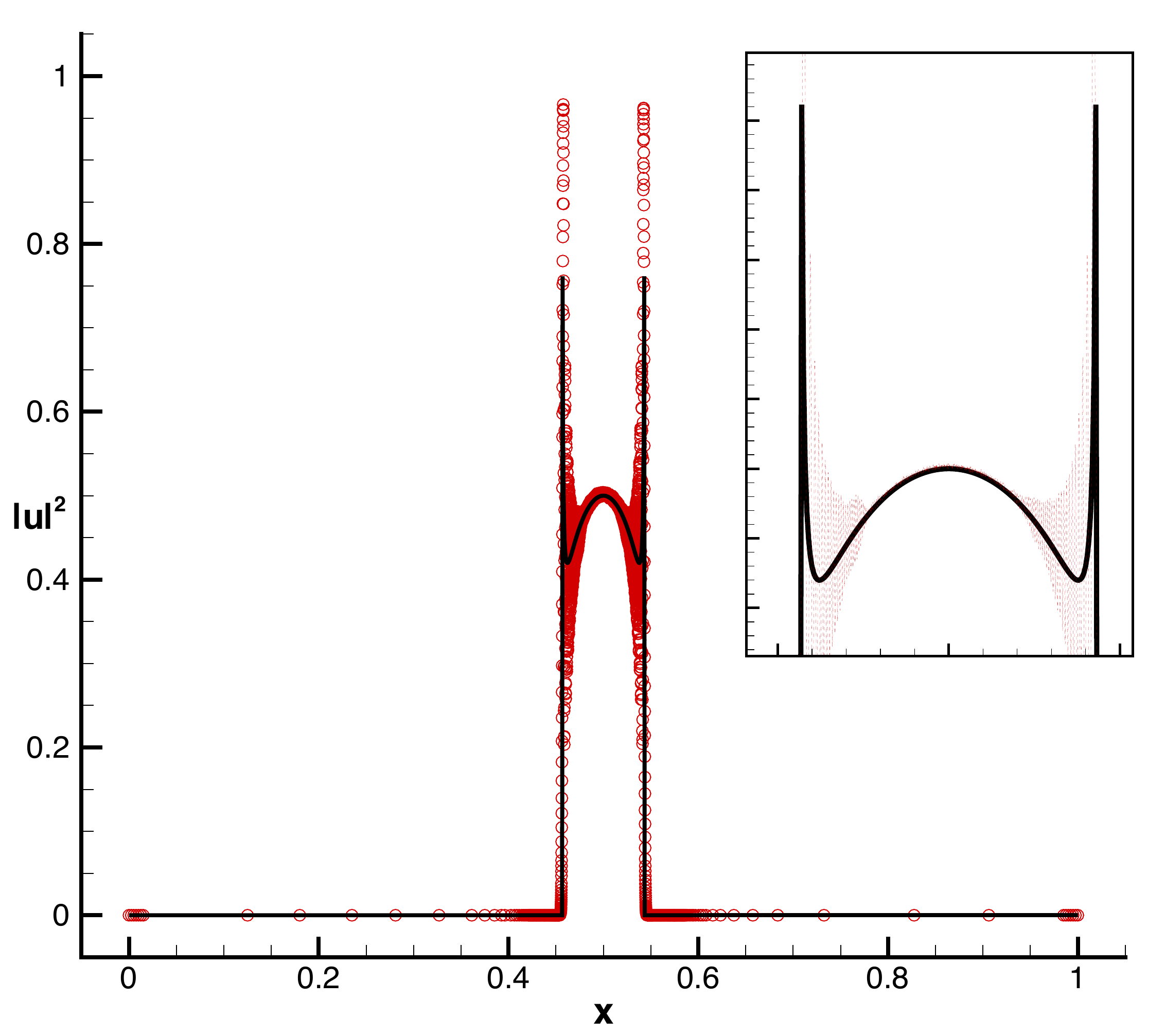}}
 {\hspace{0.5cm}
\includegraphics[scale=0.36]{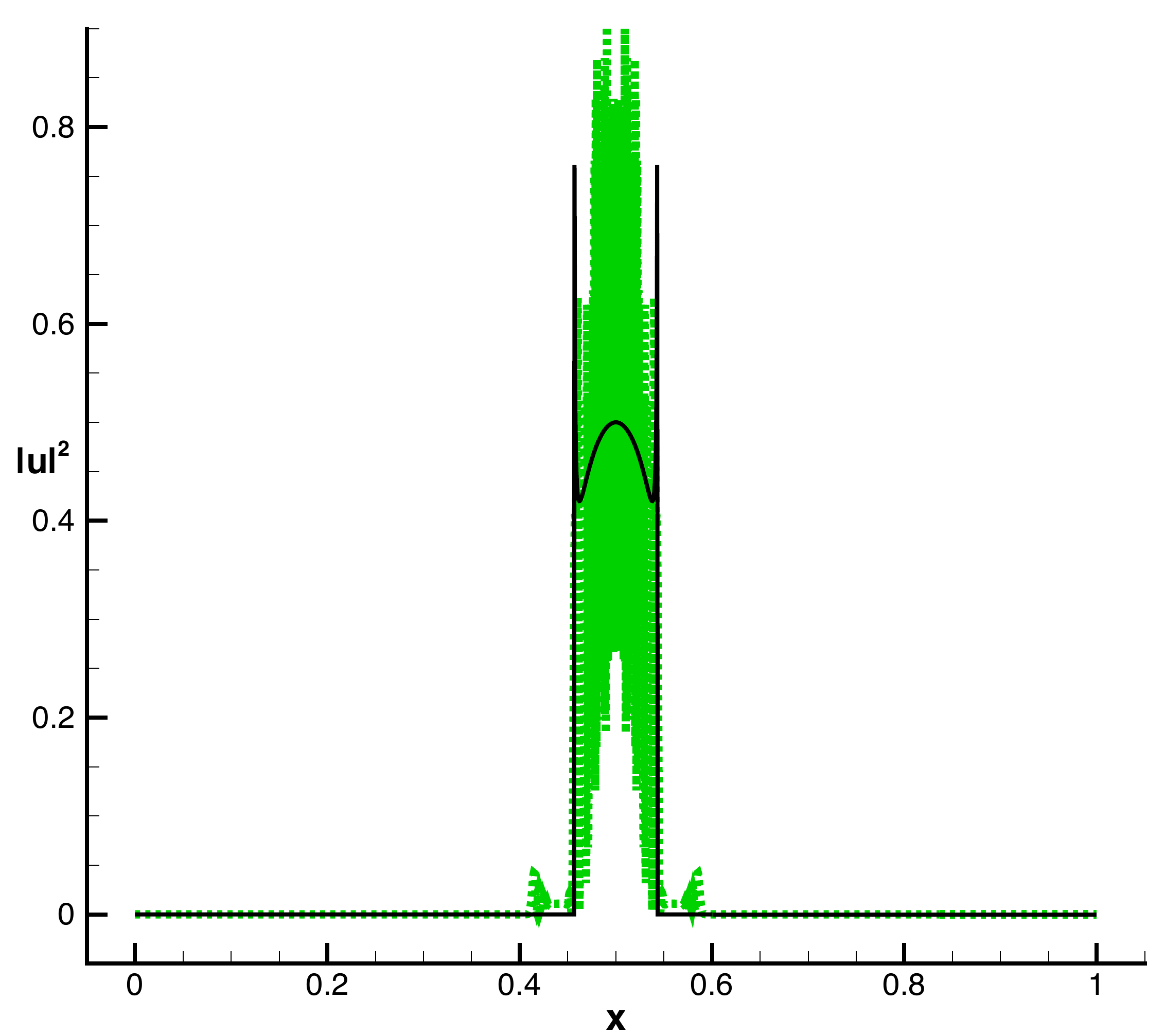} 
%
%\vspace{-0.1cm}
%
%
\caption{Position density at the final time $T=0.1$ in case $\ep=5\times10^{-5}.$ Solid line represents the semiclassical limit of the exact observable, while dot line represents the approximation using adaptivity (left) and uniform partition with the same DoF's (right).\label{obs41}}
}
\end{figure}
%-------------------------------------------------------------------------------------------------------------------
\begin{figure}[htb!]
%\centering
%
{\hspace{-1.5cm}
\includegraphics[scale=0.36]{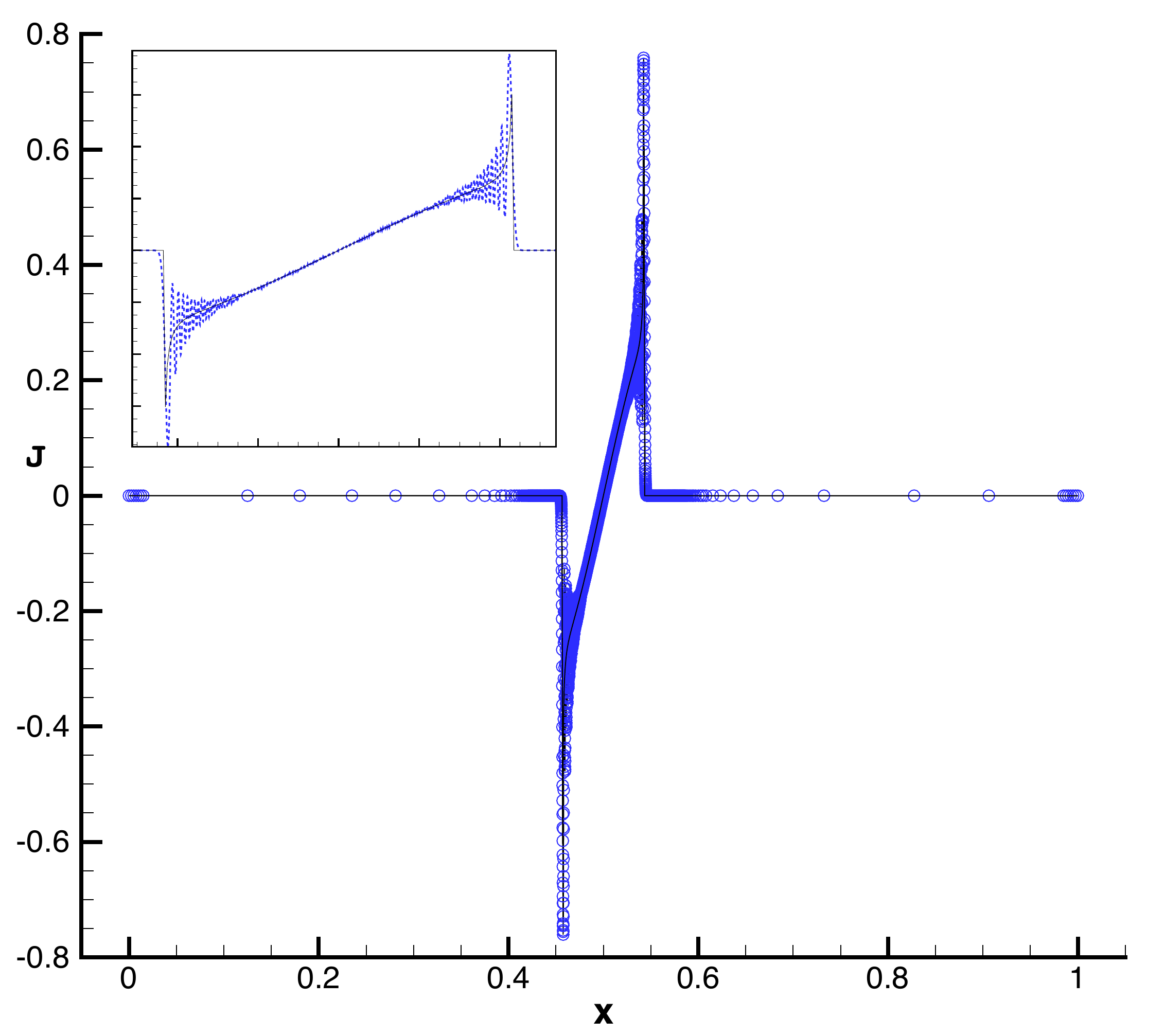}}
 {\hspace{0.5cm}
\includegraphics[scale=0.36]{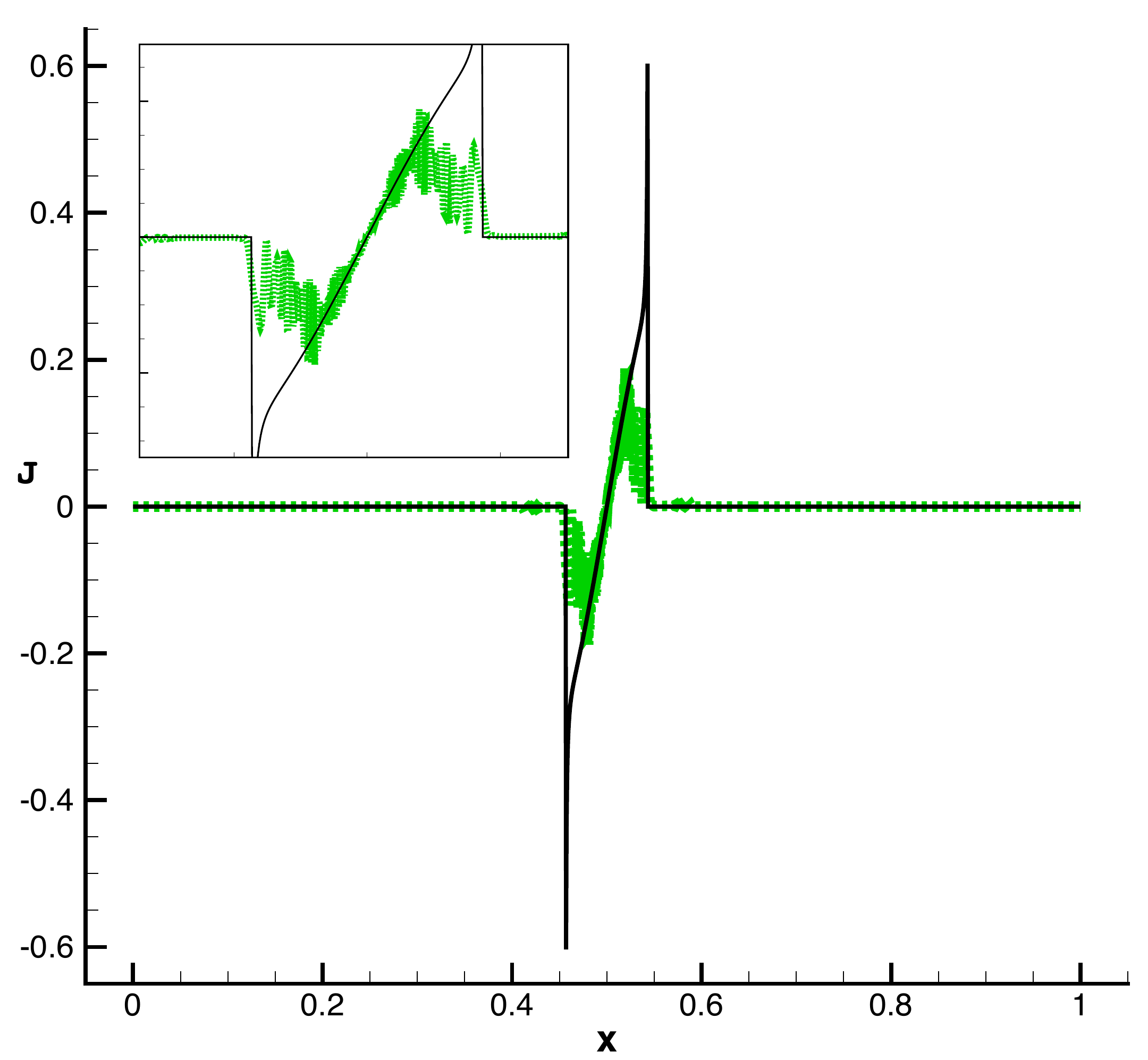} 
%
%\vspace{-0.1cm}
%
%
\caption{Position density at the final time $T=0.1$ in case $\ep=5\times10^{-5}.$ Solid line represents the semiclassical limit of the exact observable, while dotted line represents the approximation using adaptivity (left) and uniform partition with the same DoF's (right).\label{obs42}}
}
\end{figure}
%----------------------------------------------------------------------------------------------------------------------

For the first two tests, we take $[a,b]\times[0,T]=[-1,2]\times[0,0.54]$, $\lambda=5$ and $\ep=10^{-3}$ or $\ep=2.5\times10^{-4}$. Recall that the particular example, considered earlier in \cite{BJM}, is interesting because caustics are formed before the final time. For the case $\ep=10^{-3}$, we take $k=10^{-5}$ and discretize by quadratic  B-splines, whereas for $\ep=2.5\times 10^{-4}$, we take $k=3\times 10^{-6}$ and discretize by B-splines of degree $4$. In Figures~\ref{obs11}, \ref{obs21}, we plot the position density  using the adaptive algorithm (left plot) and uniform grid with the same degrees of freedom (right plot). The solid line corresponds to the semiclassical limit of the exact observable which is possible to compute for constant potentials.  The dotted lines correspond to the approximate observable. As we observe from these plots, the approximation using adaptivity is very good, while the one using uniform partition misses completely the angles and peaks. Similar comments can be made for the plots referring to the current density. These plots can be viewed in Figures~\ref{obs12} and \ref{obs22} for $\ep=10^{-3}$ and $2.5\times 10^{-4}$, respectively. In the plots concerning the 
approximations with space adaptivity, we also see the distribution of the grid points. It is remarkable that most of the points are concentrated close to the angles and peaks. On the contrary, very few points are placed around the endpoints, where the observables remain constant. The total number of degrees of freedom in adaptivity corresponds to $1458$ DoF's in each time-slot for $\ep=10^{-3}$ and to $3186$ for the case $\ep=2.5\times 10^{-4}$. The required degrees of freedom in each time-slot with uniform partition are more than $3000$ for $\ep=10^{-3}$ and more than $12 000$ for $\ep=2.5\times 10^{-4}$.
The first two tests indicate that the smaller the value of $\ep$ using adaptivity is very advantageous. To make this indication stronger, we perform a final test in which $[a,b]\times[0,T]=[0,1]\times[0,0.1]$, $\lambda=30$ and $\ep=5\times 10^{-5}$. This is another example where caustics are formed. We use cubic B-splines and $k=5\times 10^{-7}.$

%----------------------------------------------------------------------------------------------------------------------
In Figures~\ref{obs41},\ref{obs42}, we plot on the left the approximation with space adaptivity and on the right the corresponding with uniform partition and the same degrees of freedom. The result obtained using  the uniform partition is very poor. The approximate solution  misses  the angles and peaks, and, in fact, fails to approximate  the actual observables.  On the other hand, those obtained by adaptivity, appear to be very  good approximations. The number of total degrees of freedom in adaptivity corresponds to $3670$ DoF's in each time-slot, while the required DoF's in each time slot with uniform partition is more than $20000.$
%%---------------------------------------------------------------------------------------------------------------

This final set of experiments, indicates that the a posteriori error estimators can appropriately be used together with adaptive strategies not only for the efficient error control of the wave function $u$, but for the observables as well. The tests suggest that the computational cost is drastically reduced and the adaptive procedure gives encouraging results for small values of  the Planck constant $\ep$. However, no rigorous analysis has been provided and further numerical experiments including more general potentials need to be performed  in order to draw safe conclusions. This very interesting problem requires further investigation and will be the subject of a forthcoming work.  

\section*{Acknowledgments}
%------------------------------------------------------------------------------
Some of the ideas  of the theoretical part of the paper are taken from second author's  Ph.D Thesis,
\cite{PhDKyza}. I.K. is grateful to her Ph.D advisor, Prof.\ Charalambos Makridakis for suggesting the problem and for his
academic guidance and support. The authors thank Prof.\
Georgios Akrivis for many helpful remarks.
%
%
%-------------------------------------------------------------------------------

%
\end{document}